\documentclass [12pt] {article}
\usepackage{amsfonts}
\usepackage{amsmath}
\usepackage{amsthm}
\usepackage{amssymb}
\let\amslrcorner\lrcorner
\usepackage{hyperref}
\usepackage{bm,upgreek}
\usepackage[square,numbers]{natbib}
\usepackage{fullpage}
\usepackage{authblk}
\usepackage{mathabx}
\usepackage{tikz-cd}
\usepackage{relsize}
\usepackage{ulem}

\DeclareFontFamily{U}{MnSymbolC}{}
\DeclareSymbolFont{MnSyC}{U}{MnSymbolC}{m}{n}
\DeclareFontShape{U}{MnSymbolC}{m}{n}{
    <-6>  MnSymbolC5
   <6-7>  MnSymbolC6
   <7-8>  MnSymbolC7
   <8-9>  MnSymbolC8
   <9-10> MnSymbolC9
  <10-12> MnSymbolC10
  <12->   MnSymbolC12}{}
\DeclareMathSymbol{\intprod}{\mathbin}{MnSyC}{'270}
\let\lrcorner\amslrcorner

\newtheorem{thm}{Theorem}[section]
\newtheorem{cor}[thm]{Corollary}
\newtheorem{lem}[thm]{Lemma}
\newtheorem{prop}[thm]{Proposition}
\newtheorem{defn}[thm]{Definition}
\newtheorem{rem}[thm]{Remark}
\newtheorem{example}[thm]{Example}

\numberwithin{equation}{section}

\newcommand{\ol}[1]{\overline{#1}}

\newcommand{\lj}{\mbox{{\emph{\j}}}}

\newcommand{\cG}{\mathcal{G}}
\newcommand{\cV}{\mathcal{V}}
\newcommand{\cW}{\mathcal{W}}
\newcommand{\cH}{\mathcal{H}}
\newcommand{\cD}{\mathcal{D}}

\newcommand{\om}{\omega}
\newcommand{\cQ}{\mathcal{Q}}
\newcommand{\cF}{\mathcal{F}}
\newcommand{\II}{I\!I}
\newcommand{\tfII}{\mathring{I\!I}}

\renewcommand{\hat}[1]{\widehat{#1}}

\newcommand{\bc}{\boldsymbol{c}}
\newcommand{\cT}{\mathcal{T}}
\newcommand{\ce}{\mathcal{E}}
\newcommand{\cE}{\mathcal{E}}
\newcommand{\cK}{\mathcal{K}}
\newcommand{\cZ}{\mathcal{Z}}
\newcommand{\cN}{\mathcal{N}}
\newcommand{\si}{\sigma}
\newcommand{\Rho}{P}
\newcommand{\V} {\mathbb{V}}
\newcommand{\TD}{\mathbb{D}}

\newcommand{\wh} [1] {\widehat{#1}}

\newcommand{\mbf} [1]{\mathbf{#1}}

\newcommand{\mc} [1]{\mathcal{#1}}

\newcommand{\GL} {\mathrm{GL}}

\newcommand{\SO} {\mathrm{SO}}

\newcommand{\End} {\mathrm{End}}

\newcommand{\cc} {\bm{c}}
\newcommand{\bg}{\boldsymbol{g}}
\newcommand{\confmet}{\bg}
\newcommand{\be}{\boldsymbol{\epsilon}}
\newcommand{\tvol}{\mathlarger{\mathlarger{\epsilon}}}

\newcommand{\R} {\mathbb{R}}

\newcommand{\LL} {\mathbb{L}}

\newcommand{\bX}{\mathbb{X}}

\newcommand{\codim}{\mathrm{codim}}

\newcommand{\gtM}{\ensuremath{{\gamma \hookrightarrow M}}}
\newcommand{\gtG}{\ensuremath{{\gamma\hookrightarrow \Sigma}}}
\newcommand{\GtM}{\ensuremath{{\Sigma \hookrightarrow M}}}

\newcommand{\SN}{\star \hspace{-1pt} N}

\newcommand{\KN}{\mathbin{\bigcirc\mspace{-16mu}\wedge\mspace{4mu}}}

\usepackage{tensor}
\usepackage{easybmat}
\usepackage{color}
\usepackage{etex}
\usepackage[all]{xy}

\newcommand{\lpl}{
  \mbox{$
  \begin{picture}(12.7,8)(-.5,-1)
  \put(2,0.2){$+$}
  \put(6.2,2.8){\oval(8,8)[l]}
  \end{picture}$}}

\def\sideremark#1{\ifvmode\leavevmode\fi\vadjust{\vbox to0pt{\vss
 \hbox to 0pt{\hskip\hsize\hskip1em
 \vbox{\hsize2cm\tiny\raggedright\pretolerance10000
 \noindent #1\hfill}\hss}\vbox to8pt{\vfil}\vss}}}
\newcommand{\edz}[1]{\sideremark{#1}}

\newcommand{\Addresses}{{
  \bigskip
  \footnotesize
	
	S.~N.~C., \,\textsc{Department of Mathematics, Oklahoma State University,
     Mathematics, Statistics and Computer Science 401, Stillwater, O 74075, USA}\par\nopagebreak
  \textit{E-mail address},   	S.~N.~C.  \, : \texttt{sean.curry@okstate.edu}

  \medskip

  A.~R.~G.,\, \textsc{Department of Mathematics, The University of Auckland,
    Private Bag 92019, Auckland 1142, New Zealand}\par\nopagebreak
  \textit{E-mail address},   A.~R.~G. \, : \texttt{r.gover@auckland.ac.nz}

  \medskip

  D.~S.,\,  \textsc{Department of Mathematics, The University of Auckland,
    Private Bag 92019, Auckland 1142, New Zealand}\par\nopagebreak
  \textit{E-mail address},   D.~S. \, : \texttt{daniel.snell@auckland.ac.nz}

}}

\makeatletter
\def\thanks#1{\protected@xdef\@thanks{\@thanks
        \protect\footnotetext{#1}}}
\makeatother

\begin{document}
\normalem

\title{
Conformal submanifolds, distinguished submanifolds, and  integrability}


\author{Sean N.\ Curry, A.~Rod Gover, Daniel Snell\thanks{S.C. is supported in part by the Simons Foundation Grant MPS-TSM-00002876. A.R.G.\ and D.S. gratefully acknowledge support from the Royal
  Society of New Zealand via Marsden Grants 19-UOA-008 and 24-UOA-005.}}

\date{}

\maketitle

\vspace{-1.5em}

\begin{abstract}  
For conformal geometries of Riemannian signature, we provide a
comprehensive and explicit treatment of the core local theory for
embedded submanifolds of arbitrary dimension. This is based in the
conformal tractor calculus and includes a conformally invariant
Gauss formula leading to conformal versions of the Gauss, Codazzi,
and Ricci equations. It provides the tools for proliferating
submanifold conformal invariants, as well as for extending to conformally singular Riemannian manifolds the notions of mean
curvature and of minimal and CMC
submanifolds.

A notion of distinguished submanifold is defined by asking the tractor
second fundamental form to vanish.  We show that for the case of
curves this exactly characterises conformal geodesics, also called
conformal circles, while for hypersurfaces it is the totally umbilic
condition. For other codimensions, this unifying notion interpolates
between these extremes, and we prove that in all dimensions this
coincides with the submanifold being weakly conformally circular,
meaning that ambient conformal circles remain in the submanifold. We
prove that submanifolds are conformally circular, meaning submanifold
conformal circles coincide with ambient conformal circles, if and only
also a second conformal invariant also vanishes. We give a number of
examples to show that both situations occur commonly in familiar
conformal structures.


We then provide a very general theory and construction of quantities
that are necessarily conserved along distinguished submanifolds. This
first integral theory thus vastly generalises the results available
for conformal circles in \cite{GST}. We prove that any normal solution
to an equation from the class of first BGG equations yields such a
conserved quantity, and we show that it is easy to provide explicit
formulae for these.  

Finally we prove that the property of being distinguished is also
captured by a type of moving incidence relation and apply this to show that the distinguished submanifold condition is forced for zero loci associated with solutions of certain natural geometric partial differential equations.  
  \end{abstract}

\vspace{1pt}
\noindent {\small \emph{2010 Mathematics Subject Classification.}
  Primary: 53C18, 53B25, 53A55; Secondary: 37K10, 53B15, 53B21

\vspace{2pt}
\noindent {\small \emph{Key words and phrases.} Differential geometry,
  conformal geometry, submanifolds, conformal geodesics, first integrals, BGG operators, symmetries, invariants.}


\tableofcontents

\pagebreak

\section{Introduction}

Submanifolds are one of the fundamental objects of study in any class of differential geometric structures. They play a crucial role in geometric analysis and a variety of other areas including several complex variables, and the study of functional analytic inequalities. 
Their local theory is essential for the study of many global questions in differential geometry. In the special case of Riemannian
geometry, submanifold theory is a classical area and the basic local
theory is well understood and is founded in the celebrated equations of
Gauss, Codazzi, and Ricci, see, e.g.,
\cite{Kobayashi-Nomizu,ONeillRiemannianGeometry}.

Conformal manifolds $(M,\cc)$ are structures where a smooth $n$
manifold $M$ is equipped not with a metric, but rather with just an
equivalence class of smooth metrics $\cc$, where $g,\hat{g}\in \cc$
means that $\hat{g}=e^{2\om}g$ for some smooth function $\om$. There
is currently a growing interest in the study of conformal
submanifolds, and conformally distinguished curves, including the
relationships between these objects
\cite{CGK,Chang-McKeown,Friedrich2017,GrahamReichert,HerfrayFine,Juhl,JO,LuebbeKroon2013,MinucciKroon2021}.
Some of these developments have been inspired and driven by the links
to physics \cite{Friedrich1995,GrahamWitten}, the discovery that there
are higher dimensional analogues of the classical Willmore energy and
invariant \cite{Guven,GWWillmore,YuriThesis,GrRV,GWRenormVol,GrahamReichert,YZ},
and the development of a holographic approach to submanifolds (see
\cite{GWBoundaryCalc,GWRenormVol,GWConfHypYamabe}) that is an analogue
of Fefferman and Graham's holographic approach to intrinsic conformal
geometry as in \cite{FG-ast,FG-book}.

In the conformal setting there is no distinguished connection on the
tangent bundle, so even the local theory of submanifolds provides a
challenge. Toward resolving this, a logical step is to use the
conformal Cartan/tractor connection of
\cite{cartan1923espaces,ThomasOnConformal,BEG,CS-book}, and for the special
case of {\em hypersurfaces}, meaning embedded submanifolds of
codimension one, an effective approach was initiated in
\cite{BEG}. With a view to various applications, this hypersurface
theory was extended in the works
\cite{BransonGover01,GrantMSc,StaffordMSc,GRiemSig,YuriThesis} and
this approach has proved to be central in a number of further
extensions and applications
\cite{AGW,BGW,GWBoundaryCalc,GWWillmore,GWRenormVol,GWcalc,GWConfHypYamabe,Juhl}. Rather
separately from the general consideration of submanifolds, the {\em
  distinguished curves} in conformal manifolds known as {\em conformal
  circles} or {\em conformal geodesics} have been studied classically
(see, e.g., \cite{Ferrand,Muto1939,Ogiue1967,Yano1938,Yano1957}) and
from various modern perspectives recently
\cite{FriedSchmidt,Bailey1990a,BEG,Tod,SilhanVojtech,GST,dunajski2019conformal,GregZala,Kry,CDT}.

In the first part of our work here we develop a comprehensive basic
local theory for conformal submanifolds of all proper
codimensions. This is based in the conformal tractor calculus, and by
construction is conformally invariant. It builds on the mentioned
approach to hypersurfaces from
\cite{BEG,BransonGover01,GrantMSc,YuriThesis,Curry-G-conformal} and
its extension into higher codimension by the first and third named
authors in \cite{CurryThesis} and \cite{SnellPhD}. There are also
links to the somewhat more abstract theory developed in the preprint
\cite{CalderbankBurstall} (cf.\ Remark \ref{BCremark} below and also
the discussion in \cite{CurryThesis}). The result is a theory and
collection of explicit calculational tools that treats curves and
higher dimensional submanifolds, embedded in conformal manifolds, by a
single uniform approach. These tools may be used to directly
proliferate submanifold invariants, including for curves. Thus they
provide a basic machinery that may solve the conformal submanifold
analogue of Fefferman's parabolic invariant theory programme
\cite{FeffermanPIT,FG-ast,BaiEG}, and we touch on this in Section
\ref{sinvts}.  Since the first draft of this article, an application
of these tools to a simple direct construction of higher Willmore
energies has been provided in \cite{AG}. We
also use these tools to effectively capture, via tractors, the mean
curvature, and with it the notions of minimal, constant mean
curvature, and parallel mean curvature submanifolds. This leads
immediately to a generalisation of these notions that is applicable
for the study of conformally singular geometries (such as
Poincar\'e-Einstein geometries, and more generally conformally compact
structures). See Section \ref{sec_minimal_scales}.

The point of view that we develop here leads us to introduce the
notion of a {\em distinguished submanifold} of arbitrary codimension.
By definition this means that an object called the tractor second
fundamental form vanishes (a list of equivalent conditions is given in
Theorem \ref{key1}).  In the case of curves this coincides precisely
with the unparametrised conformal circle equation, see Theorem
\ref{conf_circ_2ff_zero}, whereas for hypersurfaces it recovers the
usual condition of total umbilicity \cite{BEG}. Thus the notion
interpolates between these. Interestingly, in all codimensions greater
than one the condition is stronger than total umbilicity;
distinguished submanifolds must be umbilic, but there are umbilic
submanifolds that are not distinguished (though, when the geometry of
the ambient manifold is special, it often turns out that umbilic
submanifolds are forced to be distinguished; see Section
\ref{sec-examples} for a list of detailed examples).  In the
conformally flat setting the notions of umbilic and distinguished
submanifolds coincide, and part of our motivation in introducing the
notion of distinguished submanifold is that (unlike the notion of
umbilic submanifold in general codimension) it allows us to
simultaneously generalise certain key results from the codimension one
case of umbilic submanifolds and the dimension one case of conformal
circles. These applications form the second main objective of this
work.

It has long been known that Killing tensors and Killing-Yano tensors
may be used to provide first integrals for geodesics
\cite{Carter1968,Floyd1973,Penrose1973,KillingConstantsMotion,WalkerPenrose1970},
and this is used for a host of applications
\cite{Andersson-Blue,Carter1968b,Chandrasekhar-Dirac,Frolov2017,GHKW,KKMbook,IntegrabilityKillingEqn}.
Conformal circles are governed by a higher order equation than
geodesics, so an analogous theory has been lacking aside from certain
specific examples \cite{Tod}.  However, in \cite{GST} this was solved
and a very general theory of first integrals was developed by
understanding a characterisation of conformal circles as a parallel
condition on a fundamental tractor 3-form that one can associate to
any non-null curve.  Using this, it was established that essentially
any normal solution of a class of equations known as first BGG
equations (see \cite{CGH-Duke,CSS2000}, or Section \ref{BGGsec}, for
the meaning of these terms) can provide, or contribute to, such
conserved quantities; in fact in many cases more general solutions
produce first integrals. See \cite{GregZala} for some applications of this perspective. The conformal Killing equations on tensors
and the conformal Killing-Yano equations are all first BGG equations.
But in fact the class of first BGG equations is vastly wider than this
suggests.  In Section \ref{proofkey1} we show that, just as for
curves, higher dimensional distinguished submanifolds can be
characterised by a parallel condition on a tractor form. Then, as an
application, we obtain a theory of first integrals for distinguished
submanifolds of all codimensions in a form that includes the case of
conformal circles as special case.  See Theorem
\ref{thm_submanifold_gst} and Corollary \ref{fi-cor}.

In another direction, an important question for submanifolds of
dimension $2$ or greater is characterising the conformal analogue of
the notion of being totally geodesic; that is, to capture some sense
of ``total conformal circularity''. This was touched on in
\cite{BEG} for hypersurfaces, and treated for submanifolds in general
by Belgun \cite{BelgunFlorin2015}. We show in Section \ref{sec4} that
our tools give efficient  new proofs these results, and show that Belgun's
results have an elegant interpretation in this tractor picture. See
Theorem \ref{weakly} and Theorem
\ref{thm_totally_conformally_circ_param} below. We also introduce a new notion called \emph{conformal circularity},  and show that property holds in many situations where the strongest notion conformal circularity fails.

Finally, in Section \ref{corbit-sec} we show that distinguished
submanifolds can also be characterised by a very simple moving
incidence relation, see Theorem \ref{thm_submanifold_gst-main}, or its
paraphrasing in Theorem \ref{thm_submanifold_gst}.  As an application
we prove that the zero locus of suitable overdetermined PDE solutions
are necessarily distinguished submanifolds; see Theorem
\ref{thm_submanifold_zero_locus}. This shows how distinguished
submanifolds fit into the {\em curved orbit} theory of
\cite{CGH-jlms,CGH-Duke} and, along with Proposition \ref{progress},
is a first step toward understanding how to generalise the holography
approach of
\cite{AGW,BGW,GWWillmore,GWRenormVol,GWbdy,GWcalc,GWConfHypYamabe} to
higher codimensions.

\subsection{Main results and a technical overview}

Now we give the approach and results with more technical detail. We use lower case abstract indices $a,b,c,\ldots$ for ambient tensors and indices $i,j,k, \ldots$ from later in the alphabet for submanifold tensors; we also use the corresponding capital letters as abstract indices for ambient and submanifold tractors, respectively. The dimension of our ambient manifold $M$ will always be denoted $n$, and the submanifold dimension will be denoted $m$. In the theorems below $n\geq 3$, but the results continue to hold when $n=2$ provided $M$ is endowed with a M\"obius structure (in the sense of \cite{CalderbankMobiusStructures}). The first step in developing a calculus for hypersurfaces in a Riemannian manifold is the observation that any oriented hypersurface is equipped with a
canonical unit conormal field $n_a$.  Similarly it was established in
\cite{BEG} that each hypersurface $\Sigma$ in a Riemannian signature
conformal manifold determines canonically a basic conformally
invariant tractor field $N_A$ that plays an analogous role at a
tractor level.
For example its failure to be parallel along
the hypersurface is captured by a tractor second fundamental form $\mathbb{L}$.
In particular one obtains, in a simple explicit way, conformal
analogues of the Gauss-Codazzi-Ricci theory, see \cite{Curry-G-conformal} and
references therein.
Moreover the normal tractor has a remarkable link
to other objects in, for example, Poincar\'e-Einstein manifolds and
related structures, that has led to some deep results (e.g., that
link the so-called conformal volume anomaly to higher Willmore
invariants \cite{GrRV,AGW,GWRenormVol}). 
 
Central to our approach to higher codimension submanifold theory is the fact that one higher codimension analogue of the normal
tractor is a conformally invariant alternating tractor form
$N_{A_1\cdots A_d}$, where $d=n-m$ is the codimension of the
submanifold.
An equivalent object to $N_{A_1\cdots A_d}$ is
its tractor Hodge-star, that we denote $\star \hspace{-0.5pt} N^{A_1 A_2 \cdots A_{m+2}}$,
see \eqref{tstar}.
For a submanifold $\Sigma$ (of any non-trivial codimension),
the intrinsic tractor bundle $\cT \Sigma$ can be identified with the
annihilator in $\cT M$ of $N_{A_1 A_2 \cdots A_d}$, see Section
\ref{sub-tr}. (This is still true when $\Sigma$ has dimension $1$ or
$2$, though in this case it is less obvious if one starts with the jet
bundle construction of $\cT \Sigma$; see the discussion in Section
\ref{lowdim} which treats, from our point of view here, the natural M\"{o}bius structures induced on low dimensional submanifolds, cf.\ \cite{CalderbankMobiusStructures}).  Thus one has an orthogonal decomposition of the
ambient tractor bundle $\cT M$,
\[
  \cT M|_{\Sigma} = \cT \Sigma \oplus \cN,
\]
which also defines the \emph{normal tractor bundle} $\cN$.
Denote by $\mathrm{N}^A _B$ the projection $\cT M|_{\Sigma} \to \cN$ (using abstract index notation).
There is also a projector $\cT M|_{\Sigma} \to \cT \Sigma$, and thus (provided $\dim M\geq 3$ so that the conformal structure on $M$ determines a canonical Cartan/tractor connection) one has a tractor Gau{\ss} formula which defines a tractor second fundamental form, $\LL_{i J} {}^C$; see \eqref{tractor_gauss} and its refinement \eqref{TracGaussFormula} (cf.\  \cite{CalderbankBurstall,CurryThesis}). By definition, $\LL$ measures the failure of $\cN$ (or, equivalently, of $\cT\Sigma \subseteq \cT M|_{\Sigma}$) to be parallel.

Using these tools we develop and present explicitly the fundamental
conformal submanifold calculus in Section \ref{geq3}, and, in
particular, the conformal Gauss-Codazzi-Ricci equations
\eqref{tractor_gauss_equation}, \eqref{tractor_codazzi},
\eqref{tractor_ricci}. Also, the normal forms, and their equivalents,
may be combined with standard conformal tractor calculus, and ideas
for using this to construct conformal manifold invariants (as in,
e.g., \cite{G-advances}), to manufacture submanifold invariants. This
is the subject of Section \ref{sinvts}. This is a powerful application,
as without such objects the construction of submanifold invariants is
complicated. Based on the ideas in \cite{G-advances} it seems likely that most, if not all, conformal invariants will arise using the tools developed here.

Although the tractor approach is conformal, scale-dependent quantities such as the mean curvature can be
nicely described by introducing an object called the {\em scale
tractor}, $I$, see Section \ref{scalet}. In particular minimal submanifolds
are seen to be exactly those submanifolds whose tractor
normal form is orthogonal to the scale tractor, see Corollary \ref{prop_minimal_normal_charac}, and
constant mean curvature notions are similarly captured, see
Proposition \ref{cmc-prop}. This means that these concepts generalise
to Poincar\'e-Einstein, and more generally conformally compact
manifolds, with the submanifold extending to the conformal infinity, as discussed in Section \ref{sec_minimal_scales}.

The definition here of a submanifold being {\em distinguished} is that $\LL_{i
  J} {}^C =0$, i.e.\ the vanishing of the tractor second fundamental
form.  A key result is that this may be alternatively captured as in
the following theorem.

\begin{thm}\label{key1}
  Let $(M, \bm{c})$ be a conformal manifold and $\Sigma \hookrightarrow M$ a conformal submanifold of codimension $d$.
  Then the following are equivalent:
  \begin{enumerate}
    \item $\LL_{i J} {}^{C} = 0$;
    \item $\nabla_i \mathrm{N}^{A_1} _{A_2} = 0$;
    \item $\nabla_i N_{A_1 A_2 \cdots A_{d-1} A_d} = 0$;
    \item $\nabla_i \star \hspace{-2.5pt} N^{A_1 A_2 \cdots A_{m+2}} = 0$,
  \end{enumerate}
  where $\nabla_i$ indicates the pullback to $\Sigma$ of the ambient tractor connection.
\end{thm}

A hypersurface has $\LL_{i J} {}^C = 0$ if and only if it is totally
umbilic, meaning the trace-free second fundamental form vanishes, but
for higher codimension it means a certain conditional invariant must
also vanish.  (When the trace-free second fundamental form vanishes
this conditional invariant becomes Belgun's $\mu$ invariant
\cite{BelgunFlorin2015}, and the relationship between $\mu$ and $\LL$
is discussed in Section \ref{subsec_weak_conf_circ}). Thus for
codimensions greater than one, a distinguished submanifold is
necessarily totally umbilic, but the converse is not true in
general. Moreover, in the case of $1$-dimensional submanifolds (where
the totally umbilic condition becomes vacuous) the vanishing of $\LL$
precisely characterises  conformal circles. That is, in our current
terminology, 1-dimensional conformally distinguished submanifolds are
exactly unparametrised conformal circles:

\begin{thm}\label{conf_circ_2ff_zero}
  Let $(M, \bc)$ be a conformal manifold and $\gamma \hookrightarrow
  M$ a curve.  Then $\gamma$ is an unparametrised conformal circle if,
  and only if, $\LL = 0$, or equivalently any one of the
  conditions in Theorem~\ref{key1} holds. In particular, $\LL = 0$ if and only if $\mu=0$, where $\mu$ is the conformal curvature.
\end{thm} 
\begin{proof}
Proposition 4.13 from~\cite{GST} asserts that the unparametrised
conformal circle equation is equivalent to a certain $3$-tractor being
parallel along the curve, and equation \eqref{Phi-is-star-N} of
Section \ref{sec_circles} below asserts that this $3$-tractor is
precisely $\star \hspace{-0.2pt} N^{A_1A_2A_3}$. The theorem therefore
follows from the equivalence of items 1 and 4 in Theorem~\ref{key1}.
\end{proof}

For the convenience of the reader we discuss Proposition 4.13 from~\cite{GST} in Section \ref{sec_circles} on conformal circles below.

Next we observe that combining Theorem~\ref{key1} and Theorem~\ref{conf_circ_2ff_zero} leads naturally to a generalisation of Theorem~\ref{conf_circ_2ff_zero} that characterises distinguished submanifolds in terms of their relation to the ambient conformal circles:

\begin{thm}\label{weakly}
  A submanifold $\Sigma$ is distinguished ($\,\LL = 0$) if, and only if, it is weakly conformally circular (equivalently, $\mathring{\II}=0$ and $\mu=0$).
\end{thm}

Here, adapting terminology from~\cite{BelgunFlorin2015}, \emph{weakly
  conformally circular} means that an ambient conformal circle with
tangential initial conditions remains in the submanifold for some
time, cf.~\cite{BEG} for the case of hypersurfaces.
Since $\LL=0$ can be seen to be equivalent to $\mathring{\II}=0$ and $\mu=0$ by direct calculation, this recovers Belgun's~\cite[Theorem 5.4(2)]{BelgunFlorin2015}. 

There is an alternative natural notion of conformal circularity for
submanifolds, namely that any submanifold conformal circle is also an
ambient conformal circle. This is stronger than the previous notion
and we refer to such a submanifold as \emph{conformally circular}. It
turns out that for this property, in addition to any of the
requirements of Theorem~\ref{key1}, one also requires that the
trace-free part of the Fialkow tensor vanishes; see
Theorem~\ref{thm_totally_conformally_circ_unparam}. Examples of conformally circular submanifolds include a factor trivially included in an Einstein product or any umbilic submanifold in a complex projective space (see Section \ref{sec-examples}).  Taking into
account parametrisations gives rise to yet another notion of conformal
circularity, which we refer to as being \emph{strongly conformally
  circular}. This was considered in \cite{BelgunFlorin2015} and our Theorem~\ref{thm_totally_conformally_circ_unparam}
should be contrasted with the analogous result Theorem~\ref{thm_totally_conformally_circ_param} (which recovers Belgun's \cite[Theorem
  5.4(3)]{BelgunFlorin2015}) for projectively parametrised conformal circles. A number of examples are considered in Section~\ref{sec-examples} and we find that strongly conformally circular submanifolds arise a lot less commonly than conformally circular ones, though they include the standard umbilic submanifolds in the conformal sphere and any factor trivially included in a special Einstein product. 
	
In Section \ref{sec-examples} we clarify by way of a number of examples that, outside of edge cases in the codimension (i.e.\ the cases of curves and of hypersurfaces), the notions of umbilic, distinguished, conformally circular and strongly conformally circular submanifolds are distinct. On the other hand, we show that when the geometry of $(M,\bc)$ is special then often umbilic submanifolds are automatically distinguished. For example, any umbilc submanifold in real, complex or quaternionic space form is forced to be distinguished. In support of our contention that the distinguished submanifold condition is a more natural conformal analog of the totally geodesic condition in Riemannian geometry than the totally umbilic condition, we also observe (in Theorems \ref{thm-twisted-prod} and \ref{thm-conf-prod}) that we have the following conformal analog of the classical de Rham-Wu theorem in Riemannian geometry: A conformal manifold $(M,\bc)$ is locally the conformal structure of a product manifold if and only if $(M,\bc)$ possesses a pair of complementary orthogonal foliations by distinguished submanifolds.

Another advantage that distinguished submanifolds posses over (merely) umbilic ones is that points 2.--4.\ of Theorem~\ref{key1}  characterise distinguished submanifolds in a way that immediately allows the proliferation of conserved quantities.
As mentioned above, Killing tensors, Killing-Yano tensors and their conformal analogues are well-established as tools for providing first integrals for geodesics.
These are each examples of solutions to first BGG equations, a large class of overdetermined natural equations~\cite{CSS2000,CD}.
For such equations, there is a class of solutions called \emph{normal} solutions that are in one-to-one correspondence with parallel sections of the corresponding tractor bundle~\cite{CGH-Duke}.
In particular, on conformally flat manifolds, all solutions to first BGG equations are normal.
Let us state our result rather informally as follows.

\begin{cor}\label{fi-cor}
Suppose a conformal manifold admits a BGG normal solution corresponding to a parallel tractor $S$, and $\Sigma$ is a distinguished submanifold.
Let  
\[
  \langle \otimes^\ell S, \otimes^k N \rangle
\]
denote a scalar quantity constructed from linear combinations of tensor powers of $S$ and linear combinations of tensor powers of the normal tractor form $N$ and with contractions using the conformal tractor metric and possibly the tractor volume form.
Then $\langle \otimes^\ell S, \otimes^k N \rangle$ is a first integral for the distinguished submanifold.
\end{cor}

\noindent This result generalises the large family of conformal circle
first integrals constructed in \cite{GST} to the case of distinguished
submanifolds of arbitrary dimension. We should say that rather than
using the tractor normal form in the Corollary \ref{fi-cor} above one
may equally alternatively (or additionally use) $\star N$, or
$N^A_B$. Precise statements can be found in Section~\ref{fi-thm-sec},
where we also show that it is easy to compute explicit examples.

Towards another key application, we show that there is yet another
characterisation of distinguished conformal submanifolds that takes
the form of a moving incidence relation.  For this we need the first
elements of conformal tractor calculus.  On any smooth manifold, one
has the bundle of conformal 1-densities that we call $\cE[1]$, which
is a root of the squared canonical bundle, see Section
\ref{conv-c-geom}. Its 2-jet bundle $J^2 \cE[1]$ admits the exact
sequence at 2-jets,
\begin{equation}\label{exact-2-jets}
  0 \to S^2 T^*M [1] \to J^2 \cE[1] \to J^1 \cE[1] \to 0,
\end{equation}
where $\cV[w] := \cV \otimes \cE[w]$ for any vector bundle $\cV$ and
any $w\in\mathbb{R}$.

The introduction of a conformal structure determines a canonical
splitting of $S^2 T^*M [1]$ as  $S^2 _0 T^*M [1] \oplus \bg\cdot \cE[-1]$, where $\bg\in\Gamma(S^2 T^*M [2])$ is the conformal metric.  The
standard conformal cotractor bundle $\cT^*$ (or $\cT^* M$) is the quotient of $J^2
\cE[1]$ by the image of $S^2 _0 T^*M[1]$ and so has a filtration as
given by the exact sequence
\begin{equation}\label{exact}
  0 \to \cE[-1] \stackrel{X}{\to} \cT^* \to J^1 \cE[1] \to 0.
\end{equation}
(In the case that $M$ is of dimension one then $S^2 _0 T^*M[1]$ is
trivial and $\cT^*=J^2\ce[1]$.)
There is canonically a conformally invariant  metric $h$ on the bundle
$\cT^*$, and hence $\cT^*$ is identified with its dual $\cT$, which we
call the \emph{tractor bundle}.  The bundle injection $X$ which maps
$\cE[-1] \to \cT^*$ is typically viewed as a section $X \in
\Gamma(\cT^*[1])$, and called the \emph{canonical tractor}.  This
invariantly encodes information about position on the manifold and
plays a very important role in our developments here.

A well known feature of $\cT^*$ is that, in the case where $\dim M
\geq 3$, it is naturally equipped with the canonical conformally
invariant tractor connection~\cite{BEG}, which is equivalent (see \cite{CG-tams}) to the
normal Cartan connection as in \cite{cartan1923espaces}.  This preserves the tractor metric.  Using this object and language we have the following result.

\begin{thm}\label{thm_submanifold_gst}
   Let $\Sigma$ be an embedded submanifold of codimension $d$ in a conformal manifold $(M, \bm{c})$. 
  Then $\Sigma$ is distinguished if, and only if, either (equivalently both) of the following holds
  \begin{itemize}
    \item there exists a nontrivial $\Psi \in \Gamma(\Lambda^d \cT^*)$ such that $X \intprod \Psi = 0 $ and $\nabla_i \Psi = 0$ along $\Sigma$, or
    \item there exists a nontrivial $\star \hspace{-0.5pt} \Psi \in \Gamma(\Lambda^{n+2-d} \cT^*)$ such that $X \wedge \star \hspace{-0.5pt} \Psi = 0$ and $\nabla_i \star \hspace{-2.7pt} \Psi = 0$ along $\Sigma$.
  \end{itemize}
\end{thm}


Again this generalises a result for non-null conformal circles from~\cite{GST}.
If either of the above conditions hold, then $\Psi$ is necessarily (up to locally constant factor) the tractor normal form of the submanifold. 
This will be proved in Section~\ref{corbit-sec}. Note that in the model case of the conformal sphere, viewed as the projectivised null cone of a Minkowski space $\mathbb{M}$ of two higher dimensions, the distinguished submanifolds all arise from cutting the null cone with a subspace (and projectivising); in this case the tractor $\star \Psi$ corresponding to a distinguished submanifold $\Sigma$ is constant and simple (it is the wedge product of vectors that span the corresponding subspace of the Minkowski space) and, identifying the canonical tractor $X$ with the position vector in $\mathbb{M}$, the condition $X\wedge\star\hspace{-0.5pt} \Psi=0$ is precisely the incidence relation saying that $X$ is a point in the subspace corresponding to $\star\hspace{-0.5pt} \Psi$.

Theorem~\ref{thm_submanifold_gst} is a useful result in that it allows us to immediately
conclude that certain zero loci of normal solutions of appropriate BGG
equations are distinguished submanifolds. Recall that on a Riemannian
manifold, an alternating tensor $k$ of degree $d$ is a
conformal Killing form if the trace-free part of $\nabla k$ is completely alternating.  For a suitable
conformal weight, this condition is conformally invariant, see Section
\ref{ex-sect}.  Combining the curved orbit theory of \cite{CGH-Duke} with 
Theorem~\ref{thm_submanifold_gst} we obtain the following, where the operator $L$ is explained in
\eqref{eqn_conf_killing_splitting_op} (and Theorem \ref{normp}).
\begin{thm}\label{zero-locus}
Suppose $k$ is a normal solution of the conformal
Killing form equation on $(M, \bm{c})$ of degree $d-1$ such that the parallel tractor
$\mathbb{K} =L(k) \in \Gamma(\Lambda^d \cT^*)$ is simple. 
Then the zero locus of $k$ is either empty, an isolated point, or a distinguished submanifold of codimension $d$.
\end{thm}

\begin{rem}\label{rem-zero-locus-thm}

\emph{(i)} The hypothesis that the parallel tractor $\mathbb{K}$ be simple is natural in that only in this case can $\mathbb{K}$ be a multiple of the tractor normal form of a codimension $d$ submanifold. Without the simplicity condition, higher than codimension $d$ zero loci are possible.
\emph{(ii)} If $d=1$, then $\sigma=k$ is a weight $1$ conformal density and the theorem still holds if we regard the hypothesis on $\sigma=k$ as saying that $\sigma$ is an almost-Einstein scale in the sense of \cite{GRiemSig} and interpret $\mathbb{K}=L(k)\in \Gamma(\cT^*)$ as the corresponding parallel scale tractor $I$. In this case $X \intprod \mathbb{K}=\sigma$, and the conclusion is a known result about the zero locus of an almost-Einstein scale \cite{GRiemSig,Curry-G-conformal}. \emph{(iii)} When $d=1,2$, $k$ does not need to be normal; see Section \ref{subsec-curved-orbits}.
\end{rem}
  
\noindent In fact finer information about the zero locus is available, see Theorem
\ref{thm_submanifold_zero_locus}.  This is an analogue for normal
Killing solutions of the results for almost Einstein scales found in
\cite{gover2004almost,GRiemSig,Curry-G-conformal}. Those results for
Einstein scales (and their generalisations to so-called ASC scales in
\cite{GRiemSig}) were key in the previously mentioned development of a holographic approach
to hypersurfaces via a singular Yamabe problem in
\cite{AGW,BGW,GWWillmore,GWRenormVol,GWConfHypYamabe},
as well as a
boundary calculus of asymptotically hyperbolic manifolds
\cite{GWBoundaryCalc}.  We believe the results in Theorem
\ref{thm_submanifold_zero_locus} should provide one of the key
insights for the analogous treatment of submanifolds of higher
codimension (which need not be distinguished).  Indeed, toward this end we provide a simple direct proof
of a similar zero locus result, for fields satisfying weaker (than normal
Killing-Yano) conditions, in Proposition \ref{progress}.

\section{Conventions and conformal geometry} \label{conv-c-geom}

Often we will use the standard abstract index notation of Penrose.  For
example we may write $\ce^a$ for the tangent bundle $TM$ of a manifold
$M$ and $v^a$ for a vector field on $M$. Similarly $\ce_a$ denotes the
cotangent bundle $T^*M$, and $\omega_a\in \Gamma(\ce_a)$ a 1-form
field.  Then we write $v^a\omega_a$ for the canonical pairing between
vector fields and $1$-forms. We denote by the Kronecker delta
$\delta^b{}_a$ the identity section of the bundle $\textrm{End}(TM)$
of endomorphisms of $TM$. Indices enclosed by round (respectively by
square) brackets indicate symmetrisation (respectively
skew-symmetrisation) over the enclosed indices. For example, if $T_{ab}$
is a rank 2 tensor then
$$
T_{(ab)}=\frac{1}{2}(T_{ab}+T_{ba})\qquad
\mbox{and}\qquad T_{[ab]}=\frac{1}{2}(T_{ab}-T_{ba}).
$$
We also use this notation for bundles. 
For example, $\cE_{[a_1 a_2 \cdots a_d]}$ denotes the bundle of $d$-forms.
When tractor bundles
are introduced these will also be adorned with abstract indices when
convenient, with the same convention for symmetrisation and
skew-symmetrisation. 

For simplicity of exposition we assume throughout that the basic manifold $M$ studied is  connected.

\subsection{Conventions for Riemannian geometry}

A Riemannian manifold is a pair $(M^n ,g)$, consisting of a manifold
$M$ and a positive definite metric $g$. We assume that the dimension
$n$ (of $M$) is at least 2.  All structures are assumed smooth,
meaning $C^\infty$. This is to simplify the discussion. For all the
theory a much lower level of regularity is required, but this varies
throughout and at any point is easily calculated by the reader.
We will also typically assume for convenience that $M$ is oriented, with volume form $\epsilon_{a_1 a_2 \cdots a_n}$ normalised by $\epsilon^{a_1 a_2 \cdots a_n} \epsilon_{a_1 a_2  \cdots a_n} = n!$, where indices are raised using the inverse of the
metric $g$.

Writing $\nabla$ for the Levi-Civita connection, the \emph{Riemannian curvature tensor} $R_{ab} {}^c {}_d$ is defined by
\begin{equation}\label{curvature}
  R_{ab} {}^c {}_d v^d = [ \nabla_a, \nabla_b] v^c.\, \qquad v^a\in \Gamma(\ce^a),
\end{equation}
in the abstract index notation.

In dimensions $n\geq 3$, this decomposes into trace-free and a trace part:
\begin{align*}
  R_{abcd} & = W_{abcd} + 2 \, g_{c[a} \Rho_{b]d} - 2 \, g_{d[a} \Rho_{b]c} \, ,
\end{align*}
where $W_{ab}{}^{c}{}_{d}$ is the {\em Weyl tensor} and $\Rho_{ab}$
is the  {\em Schouten tensor}. 
Equivalently, the Schouten tensor is characterised by 
\begin{equation}
  R_{ab} = (n-2) P_{ab} + J g_{ab},
\end{equation}
where $R_{ab}:=R_{ca}{}^c{}_b$ is the Ricci tensor, and $J := g^{ab}
P_{ab}$. The Weyl tensor is totally trace free, and satisfies the
algebraic Bianchi identities. In dimension 3 this implies that the Weyl
tensor is zero.

The \emph{Cotton} tensor (also for $n \geq 3$) is defined by
\begin{equation}
  C_{abc} := 2 \nabla_{[a} P_{b] c}.
\end{equation}

In dimension 2, it is easy to show that the Riemannian curvature is
pure trace:
\begin{equation*}
R_{abcd} = K(g_{ik}g_{jl}-g_{il}g_{jk}),
\end{equation*} 
where $K$ is the Gau{\ss}ian curvature. Hence the Ricci tensor is also pure trace, and the Weyl curvature is zero. 

Later, we will need to consider 1-dimensional submanifolds $\Sigma$
equipped with a Riemannian metric.  On a 1-dimensional manifold
$\Sigma$, a Riemannian metric $g_\Sigma$ takes the form $u \otimes u$,
where $u$ is a non-vanishing 1-form, and requiring $u$ to be the
volume form corresponding to $g_\Sigma$ and the orientation fixes the
sign of $u$.  Thus there is a unique connection preserving this
metric, namely the connection $D$ that preserves $u$.  We
will term this the Levi-Civita connection for $(\Sigma, g_\Sigma)$.
The curvature of any such connection is clearly zero.

\subsection{Conformal geometry}\label{sec_conf_geo}

Two metrics $g, \hat{g}$ are said to be \emph{conformally related} if 
\begin{equation}\label{conf_equivalence}
  \hat{g} = \Omega^2 g,
\end{equation}
where $\Omega \in C^\infty (M)$ is a positive function.  Then $\cc$
denotes an equivalence class of conformally related metrics, i.e.\ if
$g, \hat{g} \in \cc$, then they are related according
to~\eqref{conf_equivalence} for some smooth $\Omega$, and we may write
$\cc = [g]$.  A conformal manifold is then a pair $(M, \cc)$.

Recall that on a manifold $M$, for any $\alpha \in \R$, one has the
bundle of $\alpha$-densities.  This is the associated bundle to the
linear frame bundle of $M$ via the 1-dimensional
$\GL(n)$-representation $A \mapsto |\det(A)|^{-\alpha}$.  Sections of
this bundle are called $\alpha$-densities.  There is a correspondence
between 1-densities and sections of $\Lambda^n T^*M$ when $M$ is
oriented, or in general between the square of these bundles; e.g., if $M$ is oriented then, given a local oriented frame field $(e_1,\ldots, e_n)$, the function representing a (local) section $v$ of $\Lambda^n T^*M$ as a $1$-density is simply $v(e_1,\ldots, e_n)$.


Separately, we have the bundle of conformal densities of weight $w$, which we denote by $\cE[w]$, and which are defined by 
\begin{equation}
  \cE[w] := \cQ \times_{\rho_w} \R, 
\end{equation}
where $\cQ$ is the ray bundle of conformally related metrics and
$\rho_w$ is the 1-dimensional $\R_+$-representation $\rho_w(s)(t) :=
s^{-\frac{w}{2}}t$. The bundles $\ce[w]$ are evidently oriented and we
write $\ce_+[w]$ for the ray subundle of positive elements.  As
detailed in, e.g.,~\cite{CGAmbientMetric}, the conformal densities of
weight $w$ are in bijective correspondence with densities of weight
$\left( -\frac{w}{n} \right)$.  In particular, this means that
1-densities also correspond to conformal densities of weight $-n$, and
so together with the discussion of the previous paragraph we have an
isomorphism $(\Lambda^n T^*M)^2 \stackrel{\cong}{\to} \cE[-2n]$, and
dually $(\Lambda^n TM)^2 \stackrel{\cong}{\to} \cE[2n]$.  If
$\mathcal{B}$ is a vector bundle on $M$ we will write $\mathcal{B}[w]$
as a shorthand for $\mathcal{B}\otimes [w]$.

If $M$ is oriented, as we henceforth assume, we write $\be = \be_{a_1
  a_2 \cdots a_n} \in \Gamma(\cE_{[a_1 a_2 \cdots a_n]} [n])$ for the
canonical map $\Lambda^n TM \to \cE[n]$, given by contraction, and
call $\be_{a_1 a_2 \cdots a_n}$ the \emph{conformal volume form} or
\emph{weighted volume form}.  Since $\cE[w]$ is an associated bundle,
its sections may be thought of as equivariant functions $f : \cQ \to
\R$ such that $f(s^2 g_x) = s^w f(g_x)$.  So we may think of a section
of $\cE[w]$ as an equivalence class of pairs $(g,f)$, where $(g,f)
\sim (\Omega^2 g, \Omega^w f)$. The conformal volume form can
therefore similarly be thought of as the equivalence class of
$(g,\epsilon)$ for any $g\in\bc$, where $\epsilon$ is the Riemannian
volume form of $g$ and $(g,\epsilon)\sim (\Omega^2g,
\Omega^{n}\epsilon)$.

Corresponding to $g\in\cc$ there is evidently a corresponding section $\sigma_g\in \Gamma(\ce[1])$, represented by the pair $(g,1)$. 
It follows that the conformal structure $\cc$ determines a tautological section $\bg \in \Gamma(S^2 T^*M \otimes
\cE_+[2])$ that is given by
\begin{equation} \label{c-met}
  \bg = (\sigma_g)^2 g,
\end{equation}
for any metric $g\in \cc$ (but which is independent of this choice);
equivalently, the tautological section $\bg \in \Gamma(S^2 T^*M
\otimes \cE_+[2])$ may be thought of as the equivalence class of
$(g,g)$ for any metric $g\in \bc$, where $(g,g)\sim (\Omega^2g,
\Omega^2g)$.  This is called the \emph{conformal metric}. We will
henceforth typically use the conformal metric to raise and lower
indices, even when a choice of $g\in \cc$ has been made (thus raising
and lowering indices typically introduces a density bundle
weight). For example the Riemann curvature with indices all down
$R_{abcd}$ will now be considered to have weight 2 as it is
$$
R_{abcd}=\bg_{ec}R_{ab}{}^e{}_d ~ ,
$$
the Weyl-Schouten decomposition of the Riemann curvature becomes 
\begin{align*}
  R_{abcd} & = W_{abcd} + 2 \, \bg_{c[a} \Rho_{b]d} - 2 \, \bg_{d[a} \Rho_{b]c} ~ ,
\end{align*}
for dimensions $n\geq 3$, and $J$ will mean $\bg^{ab}P_{ab}$
From (\ref{c-met}) we see that if we use $\sigma_g$ to trivialise density bundles, then the conformal metric $\bg$ becomes $g$. However, usually we will avoid trivialising density bundles. This becomes significant when we write down conformal rescaling laws, since then there are two different metrics that could be used to trivialize the density bundles (and correspondingly two different flat connections on sections of density bundles) and many formulae are simplified when we work with weighted objects.

\medskip

Each metric $g\in \cc$ determines a corresponding Levi-Civita
connection $\nabla$. This naturally acts on sections of density
bundles and, tautologically from the construction above, preserves
$\sigma_g$. Thus as well as preserving $g$, the Levi-Civita
connection $\nabla$ preserves $\bg$ and
$\be$ (cf.\ \cite{Curry-G-conformal}).  Under a change to
$\widehat{g}=\Omega^2 g\in \cc$ we have
\begin{equation}\label{vector_trans}
  \hat{\nabla}_a u^b = \nabla_a u^b + \Upsilon_a u^b - \Upsilon^b u_a + \Upsilon_c u^c \delta_a ^b,  \qquad \mbox{on} \,\,\, u^b\in \Gamma(\ce^b), 
\end{equation}
\begin{equation}\label{covector_trans}
  \hat{\nabla}_a \omega_b = \nabla_a \omega_b - \Upsilon_a \omega_b - \Upsilon_b \omega_a + \Upsilon^c \omega_c \bg_{ab}, \qquad \mbox{on} \,\,\, \omega_b\in \Gamma(\ce_b),
\end{equation}
and
\begin{equation}\label{dens-trans}
\hat{\nabla}_a \tau = \nabla_a \tau +w \Upsilon_a \tau  \qquad \mbox{on}\,\,\, \tau\in \Gamma(\ce[w]),
\end{equation}
where $\Upsilon_a:=\Omega^{-1} \nabla_a \Omega $.

The Weyl curvature $W_{ab}{}^c{}_d$ is conformally invariant, while the Schouten tensor transforms according to,
\begin{equation}\label{schouten_trans}
  \hat{P}_{ab} = P_{ab} - \nabla_a \Upsilon_b + \Upsilon_a \Upsilon_b - \frac{1}{2} \Upsilon_c \Upsilon^c \bg_{ab} .
\end{equation}
Equations~\eqref{vector_trans} and~\eqref{covector_trans} still hold when $M$ has dimension 1, although the final two terms of both equations cancel.
Equation~\eqref{schouten_trans} only holds when $\dim M \geq 3$, since in lower dimensions the Schouten tensor is not defined.

\subsection{The tractor connection and calculus}\label{tractor_calc}

Recall from the introduction, the tractor bundle is recovered from jets of the conformal density bundle $\cE[1]$.
The inverse of the conformal metric maps $S^2T^*M [1] \to \cE[-1]$, with kernel $S^2 _0 T^*M [1]$ and hence we have a decomposition 
\begin{equation*}
  S^2 T^*M [1] = S^2 _0 T^* M[1] \oplus \bg \cdot \cE[-1].
\end{equation*}

Then the bundle $\cT^*$ (that we also denote $\ce_A$ in the abstract index notation)
is $J^2 \cE[1]$ modulo the image of $S^2 _0 T^* M[1]$ under the map $S^2 T^*M [1] \to J^2 \cE[1]$ of the jet exact sequence at 2-jets~\eqref{exact-2-jets}.
Thus we obtain~\eqref{exact}.
From the jet exact sequence at 1-jets 
\begin{equation}\label{1-jet}
  0 \to T^*M [1] \to J^1 \cE[1] \to \cE[1] \to 0,
\end{equation}
we then see that $\cT^*$ has the composition series 
\begin{equation*}
  \cT^* = \cE[1] \lpl T^*M [1] \lpl \cE[-1].
\end{equation*}
Here the semidirect sum notation $\lpl$ simply encodes the information of the exact sequences (cf.\ \cite{BEG}).
Note that this construction still applies when $M$ has dimension 1 or 2, but in dimension 1, $S^2 _0 T^*M [1]$ is trivial and hence $\cT^*$ is simply $J^2 \cE[1]$.
Recall that we denote by $X_A \in \Gamma(\cE_A [1])$ the canonical tractor which provides the embedding $\cE[-1] \to \cE_A$.
Let us also note that by the definition of the tractor bundle, there is an invariant differential operator $\TD : \Gamma(\cE[1]) \to \Gamma(\cT^*)$, where $\frac{1}{n} \TD$ is the differential operator corresponding to the linear map $J^2 \cE[1] \to \cT^*$.

Let us now fix $n \geq 3$.
Given a choice of metric $g\in\bc$, the formula
\begin{equation}\label{thomas_D}
  \sigma \mapsto \frac{1}{n} {[\TD_A \sigma]}_g := 
  \begin{pmatrix}
    \sigma \\ 
    \nabla_a \sigma \\ 
    - \frac{1}{n} \left( \Delta + J \right) \sigma 
  \end{pmatrix},
\end{equation} 
where $\Delta=\nabla^a\nabla_a$, gives a second-order differential operator on $\cE[1]$ which is a linear map $J^2 \cE[1] \to \cE[1] \oplus \cE_a [1] \oplus \cE[-1]$ that clearly factors through $\cT^*$ and so determines an isomorphism
\begin{equation}\label{T_isom}
  \cT^* \stackrel{\sim}{\longrightarrow} {[\cT^*]}_g = \cE[1] \oplus \cE_a [1] \oplus \cE[-1],
\end{equation}
and hence the sequences~\eqref{exact} and~\eqref{1-jet} split, as
discussed in, e.g., \cite{CG-irred,Curry-G-conformal}. When using a choice of metric $g$ to split the tractor bundle we will typically indicate this by writing $\overset{g}{=}$ rather than applying the bracket notation $[ \,\cdot\, ]_g$ to the object we are breaking up into slots.

In the subsequent discussions, we will use~\eqref{T_isom} to split the tractor bundles without further comment.
Thus, given $g \in \bc$, an element $V_A$ of $\ce_A$ may be represented by a triple
$(\si,\mu_a,\rho)$, or equivalently by
\begin{equation}\label{Vsplit}
  V_A=\si Y_A+\mu_a Z_A ^a+\rho X_A.
\end{equation}
The last display defines the algebraic splitting
operators $Y:\ce[1]\to \cT^*$ and $Z :T^*M[1]\to \cT^*$ (determined by the
choice $g\in \bc$) which may be viewed as sections $Y_A\in
  \Gamma(\ce_A[-1])$ and $Z_A ^a\in \Gamma(\ce_A ^a [-1])$.
  We call these sections $X_A, Y_A$ and $Z_A ^a$ \emph{tractor projectors}.
Note that with this convention,~\eqref{thomas_D} is, tautologically, an explicit formula for the invariant operator $\mathbb{D}$, in terms of the splitting given by the choice of metric $g$.

While $X_{A}$ is conformally invariant, a change of tractor splitting given
by~\eqref{conf_equivalence} determines the transformations
\begin{equation}\label{tractor_proj_transforms}
  \wh{Z}_{A} ^a = Z_{A} ^a + \Upsilon^a X_{A}, \hspace{2em} \wh{Y}_{A} = Y_{A} - \Upsilon_a Z_{A} ^a - \frac{1}{2} \Upsilon^a \Upsilon_a X_{A}
\end{equation}
where, as usual, $\Upsilon_a=\Omega^{-1} \nabla_a \Omega $.
These transformations mean that the tractor triples transform by
\begin{equation}\label{3-trans}
 \begin{pmatrix}
      \hat{\sigma}  \\
      \hat{\mu}_a\\
      \hat{\rho}
    \end{pmatrix} 
		=
 \begin{pmatrix}
      1 & 0 & 0 \\
      \Upsilon^b & \delta^b _a & 0 \\
      -\frac{1}{2} \Upsilon^c \Upsilon_c & - \Upsilon_a & 1
    \end{pmatrix}  
		\begin{pmatrix}
      \sigma \\
      \mu_b \\
      \rho
    \end{pmatrix}.
\end{equation}

One then observes that the symmetric tractor field given by 
\begin{equation}\label{tr-met-form}
  h^{AB} := 2X^{(A}Y^{B)}+\bg^{ab}Z^A _a Z^B _b
\end{equation}
is invariant under~\eqref{tractor_proj_transforms}, and so determines a conformally invariant metric on $\cT^*$.
We will hence use this and its inverse $h_{AB}$ (called the \emph{tractor metric}) to identify $\cT^*$ and its dual, the \emph{standard tractor bundle}, which we denote by simply $\cT$.
Using this we obtain
\begin{equation}\label{projector_contractions}
  X^AY_A=1, \hspace{2em} Z_A ^a Z^A _b=\delta^{a}_b,
\end{equation}
and all other (tractor-index) pairings of the
splitting operators give a zero section. For example $X^A X_A=0$.

The canonical conformally invariant (normal) tractor connection on $\cT$ will also be denoted $\nabla_a$, or sometimes $\nabla_a ^\cT$ for emphasis.
It can be coupled to the Levi-Civita connection of any metric $g \in \bc$, and its
action on the tractor projectors is then given by
\begin{align}\label{ctrids}\edz{SS: Also give standard formula for $\nabla$ in slots?}
  \nabla_a X^{A} & =  Z^{A}_a \, , & \nabla_a Z^{A}_b & = - P_{ab} X^{A} - \bg_{ab} Y^{A} \, , &
  \nabla_a Y^{A} & = P_a{}^{b} Z^{A}_b \, .
\end{align}
In fact, these formulae determine the tractor connection as the
general action on a section of a tractor bundle follows from the
Leibniz rule.  It is easily verified that the tractor connection is
conformally invariant and preserves the tractor metric. The latter
means that the tractor connection agrees with its dual. It extends in
the obvious way to tensor powers of the tractor bundle and these extensions are all referred to as the tractor connection. 
The coupled tractor-Levi Civita connection will always be denoted simply $\nabla$ and will be used, usually without  comment, according to context.

As with any linear connection, $\nabla= \nabla^\cT$ has a
curvature. The \emph{tractor curvature} $\Omega_{ab}{}^{C}{}_{D} $ of
the tractor connection is defined by $\Omega_{ab}{}^{{C}}{}_{{D}}
\Phi^{D} := 2 \, \nabla_{[a} \nabla_{b]} \Phi^{C}$, for any $\Phi^A
\in \Gamma( \mc{T} )$. In the splitting determined by a choice of
metric $g\in\cc$ it is given explicitly by the formula
\begin{align}\label{ctract-curv}
  \Omega_{ab}{}_{C D} & = W_{abcd} Z_{C}{}^c Z_{D}{}^d -2 C_{abc} X^{\phantom{c}}_{[C} Z_{D]}{}^c .
\end{align}
A conformal structure is said to be \emph{(locally) flat} if this
tractor curvature vanishes as this happens if and only if, locally,
there is a metric in the conformal class that is flat.

The tractor objects developed above form the initial objects of a
conformal tractor calculus that can be used, for example, to construct
conformal invariants \cite{G-aspects,G-advances}. We will not discuss this in
detail, but one particularly important object is the {\em Thomas
operator} $\TD$ that extends \eqref{thomas_D} to a conformally invariant operator between weighted tractor bundles,
$$
 \TD_A:\Gamma (\ce^{\Phi}[w]) \to \Gamma (\ce_A\otimes \ce^{\Phi}[w-1]) ,
$$
where $\ce^{\Phi}$ indicates any tensor power of $\ce^A$, or $SO(h)$-invariant part thereof. It is 
 given, with respect to $g\in \cc$, by the formula
\begin{equation}\label{tD-formula}
\Gamma(\ce^\Phi[w])\in V \mapsto \TD_A V \stackrel{g}{=} 
\left(\begin{array}{c}(n+2w-2)w V\\
(n+2w-2)\nabla_a V\\
-(\Delta V + w J V)\end{array}\right) .
\end{equation}
where (as usual) $\nabla$ is the coupled tractor-Levi-Civita connection and $\Delta$ the corresponding Laplacian. 

All of the above has a clear geometric interpretation in the case of the model, the conformal $n$-sphere. This should be thought of as the ray projectivisation of $\mathcal{C}_+$, where $\mathcal{C}_+$ is
the future directed part of the null quadric
$\mathcal{C}:=\{X\in\mathbb{R}^{n+2} \mid h(X,X) =0\}$ in
$\mathbb{R}^{n+2}$ equipped with a fixed symmetric non-degenerate
bilinear form $h$ of signature $(n+1,1)$ and a time-orientation. The resulting resulting manifold $M:=\mathbb{P}_+(\mathcal{C}_+)\cong S^n$ is acted on transitively by
$G: = \SO_0(h)\cong \SO_0(n+1,1)$, where the $0$ here denotes taking
the connected component of the identity, and the stabiliser of a point
is a parabolic subgroup that we denote $P$ (so $M\cong G/P$).  Moreover, it is
straightforward to verify that $h$ induces a Riemannian metric on each
section of the map $\mathcal{C}_+\to M$, and different sections result
in conformally related metrics. Thus $M$ is equipped canonically with
a conformal structure, and clearly the group $G$ acts on $M$ by
conformal isometries, see, e.g., \cite{Curry-G-conformal,GRiemSig} for
a more detailed discussion of this model.

  From this point of view the standard tractor bundle for the model is
  $T\R^{n+2}|_{\mathcal{C}_+}/\sim$, where the equivalence relation is
  $U_p\sim V_q$ if one is mapped to the other by standard $\R^{n+2}$
  parallel transport (i.e., from the affine structure of $\R^{n+2}$)
  along a null ray.  The tractor metric is then induced in an obvious
  way from the ambient Minkowski signature metric $h$, and parallel
  tractor fields are equivalent to vector fields in
  $\Gamma(T\R^{n+2}|_{\mathcal{C}_+})$ that are constant along
  $\mathcal{C}_+$.  Moreover the parallel tensor fields on connected
  regions of $\mathcal{C}_+$ may all be viewed as arising from the
  restriction of tensor fields parallel on $\R^{n+1,1}$, and these
  give the parallel sections of the corresponding tensor powers of the
  tractor bundle. Finally, in this picture, the canonical tractor
  $X^A$ is identified with the Euler vector field of
  $\mathbb{R}^{n+2}$ along $\mathcal{C}_+$.

The Thomas operator $\TD_A$ also has a concrete geometric interpretation in the model. Sections of the weight $w$ conformal density bundle on the model can be identified with functions on $\mathcal{C}_+$ that are homogeneous of degree $w$ with respect to the $\mathbb{R}_+$-action. Weighted tractors on the model can therefore be identified with tensor fields along $\mathcal{C}_+$ of the appropriate homogeneity. The Thomas operator $\TD_A$ on such sections is then given (up to an overall factor) by formally extending such tensor fields off $\mathcal{C}_+$ to be ``harmonic'' with respect to the ambient Minkowski metric $h$ and then taking the directional derivative (at points along $\mathcal{C}_+$) in the flat ambient space $\mathbb{R}^{n+2}$; see, \cite{CGAmbientMetric}.

\subsection{The scale tractor} \label{scalet}

Recall that, from Section \ref{sec_conf_geo}, a metric $g\in \cc$ is
equivalent to a section $\si_g\in \Gamma(\ce_+[1])$ by the relation
$$
g=\si_g^{-2} \bg. 
$$
Given any section $\si\in \Gamma(\ce[1])$ we can form
$$
I_A:= \frac{1}{n}\TD_A \si ,
$$
and we will term this a {\em scale tractor} if $I_A$ is nowhere
zero. In this case $\si$ is clearly non-vanishing on an open dense
subset of $M$, on which it determines a metric $g:=\si^{-2}\bg$ from
the conformal class. So for a scale tractor $I_A$ we will term
$\si=X^AI_A$ a {\em generalised scale} -- or sometimes simply a
\emph{scale}. Following \cite{GRiemSig}, a conformal manifold $(M,\cc)$
equipped with a {\em scale tractor} will said to be an {\em
  almost-Riemannian} manifold (since it has a metric almost
everywhere). Given a Riemannian metric $g=\si_g^{-2} \bg$, we term
$I_A:= \frac{1}{n}\TD_A \si_g$ the scale tractor of $g$.

It follows easily from \eqref{ctrids} that if a tractor $I_A\neq 0$ is
parallel then it is a scale tractor and $g:=\si^{-2}\bg$ is Einstein; see \cite{G-Nurowski,GRiemSig}.  In this case we say $(M,\cc,I)$ is
{\em almost Einstein}. 

An important example of almost-Riemannian manifolds arise in
connection with conformally compact manifolds: A complete Riemannian
manifold $(M,g)$ is \emph{conformally compact} if $M$ is the interior of a
manifold with boundary $\ol{M}$, and on $\ol{M}$ there is a metric
$\ol{g}$ (so a metric that is smooth up to the boundary) such that on $M$
$$
g_{ab}=r^{-2}\ol{g}_{ab}
$$ for some smooth defining function $r$ for the boundary $\partial M$ (meaning that 
$r>0$ on $M$,  $\partial M$ is the zero
locus of $r$, and $dr$ is nowhere zero on $\partial
M $). A conformally compact manifold is said to {\em asymptotically
  hyperbolic} if $|dr|_{\ol{g}}=1$ along $\partial M$ (which is equivalent to requiring that the sectional curvatures of $g$ all tend to $-1$ as one approaches $\partial M$) and {\em
  Poincar\'e-Einstein} if $g$ is Einstein. It is easily verified that
in the latter case the scalar curvature is negative. If the Poincar\'e-Einstein metric is normalised so that
$$
\operatorname{Sc}^g =-n(n-1)
$$
(as is usually assumed) then the manifold is necessarily
asymptotically hyperbolic. These structures have been the subject of sustained interest; see, e.g., \cite{Chang2018,ChangGe-survey,ChangGeQing,GrZ,GurskyHanStolz,LeeNeves2015,Vasy2013} and the many references therein.

It is easily verified that a conformal compactification of a manifold
$(M,g)$ is the same as a conformal manifold with boundary
$(\ol{M},\cc)$, with interior $M$, and equipped with a scale tractor
$I_A$ with the following properties: the zero locus
$Z(\si)=\si^{-1}(0)$ of $\si:=X^AI_A$ is $\partial M$, and along
$\partial M$ the 1-jet $j^1\si$ (of $\si$), is nowhere zero (we say that $\sigma$ is a \emph{defining density} for $\partial M$). Thus the
conformal compactification is almost-Riemannian; in the following we will therefore think of a conformally compact manifold as an almost-Riemannian manifold for which $\sigma$ is a defining density for $\partial M$. Such a manifold is asymptotically
hyperbolic if $I^AI_A=1$ along $\partial M$, and Poincar\'e-Einstein
if $I_A$ is parallel. If $I^AI_A=1$ on $\ol{M}$ then
$\operatorname{Sc}^g =-n(n-1) $. See,
e.g., \cite{GRiemSig,GWConfHypYamabe} for more details.

There are many structures such as certain notions of asymptotically
flat manifolds that can be similarly be understood in terms of
almost-Riemannian structures. So this notion provides a uniform
framework for approaching a range of singular geometries
\cite{Curry-G-conformal}.

\subsection{Form tractors} \label{formT-sec}

We will use the term \emph{form tractor} to describe sections of the exterior powers of the tractor bundle \cite{BransonGover05,G-Sil-ckforms}.
It is useful to introduce some notation for form tractors.
From the composition series for the standard tractor bundle, one sees that for the $k$-th exterior power of the standard tractor bundle, one has the composition series
\begin{equation}\label{eqn_tractor_form_comp_series}
  \cE_{[A_1 A_2 \cdots A_{k-1} A_k]} = \cE_{[a_2 \cdots a_k]} [k] \lpl 
  \begin{matrix} 
    \cE_{[a_1 a_2 \cdots a_{k-1} a_k]} [k] \\
    \oplus \\
    \cE_{[a_3 \cdots a_k]} [k-2]
  \end{matrix}
  \lpl \cE_{[a_2 \cdots a_k]} [k-2].
\end{equation}

The tractor projectors for the standard tractor bundle induce tractor projectors on the bundles of tractor forms. 
Since these will be very important for us, we introduce dedicated notation for these: 
\begin{equation}\label{form-ps}
\begin{matrix}
  \mathbb{Y} _{A_1 A_2 \cdots A_{k-1} A_k} ^{\phantom{A_1} a_2 \cdots a_{k-1} a_k}
  &:= & Y_{[A_1} Z^{a_2} _{A_2} \cdots Z^{a_{k-1}} _{A_{k-1}} Z^{a_k} _{A_k]}
  &\;\in& \cE^{\phantom{A_1} a_2 \cdots a_{k-1} a_k} _{[A_1 A_2 \cdots A_{k-1} A_k]} [-k] \\[0.7em]
  \mathbb{Z}^{a_1 a_2 \cdots a_{k-1} a_k} _{A_1 A_2 \cdots A_{k-1} A_k}
  &:= &Z^{a_1} _{[A_1} Z^{a_2} _{A_2} \cdots Z^{a_{k-1}} _{A_{k-1}} Z^{a_k} _{A_k]}
  &\;\in& \cE^{a_1 a_2 \cdots a_{k-1} a_k} _{[A_1 A_2 \cdots A_{k-1} A_k]} [-k] \\[0.5em]
  \mathbb{W} _{A_1 A_2 A_3 \cdots A_{k-1} A_k} ^{\phantom{A_1 A_2} a_3 \cdots a_{k-1} a_k}
  &:=& X_{[A_1} Y_{A_2} Z^{a_3} _{A_3} \cdots Z^{a_{k-1}} _{A_{k-1}} Z^{a_k} _{A_k]} 
  &\;\in&\cE^{\phantom{A_1 A_2} a_3 \cdots a_{k-1} a_k} _{[A_1 A_2 \cdots A_{k-1} A_k]} [-k+2] \\[0.5em]
  \mathbb{X} _{A_1 A_2 \cdots A_{k-1} A_k} ^{\phantom{A_1} a_2 \cdots a_{k-1} a_k} 
  &:= & X_{[A_1} Z^{a_2} _{A_2} \cdots Z^{a_{k-1}} _{A_{k-1}} Z^{a_k} _{A_k]} 
  &\;\in &\cE^{\phantom{A_1} a_2 \cdots a_{k-1} a_k} _{[A_1 A_2 \cdots A_{k-1} A_k]} [-k+2] \\
\end{matrix}
\end{equation}
For example, $ \mathbb{Y} _{A_1 A_2 \cdots A_{k-1} A_k} ^{\phantom{A_1}
  a_2 \cdots a_{k-1} a_k}$ gives the injection
$$
\mathbb{Y} _{A_1 A_2 \cdots A_{k-1} A_k} ^{\phantom{A_1} a_2 \cdots a_{k-1} a_k} : \cE_{[a_2 \cdots a_k]} [k] \to  \cE_{[A_1 A_2 \cdots A_{k-1} A_k]},
$$
determined by a choice of metric $g\in \cc$.
Similarly $ \mathbb{X} _{A_1 A_2 \cdots A_{k-1} A_k} ^{\phantom{A_1} a_2 \cdots a_{k-1} a_k} $
\begin{equation}\label{Xinj}
 \mathbb{X} _{A_1 A_2 \cdots A_{k-1} A_k} ^{\phantom{A_1} a_2 \cdots a_{k-1} a_k}: \cE_{[a_2 \cdots a_k]} [k-2]   \to  \cE_{[A_1 A_2 \cdots A_{k-1} A_k]},
 \end{equation}
 but in this case the map is not dependent on any choice of metric in $\cc$. For $1 \leq k \leq n+2$ and a choice of scale, one has  
\begin{equation}\label{eqn_tractor_form_projector_derivatives}
  \begin{split}
    \nabla_b \mathbb{Y}_{A_1 A_2 A_3 \cdots A_k} ^{\phantom{A_1} a_2 a_3 \cdots a_k} 
    &= P_{b a_1} \mathbb{Z}^{a_1 a_2 a_3 \cdots a_k} _{A_1 A_2 A_3 \cdots A_k} + (k-1) 
    P_b {}^{a_2} \mathbb{W}^{\phantom{A_1 A_2} a_3 \cdots a_k} _{A_1 A_2 A_3 \cdots A_k} \\[0.5em]
    \nabla_b \mathbb{Z}^{a_1 a_2 \cdots a_k} _{A_1 A_2 \cdots A_k}
    &= -k \cdot P_b {}^{a_1} \mathbb{X}^{\phantom{A_1} a_2 \cdots a_k} _{A_1 A_2 \cdots A_k} - k \cdot \delta_b {}^{a_1} \mathbb{Y}^{\phantom{A_1} a_2 \cdots a_k} _{A_1 A_2 \cdots A_k} \\[0.5em]
    \nabla_b \mathbb{W}^{\phantom{A_1 A_2} a_3 \cdots a_k} _{A_1 A_2 A_3 \cdots A_k}
    &= -\bg_{b a_2} \mathbb{Y}^{\phantom{A_1} a_2 \cdots a_k} _{A_1 A_2 \cdots A_k} + P_{b a_2} \mathbb{X}^{\phantom{A_1} a_2 \cdots a_k} _{A_1 A_2 \cdots A_k} \\[0.5em]
    \nabla_b \mathbb{X}^{\phantom{A_1} a_2 a_3 \cdots a_k} _{A_1 A_2 A_3 \cdots A_k} 
    &= \bg_{b a_1} \mathbb{Z}^{a_1 a_2 a_3 \cdots a_k} _{A_1 A_2 A_3 \cdots A_k} - (k-1) \delta_b {}^{a_2} \mathbb{W}^{\phantom{A_1 A_2} a_3 \cdots a_k} _{A_1 A_2 A_3 \cdots A_k},
  \end{split}
\end{equation} 
where sequentially labeled indices are alternating, and any term involving the alternation of $n+1$ or more tensor (i.e. lower case) indices should be interpreted as zero.

In particular, we observe that $\mathbb{W}^{\phantom{A_1 A_2} a_3 \cdots a_{n+2}} _{A_1 A_2 A_3 \cdots A_{n+2}}$ is parallel in any scale, and hence there is a distinguished parallel section of the top exterior power of the standard tractor bundle, which we term the \emph{tractor volume form} 
\begin{equation} \label{tvf}
  \tvol_{A_1 A_2 A_3 \cdots A_{n+2}} := (n+2)(n+1)\be_{a_3 \cdots a_{n+2}} \mathbb{W}_{A_1 A_2 A_3 \cdots A_{n+2}} ^{\phantom{A_1 A_2} a_3 \cdots a_{n+2}},
\end{equation}
where $\be_{a_3 \cdots a_{n+2}} \in \cE_{[a_3 \cdots a_{n+2}]} [n]$ is
the weighted volume form of Section~\ref{sec_conf_geo}  (note that our normalisation is such that $\tvol^{A_1 A_2 A_3 \cdots A_{n+2}} \tvol_{A_1 A_2 A_3 \cdots A_{n+2}} =-(n+2)!$). That this is
parallel now follows from the fact that $\be_{a_3 \cdots a_{n+2}}$ is
parallel for any Levi-Civita connection in the conformal class.  Of
course, the existence of the tractor volume form also reflects the
fact that the conformal tractor connection is equivalent to an
$\SO(n+1,1)$-Cartan connection. 

Finally in this section we need the tractor Hodge-star.
For a tractor $k$-form $\Psi_{A_1\cdots A_k}$ this is 
\begin{equation}\label{tstar}
  \star \hspace{-0.3pt} \Psi_{B_1\cdots B_{n+2-k}}= \frac{1}{k!}\tvol^{A_1\cdots
  A_k}{}_{B_1\cdots B_{n+2-k}}  \Psi_{A_1\cdots A_k} .
\end{equation}
This satisfies $\star \star=-(-1)^{k(n-k)}$, since the tractor metric
has Lorentzian signature. Note also that this  tractor Hodge-star operation commutes with the tractor covariant derivative:
$$
\nabla_a\star \hspace{-2pt} \Psi_{B_1\cdots B_{n+2-k}}:= \frac{1}{k!}\tvol^{A_1\cdots
  A_k}{}_{B_1\cdots B_{n+2-k}} \nabla_a \Psi_{A_1\cdots A_k} .
$$

\section{Submanifold geometry and submanifold tractors}

Given a smooth $n$-manifold $M$, a \emph{submanifold} will mean a
smooth embedding $\iota : \Sigma \to M$ of a smooth $m$-dimensional
manifold $\Sigma$, where $1 \leq m \leq n-1$, and the image has
\emph{codimension} $d:= n-m$.  Typically we will suppress explicit mention
of the embedding map and identify $\Sigma$ with its image
$\iota(\Sigma) \subset M$.  We refer to $M$ as the \emph{ambient}
manifold.

Regarding abstract indices, we adopt the convention that Latin letters from the start of the alphabet ($a,b,c,\ldots$) will denote ambient tensor indices, while indices from later in the alphabet ($i,j,k,\ldots$) will denote submanifold tensor indices.
So, for example, $\cE^a$ is the usual tangent bundle $TM$, $\cE^i$ is the tangent bundle of the submanifold $T\Sigma$, and $\cE_i{}^a$ denotes the bundle
$T^* \Sigma \otimes TM|_{\Sigma} $.
Note that indices alone will not distinguish sections of $TM$ and $TM|_\Sigma$, so $v^a$ could be a section of either $\cE^a$ or a section of $\cE^a |_\Sigma$, where $\cE^a |_\Sigma \to \Sigma$ is the pullback bundle $\iota^* TM$.

Given a submanifold $\iota : \Sigma \to M$, its derivative $T \iota : T \Sigma \to TM$ will be written $\Pi^a _i$ and viewed as a section of $T^* \Sigma\otimes TM |_\Sigma$.
We frequently identify $T\Sigma$ with its image in $TM |_\Sigma$ under this map.
Note that $\Pi^a _i$ also gives the canonical map $\Pi^a _i : \cE_a |_\Sigma \to \cE_i$, which is dual to $T \iota$.
We will temporarily use the notation $T_{M/\Sigma}$ for the normal bundle 
$TM|_\Sigma / T\iota(T\Sigma)$, and $(T_{M/\Sigma})^* \subset T^*M |_\Sigma$ for the conormal bundle.

\subsection{Basic Riemannian submanifold theory}\label{riem_subs}

We now move to the setting of a submanifold $\Sigma$ in a Riemannian
manifold $(M,g)$ (cf.\ discussions in, e.g., \cite{Kobayashi-Nomizu,ONeillRiemannianGeometry}).
In the Riemannian setting, we only require that $\dim M \geq 2$, and
$\Sigma$ satisfies $1 \leq m= \dim \Sigma \leq n-1$.  The exact sequence
defining the normal bundle $T_{M/\Sigma}$ then splits

\begin{equation}\label{eqn_normal_bundle_ses_final}
  \begin{tikzcd}
    0 \arrow[r] & \cE^i \arrow{r}{\Pi^a _i} & \cE^a |_\Sigma \arrow[l,bend left=45,pos=0.45,"\Pi^i _a"] \arrow{r} & T_{M/\Sigma} \arrow{r} & 0, \\
  \end{tikzcd}
\end{equation}
where $\Pi^i _a$ is the orthogonal projection map  $TM|_{\Sigma}\to T\Sigma$.  We may then identify $T_{M/\Sigma}$ with the kernel $N
\Sigma$ of $\Pi^i _a$, via the splitting, and we denote the orthogonal
projection onto this by $\mathrm{N}^b _a : \cE^a |_\Sigma \to N
\Sigma^b$.
The complementary projection is $\Pi^a_b =\delta^a_b-\mathrm{N}^a _b =\Pi^a_i\Pi^i_b$ which is the orthogonal projection onto $T\Sigma$ viewed as submanifold of
$TM|_{\Sigma}$.

The Riemannian metric $g$ on $M$ induces a Riemannian metric
$g_\Sigma$ on $\Sigma$ by restriction, which we call the \emph{induced
  metric}.  We usually omit the explicit reference to $\Sigma$ when
abstract indices are used.  So the induced metric will be denoted by
$g_{ij}$. Note that
\begin{equation}
  g_{ij} = \Pi^a _i \Pi^b _j g_{ab}.
\end{equation}

Next we observe that~\eqref{eqn_normal_bundle_ses_final} can be used to decompose the ambient Levi-Civita connection.
First and most simply, we have the \emph{normal connection} $\nabla^\perp$ which is a connection on the bundle $N \Sigma \to \Sigma$ defined by 

\begin{equation}\label{eqn_normal_connection}
  \nabla^\perp _i \nu^a := \mathrm{N}^a _b \nabla_i \nu^b,
\end{equation}
where $\nabla_i$ denotes the pullback connection of the ambient
Levi-Civita connection (meaning, in this context, its restriction to differentiating
along vectors tangent to $\Sigma$).  Complementary to this, we also have induced a
connection $D_i$ on $T \Sigma \to \Sigma$ defined by

\begin{equation}\label{eqn_intrinsic_levi_civita}
  D_i V^j := \Pi ^j _b \nabla_i \left( \Pi ^b _k V^k \right).
\end{equation}

It is elementary to verify that both~\eqref{eqn_normal_connection}
and~\eqref{eqn_intrinsic_levi_civita} define connections.  Indeed, it
is also straightforward to verify that $D$ is torsion-free and
preserves the induced metric, and so is in fact the Levi-Civita
connection of $(\Sigma, g_\Sigma)$.  The fundamental ingredient of
submanifold calculus, in this setting, is the \emph{Gau{\ss} formula}
which, for a section $V \in \Gamma(T\Sigma|_\Sigma)$, provides the
decomposition of $\nabla_i V^c =\nabla_i (\Pi^c_jV^j)$ into its
tangential and normal parts:

\begin{equation}\label{eqn_tensor_gauss_formula}
  \nabla_i V^c = \Pi^c _j D_i V^j + \II_{ij} {}^c V^j,
\end{equation}
and this defines $\II_{ij} {}^c \in \Gamma(S^2 T^* \Sigma \otimes N \Sigma)$, which is the \emph{second fundamental form} of $\Sigma$ in $(M,g)$.
We also define the \emph{mean curvature}
\begin{equation}
  H^c := \frac{1}{m} g^{ij} \II_{ij} {}^c
\end{equation}
and set $\mathring{\II}_{ij} {}^c := \II_{{(ij)}_0} {}^c$, the trace-free part of the second fundamental form.
Thus one has 
\begin{equation}
  \II_{ij} {}^c = \mathring{\II}_{ij} {}^c + g_{ij} H^c.
\end{equation}

Using~\eqref{eqn_tensor_gauss_formula}, one can derive
\begin{equation} \label{eqn_gauss}
  R_{ijkl} = R^\Sigma _{ijkl} + 2 g_{cd} \II_{l[i} {}^c \II_{j]k} {}^d,
\end{equation}
\begin{equation} \label{eqn_codazzi}
	R_{ij} {}^c {}_k \mathrm{N}^d _c = 2 D_{[i} \II_{j]k} {}^d,
\end{equation}
and
\begin{equation} \label{eqn_ricci}
	R_{ij} {}^a {}_b \mathrm{N}^c _a \mathrm{N}^b _d = R^\perp {}_{ij} {}^c {}_d - 2 g^{kl} \II_{k[i} {}^c \II_{j]ld},
\end{equation}
where $R_{i j k l} := \Pi^a _i \Pi^b _j \Pi^c _k \Pi^d _l R_{abcd}$ is
the curvature of the ambient Levi-Civita connection restricted to
$\Sigma$, $R^\Sigma _{ijkl}$ is the intrinsic Riemann curvature tensor
(i.e. the curvature of the connection $D$), $D$ is the intrinsic
Levi-Civita connection coupled to the normal connection and $R^\perp
{}_{ij} {}^c {}_d$ is the curvature of the normal connection $\nabla_i
^\perp$.  All these formulae are derived by substituting the Gau{\ss}
formula~\eqref{eqn_tensor_gauss_formula} into
equation~\eqref{curvature} which defines the curvature of the pullback
connection $\nabla_i$, as follows.  Using the decomposition $T M|_{\Sigma} =
T\Sigma \oplus N \Sigma$, we may write a section $V^c \in
\Gamma(\cE^c|_{\Sigma})$ as a tuple $(\Pi^c _d V^d, \mathrm{N}^c _d
V^d)$.  Since $\Pi^c _d V^d $ is a tangent vector to $\Sigma$, we
abuse notation slightly and typically write the tuple as $ ( V^k,
\mathrm{N}^c _d V^d)$ where $ V^k=\Pi^k_dV^d$.
Computing the
action of the Riemann curvature $R_{ij} {}^c {}_d$ on such a tuple, we
see that
\begin{align*}
      &\left( \nabla_i \nabla_j - \nabla_j \nabla_i \right)
      \begin{pmatrix}
      V^k \\
      \mathrm{N}^c _d V^d
    \end{pmatrix}\\ 
    &=
    \begin{pmatrix}
      R^\Sigma {}_{ij} {}^k {}_l V^l - 2 D_{[i} \left( \II_{j]} {}^k {}_d V^d \right)- 2 g_{ef} \II^k {}_{[i} {}^f \II_{j]l} {}^e V^l - 2 \II_{[i} {}^k {}_{|e|} \nabla_{j]} ^\perp \left( \mathrm{N}^e _d V^d \right) \\ 
      2 \II_{k [i} {}^c \left( D_{j]} V^k \right) - 2 g^{kl} \II_{k[i} {}^c \II_{j]l d} V^d + 2 \nabla_{[i} ^\perp \left( \II_{j]k} {}^c V^k \right) + R^\perp {}_{ij} {}^c {}_d V^d
    \end{pmatrix} \\
    &= 
    \begin{pmatrix}
      R^\Sigma {}_{ij} {}^k {}_l + 2 g_{cd} \II_{l[i} {}^c \II^k {}_{j]} {}^d & -2 D_{[i} \II_{j]} {}^k {}_d \\ 
      2 D_{[i} \II_{j] l} {}^c & R^\perp _{ij} {}^c {}_d - 2 g^{kl} \II_{k [i} {}^c \II_{j] l d} 
    \end{pmatrix}
    \begin{pmatrix}
      V^l \\
      \mathrm{N}^d _e V^e
    \end{pmatrix},
\end{align*}
and the equations~\eqref{eqn_gauss},~\eqref{eqn_codazzi} and~\eqref{eqn_ricci} all follow from this by simply projecting the appropriate entry of the matrix.

Note that in dimension $m=1$, the trace-free part of the second
fundamental form is zero. Also, in dimension $m=1$
equations~\eqref{eqn_gauss}-\eqref{eqn_ricci} are valid, but trivial
in that in each case both sides are identically zero.

\subsection{Conformal Submanifolds}\label{conf_subm} 

Consider now a submanifold $\Sigma$ satisfying $1 \leq \dim \Sigma \leq n-1$ in a conformal manifold $(M ,\bc)$.
Observe that the conformal structure is sufficient to determine an orthogonal complement of $T \Sigma \subset TM$ and so the splitting of~\eqref{eqn_normal_bundle_ses_final} is in fact conformally invariant.
We carry over to this setting the same notation for the normal projection $\mathrm{N}^a _b$ and the orthogonal projection $\Pi^i _a$.
Since each $g \in \bc$ induces a Riemannian metric on $\Sigma$ by restriction, it follows immediately that $\bc$ induces a conformal structure on $\Sigma$ that we denote $\bc_\Sigma$.
We therefore have intrinsic to $(\Sigma, \bc_\Sigma)$ density bundles $\cE_\Sigma [w]$.
In fact, for any $w \in \R$,
\begin{equation*}
  \cE_\Sigma [w] = \cE[w] |_\Sigma,
\end{equation*}
which can be seen immediately from the interpretation of densities as equivalence classes of a metric and a function, see the discussion of Section~\ref{sec_conf_geo}.
From these observations, it follows immediately that the conformal metric $\bg_\Sigma$ is simply the restriction to $\otimes^2 T \Sigma$ of the ambient conformal metric $\bg$.
Using these, it is straightforward to see that the orthogonal projection may be thought of as a composition $\bg^{-1}_\Sigma \circ \Pi \circ \bg$, meaning
\begin{equation}
  \Pi^i _a = \bg^{ij} \Pi^b _j \bg_{ab},
\end{equation}
where we are raising and lowering indices using the conformal metric (and we omit the subscript \mbox{\scriptsize${\Sigma}$} when the indices imply that we are using $\bg_{\Sigma}$ or its inverse).
Since this involves $\bg$ and $\bg^{-1}$, the resulting section still
has conformal weight zero. When a  metric $g\in \cc$ is chosen we have as usual the
Gau{\ss}, Codazzi and Ricci equations
(\eqref{eqn_gauss},~\eqref{eqn_codazzi} and~\eqref{eqn_ricci}), but it will be convenient to work with weighted versions of these, with the ambient or intrinsic
conformal metrics replacing any instances of their scale-dependent
counterparts. For example the weighted version of the Gau{\ss} equation is
$$
  R_{ijkl} = R^\Sigma _{ijkl} + 2 \bg_{cd} \II_{l[i} {}^c \II_{j]k} {}^d,
$$
where now $ R_{ijkl}$ and $R^\Sigma _{ijkl}$ have weight 2 (and the second fundamental form naturally has conformal weight 0).

Coupling the normal connection~\eqref{eqn_normal_connection} to the
Levi-Civita connection on the bundle $\cE_\Sigma [-1]$ yields a
connection on $N \Sigma [-1]$, which we shall also denote
$\nabla^\perp$.  It is easily verified using~\eqref{vector_trans}
and~\eqref{dens-trans} that this is conformally invariant.  In fact,
more generally, on sections of $N \Sigma [w]$ the transformation law is 
\begin{equation}
	\label{eqn_norm_conn_transform}
	\hat{\nabla}^\perp _i \nu^a = \nabla^\perp _i \nu^a + (w + 1) \Upsilon_i
	\nu^a
\end{equation}
when $\hat{g}=\Omega^2g$ and $\Upsilon_i = \Omega^{-1}\nabla_i\Omega$. The
conformal metric $\bg$ induces a bundle metric on $N \Sigma [-1]$, and
this is preserved by $\nabla^\perp$.

Since the Levi-Civita connection changes under a conformal rescaling, the Gau{\ss} formula is not conformally invariant.
Using~\eqref{vector_trans}, we conclude that under a conformal transformation, 
\begin{equation}\label{2ff_transformation}
  \hat{\II}_{ij} {}^c = \II_{ij} {}^c - \bg_{ij} \mathrm{N}^c _d \Upsilon^d.
\end{equation}

Since this transformation is by pure trace, it follows immediately that $\mathring{\II}_{ij} {}^c$ is conformally invariant:
\begin{equation*}
  \hat{\mathring{\II}}_{ij} {}^c = \mathring{\II}_{ij} {}^c.
\end{equation*}
Thus the transformation~\eqref{2ff_transformation} is entirely due to the transformed mean curvature, whence 
\begin{equation}\label{mean_curvature_transformation}
  \hat{H}^c = H^c - \mathrm{N}^c _d \Upsilon^d.
\end{equation}
Note that the mean curvature is now defined using the conformal metric: $H^c := \frac{1}{m} \bg^{ij} \II_{ij} {}^c$. As a consequence of \eqref{mean_curvature_transformation} we have the following very useful proposition \cite{BransonGover01,GRiemSig,CurryThesis}:
\begin{prop}
Let $\Sigma$ be a submanifold of a conformal manifold $(M,\bc)$. Then any metric $g_{\Sigma}$ in the induced conformal class of metrics on $\Sigma$ can be extended to a metric $g\in \bc$ such that the mean curvature of $\Sigma$ with respect to $g$ vanishes.
\end{prop}
\begin{proof}\label{prop_min_scale}
Let $g_{\Sigma}$ be as in the proposition, and let $g\in \bc$ be any extension of $g_{\Sigma}$. We look for a rescaled metric $\hat{g}$ satisfying the requirements of the proposition. Set $\hat{g}=e^{2\omega}g$ with $\omega$ to be determined. Since we require $\hat{g}|_{\Sigma}=g_{\Sigma}$ we set $\omega=0$ along $\Sigma$. Now, by \eqref{mean_curvature_transformation}, $\hat{H}^c=0$ if, and only if, $H^c = N^c_d \nabla^d \omega$ along $\Sigma$. Since the latter merely amounts to specifying the normal derivatives of $\omega$ along $\Sigma$, such an $\omega$ clearly exists (unique modulo functions that vanish and have vanishing differential along $\Sigma$). This proves the proposition. 
\end{proof}
We refer to a metric $g\in \bc$ such that $H^c=0$ as a \emph{minimal scale} for $\Sigma$. The freedom to work in a minimal scale when computing conformally invariant quantities helps to simplify many calculations. 

\begin{rem}\label{rem_minimal_sc_1d}
In the case of a $1$-dimensional submanifold a minimal scale $g$ is easily seen to be one for which the curve is an unparametrised geodesic, since in this case, after parametrising the curve by arc length, the mean curvature vector can be identified with the acceleration of the curve.  Moreover, in the case where the curve $\gamma$ in $(M,\bc)$ is already parametrised one can take $g_{\Sigma}$ to be the ``metric'' on $\gamma$ corresponding to the parametrisation and extend $g_{\Sigma}$ to $g\in \bc$ as in Proposition \ref{prop_min_scale} to obtain a metric $g$ for which $\gamma$, with its original parametrisation, is a parametrised geodesic.
\end{rem}

\subsection{Submanifold tractors} \label{sub-tr}

As a conformal manifold in its own right, $(\Sigma, \bc_\Sigma)$
possesses its own standard tractor bundle, which we will call the
\emph{intrinsic tractor bundle}, and denote by $\cT \Sigma$, or
$\cE^I$ in abstract indices, carrying over our convention that later
Latin letters will be used for sections of submanifold bundles, with
upper case indices for tractor bundles.  We now wish to relate this to
the corresponding ambient tractor bundle, $\cT M$, which will continue to denote by $\cT$.

As we have already seen in Section~\ref{riem_subs}, the Gau{\ss}
formula plays a central role in the setting of Riemannian submanifold
geometry.  Crucially, the Gau{\ss} formula uses that $T \Sigma$ may be
identified with a subbundle of $TM$.  In fact, there is an analogous
notion for the intrinsic and ambient tractor bundles, and this
explains our abstract index notation for submanifolds being similar to
that for ambient tractors.

First, there is a mapping $N\Sigma [-1] \to \cT$ defined by 
\begin{equation}\label{norm_bundle}
  n^a \mapsto N^A= N^A_a n^a \overset{g}{=}
  \begin{pmatrix}
    0 \\ 
    n^a \\ 
    n_c H^c
  \end{pmatrix}.
\end{equation}
This is easily seen to be conformally invariant using the transformation laws for the tractor projectors~\eqref{tractor_proj_transforms} and the mean curvature~\eqref{mean_curvature_transformation}:
\begin{align*}
  \hat{N}^A & = \hat{n}^a \hat{Z}_a ^A + \hat{n}_a \hat{H}^a \hat{X}^A                      \\
            & = n^a (Z_a ^A + \Upsilon_a X^A) + n_a (H^a - \Upsilon^b \mathrm{N}^a _b ) X^A \\
            & = n^a Z_a ^A + n^a \Upsilon_a X^A + n_a H^a X^A - n_b \Upsilon^b X^A          \\
            & = n^a Z_a ^A + n_a H^a X^A                                                    \\
            & = N^A.
\end{align*}
Thus the image of the injective map~\eqref{norm_bundle} defines, along
$\Sigma$, a subbundle of $\cT$ which is canonically isomorphic to $N
\Sigma [-1]$.  We call this the \emph{normal tractor bundle} and
denote this $\cN$, or $\cN^A$ with indices. We summarise, as follows.
\begin{lem}\label{Nlem}
  The map \eqref{norm_bundle} defines a conformally invariant isomorphism
  \begin{equation}\label{nNiso}
N_a^A: N\Sigma [-1] \stackrel{\simeq}{\longrightarrow} \cN \subset \cT|_\Sigma .
\end{equation}
\end{lem}

The bundle $\cN$ admits an orthogonal complement in $\cT$,
$\cN^\perp$, and so $\cT|_{\Sigma}$ decomposes as
\begin{equation*}
  \cT|_{\Sigma} = \cN^\perp \oplus \cN.
\end{equation*}
Write $\Pi^A _B : \cE^B \to \cN^{\perp A}$ and $\mathrm{N}^A _B :
\cE^B \to \cN^A$ for the orthogonal projections onto the respective
factors of this decomposition.  So $\delta^A _B = \Pi^A _B +
\mathrm{N}^A _B$.  Note that for any $N^B \in \Gamma(\cN^B)$, one has
$h_{AC} N^C \Pi^A _B = 0$ since $\Pi^A _B$ is valued in $\cN^{\perp
  A}$.  Thus, $h_{AC} \mathrm{N}^A _B \Pi^C _D = 0$.  Substituting
$\Pi^C _D = \delta^C _D - \mathrm{N}^C _D$, it follows that
\begin{equation*}
  \mathrm{N}_{BC} = \mathrm{N}^A _B \mathrm{N}^C _D h_{AC}.
\end{equation*}
So $\mathrm{N}_{AB}$ and $\Pi_{AB}$ are symmetric, where in each case
an index has been lowered with the tractor metric, and $\Pi^A _B$ and
$\mathrm{N}^A _B$ give the orthogonal decomposition of the cotractor
bundle $\cE_A$.

Note that 
\begin{equation}\label{fact-or}
  N^A_a N_B^a= N^A_B \quad \mbox{and} \quad
  N_B^aN^B_b = N^a_b ,
  \end{equation}
where $N_A^a$ is the inverse to \eqref{nNiso}. From the symmetry of
$N_{AB}$, and corresponding observation of symmetry for $N_{ab}$, it
follows that $N_A^a$ is obtained from $N^B_b$ by raising and lowering
indices using the tractor and conformal metrics.

A straightforward direct calculation shows that the isomorphism \eqref{nNiso} intertwines the tractor and normal Levi-Civita connections in
the sense of the following lemma.
\begin{lem}\label{nablaNlem}
  For any section $n^a\in \Gamma(N\Sigma [-1])  $ we have
  $$
  N_b^C \nabla_i^{\perp} n^b= N^C_B\nabla_i (N_b^B  n^b) .
  $$
\end{lem}
\begin{proof}
This follows immediately from the definitions if we work in a minimal scale.
\end{proof}
We note here that $\cN$ is a rank $d = n-m$ vector bundle, and hence
$\cN^\perp$ has rank $(n+2)-d = m+2$, which coincides with the rank of
$\cT \Sigma$.  This is not a coincidence, and it turns out that there
is an isomorphism of vector bundles $\cT \Sigma \to \cN^\perp$. Let us initially understand this in submanifold dimensions $m\geq 3$.

\begin{thm}\label{thm_norm_perp_intrinsic_isom}
  Let $\Sigma$ be a submanifold of dimension $m\geq 3$ in a conformal manifold $(M, \bc)$. 
	The intrinsic tractor bundle $\cT \Sigma$ is canonically isomorphic to the orthogonal complement $\cN^\perp$ of the normal tractor bundle via a bundle isomorphism which preserves both the metric and the filtration.
	We denote this isomorphism $\Pi^A _I$.
  Explicitly, in a general ambient scale $g\in \bm{c}$, it is given by
	\begin{equation}\label{eqn_norm_perp_intrinsic_isom}
	\cT\Sigma \ni	V^I
		\overset{g_\Sigma}{=}
		\begin{pmatrix}
			\sigma \\ \mu^i \\ \rho
		\end{pmatrix}
		\overset{\Pi^A _I}{\longmapsto}
		\begin{pmatrix}
			\sigma \\ \mu^a - H^a \sigma \\ \rho - \frac{1}{2} H^a H_a \sigma
		\end{pmatrix}\stackrel{g}{=} V^A\in \cN^\perp ,
	\end{equation}
where  $\mu^a = \Pi^a _i \mu^i$.

The map $\Pi: \cT\Sigma\to \cN^\perp$ is a filtration and metric preserving  isomorphism. 
\end{thm}

\begin{proof}
  Fix a scale $g_\Sigma \in \bm{c}_\Sigma$, and let $g \in \bm{c}$ be
  a scale that satisfies $\iota^* g = g_\Sigma$.  The map is clearly
  injective, and the image is also clearly annihilated by any section
  of $\cN$.  We need to show that the
  map~\eqref{eqn_norm_perp_intrinsic_isom} is unchanged if we replace
  $g$ by some conformally related $\hat{g} = \Omega^2 g$, and
  $g_\Sigma$ by $\hat{g_\Sigma} = \Omega^2 g_\Sigma$. (In the latter
  $\Omega$ is restricted to $\Sigma$ -- this is clear by context and
  so we do introduce additional notation.)  Equivalently, we need to show
  that the following diagram commutes
  \begin{equation}\label{eqn_intrinsic_tractor_comm_diagram}
    \begin{tikzcd}
      {[\cE^I]}_{g_\Sigma} \arrow{r}{\Pi^A _I} \arrow[d] &  {[\cE^A]}_{g} |_\Sigma \arrow[d] \\
      {[\cE^J]}_{\hat{g_\Sigma}} \arrow{r}[swap]{\Pi^B _J} &  {[\cE^B]}_{\hat{g}} |_\Sigma \\
    \end{tikzcd}
  \end{equation}
  where the vertical maps are conformal change of tractor splitting, as given in (\ref{3-trans}), and the horizontal maps are~\eqref{eqn_norm_perp_intrinsic_isom} in the appropriate scale.

  Write $\Upsilon_a = \Omega^{-1} \nabla_a \Omega$, and $\Upsilon_i = \Omega^{-1} D_i \Omega$.
  Note that $\Upsilon_i = \Pi^a _i \Upsilon_a$.
  Applying $\Pi^A _I$ and then rescaling is given by 
  \begin{equation*}
    \begin{pmatrix}
      1 & 0 & 0 \\
      \Upsilon^b & \delta^b _a & 0 \\
      -\frac{1}{2} \Upsilon^c \Upsilon_c & - \Upsilon_a & 1
    \end{pmatrix} 
    \begin{pmatrix}
      1 & 0 & 0 \\
      -H^a & \Pi^a _i & 0 \\
      -\frac{1}{2} H^c H_c & 0 & 1 
    \end{pmatrix}
    =
    \begin{pmatrix}
      1 & 0 & 0 \\
      \Upsilon^b - H^b & \Pi^b _i & 0 \\
      -\frac{1}{2} \Upsilon^c \Upsilon_c + H^a \Upsilon_a - \frac{1}{2} H^c H_c & -\Upsilon_i & 1
    \end{pmatrix},
  \end{equation*}
  while first rescaling and then applying $\Pi^B _J$ corresponds to the matrix 
  \begin{equation*}
    \begin{pmatrix}
      1 & 0 & 0 \\
      -\hat{H}^b & \Pi^b _j & 0 \\
      -\frac{1}{2} \hat{H}^c  \hat{H}_c & 0 & 1 
    \end{pmatrix}
    \begin{pmatrix}
      1 & 0 & 0 \\
      \Upsilon^j & \delta^j _i & 0 \\
      -\frac{1}{2} \Upsilon^k \Upsilon_k & - \Upsilon_i & 1
    \end{pmatrix}
    =
    \begin{pmatrix}
      1 & 0 & 0 \\
      -\hat{H}^b + \Pi^b _j \Upsilon^j & \Pi^b _i & 0 \\
      -\frac{1}{2} \hat{H}^c \hat{H}_c -\frac{1}{2} \Upsilon^k \Upsilon_k & -\Upsilon_i & 1
    \end{pmatrix}.
  \end{equation*}
  Using equation~\eqref{mean_curvature_transformation}, we see that 
  \begin{equation*}
    -\hat{H}^b + \Pi^b _j \Upsilon^j = - H^b + \mathrm{N}^b _c \Upsilon^c + \Pi^b _j \Upsilon^j = - H^b + \Upsilon^b
  \end{equation*}
  and 
  \begin{align*}
    -\frac{1}{2} \hat{H}^c \hat{H}_c -\frac{1}{2} \Upsilon^k \Upsilon_k
    &= -\frac{1}{2} H^c H_c + H^e \Upsilon_e - \frac{1}{2} \left( \Upsilon^k \Upsilon_k + \mathrm{N}^{cd} \Upsilon_c \Upsilon_d \right) \\ 
    &= -\frac{1}{2} H^c H_c + H^e \Upsilon_e - \frac{1}{2} \Upsilon^c \Upsilon_c,
  \end{align*}
  whence the above two matrix products are equal.
  Hence the map is conformally invariant.
 
 It is easily verified that the map $\Pi^A_I$ is metric preserving and
 sends $X^I$ to $X^A$, and so is filtration preserving. An easy calculation shows that it is metric preserving, cf.\ Remark \ref{min-sc} below.
  
\end{proof}

\begin{rem} \label{min-sc} 
Note that the calculations in the above proof and hence the existence of the canonical metric and filtration preserving map  $\Pi: \cT\Sigma\to \cN^\perp$ are greatly simplified if we choose to work only with minimal scales ($g\in \bc$ with $H^c=0$, cf.\ Proposition \ref{prop_min_scale}); in a minimal scale, the map $\Pi^A _I$ simply maps $(\sigma, \mu_i, \rho)
  \mapsto (\sigma, \mu_a, \rho)$ and it is clear that this map
  preserves the metric and the filtration.  
\end{rem}

Motivated by the result in Theorem \ref{thm_norm_perp_intrinsic_isom} above, for the cases of
dimensions $m=1,2$ we (for now) define $\cT\Sigma$ to be
$\cN^\perp$. (Then Theorem \ref{thm_norm_perp_intrinsic_isom} again applies,  and is effectively just
changing the splitting to give triples that transform in the usual
way.) In Section \ref{lowdim} below we will then show that in  dimension
$m=1$ and $m=2$ it is still that case that $\cT\Sigma$ is canonically
$J^2\ce[1]/S^2_0 T^*\Sigma [1]$ (where in the $m=1$ case $S^2_0 T^*\Sigma [1]$ is the ``zero vector bundle''), consistent with the discussion given in the introduction.

For convenience, we will say that sections of $\cN^\perp$ are tractors \emph{tangent} to the submanifold, and similarly, sections of $\cN$ are tractors \emph{normal} to the submanifold.

It will also be convenient to record the relationship between the submanifold and ambient splitting tractors corresponding to the isomorphism in Theorem \ref{thm_norm_perp_intrinsic_isom}, namely:
\begin{equation}\label{subm_vs_amb_XYZ}
X^I =  \Pi^I_A X^A, \quad Z^I_i = \Pi^I_A \Pi^a_i Z^A_a, \quad \text{and} \quad Y^I = \Pi^I_A(Y^A - H^aZ_a^A -\frac{1}{2}H^aH_a X^A)
\end{equation}
along $\Sigma$, where $\Pi^{I}_A$ can be interpreted as the inverse of the map $\Pi^A_I$ given by Theorem \ref{thm_norm_perp_intrinsic_isom} or, better, as the orthogonal projection from $\cT$ to $\mathcal{N}^{\perp}$ followed by the isomorphism $\mathcal{N}^{\perp}\to \cT\Sigma$ (this is completely analogous to our use of the notation $\Pi^i_a$ applied to tangent vectors, see \eqref{eqn_normal_bundle_ses_final}).

We have already mentioned that there is a tractor Gau{\ss} formula,
namely a decomposition of the ambient tractor connection which is
compatible with the decomposition $\cT = \cN^\perp \oplus \cN$.
Define the ``checked'' connection $\check{\nabla}$ on $\cT\Sigma$ by
\begin{equation}\label{checked_connection}
  \check{\nabla}_i V^J := \Pi^J _B \nabla_i \left( \Pi^B _K V^K \right),
\end{equation}
where $\nabla_i$ on the right-hand side is the (pullback of the) ambient tractor connection.
This is essentially the tangential part of the ambient connection.
We may then define the \emph{tractor second fundamental form} $\LL_{i J} {}^C$ analogously to the Riemannian case, namely as the 1-form with values in maps $\cT \Sigma \to \cN$ which characterises the normal part of the ambient connection:
\begin{equation}\label{tractor_gauss}
  \nabla_i V^B = \Pi^B _J \check{\nabla}_i V^J + \LL_{i J} {}^B V^J,
\end{equation}
where $V^B = \Pi^B _J V^J$ is a section of the ambient tractor bundle which is tangent to the submanifold.
The linear operator $\LL_{i J} {}^B$ is well defined by this since both $\iota^* \nabla$ and $\check{\nabla}$ satisfy the Leibniz rule.
We call~\eqref{tractor_gauss} the \emph{tractor Gau{\ss} formula}.

The ambient tractor connection also induces a connection on the normal tractor bundle in the obvious way:
\begin{equation}\label{normal_tractor_connection}
  \nabla^\cN _i N^A := \mathrm{N}^A _B \nabla_i N^B,
\end{equation}
where $N^A$ is a section of $\cN$.
Such $N^A$ are of the form $N^A\stackrel{g}{=} (0, n^a, n_c H^c)$, where $n^a \in \Gamma(N\Sigma^a [-1])$.
As a tractor with zero in the top slot, it follows from~\eqref{tractor_proj_transforms} that the middle slot of $\nabla_i^\cN N^A$ is necessarily conformally invariant. 
But this exactly recovers the invariant connection on $N \Sigma[-1]$ discussed in Section~\ref{conf_subm} (cf.\ Lemma \ref{nablaNlem}). In summary, we have the following.
\begin{prop}\label{norm-tr-conn}
The canonical isomorphism $N\Sigma[-1]\stackrel{\simeq}{\to} \cN$
preserves the invariant parallel transports defined on each bundle.
  \end{prop}

Essential to our direction in this article is that the tractor
fundamental form may be captured in several equivalent ways, the first
of which we give here.

\begin{prop}\label{prop_alt_tractor_2ff}
	The tractor second fundamental form is given by
  \begin{equation}\label{eqn_tractor_2ff_der_proj}
		\LL_{i K} {}^B = \Pi^C _K \mathrm{N}_A ^B \nabla_i \Pi_C ^A,
	\end{equation}
  or equivalently,
  \begin{equation}\label{n_grad_n}
    \LL_{i K} {}^B = - \Pi^C _K \mathrm{N}^B _A \nabla_i \mathrm{N}^A _C.
  \end{equation}
\end{prop}

\begin{proof}
	Let $N^A$ be a section of the normal tractor bundle $\cN$.
	Note that $\Pi^A _C N_A = 0$, and hence

	\begin{equation*}
		0 = \nabla_i(\Pi^A _C N_A) = (\nabla_i \Pi^A _C) N_A + \Pi^A _B \nabla_i N_A,
	\end{equation*}
        whence	
  \begin{equation}\label{eqn_tractor_2ff_proj_intermediate}
    \Pi^A _C \nabla_i N_A = - N_A \nabla_i \Pi^A _C.  
	\end{equation}
  As a consequence of the tractor Gau{\ss} formula~\eqref{tractor_gauss},
  \begin{equation*}
    N_B \mathbb{L}_{iK} {}^B V^K = N_B \nabla_i V^B = -V^B \nabla_i N_B = - V^K \Pi ^B _K \nabla_i N_B
  \end{equation*}
  for all $V^K \in \Gamma(\cE^K)$, and therefore
  \begin{equation*}
    N_B \LL _{iK} {}^B = - \Pi^B _K \nabla_i N_B = - \Pi^C _K \Pi^A _C \nabla_i N_A.
  \end{equation*}
Combining this with~\eqref{eqn_tractor_2ff_proj_intermediate}, we have that
	\begin{equation*}
    N_B \LL _{i K} {}^B = -\Pi^C _K \left( - N_A \nabla_i \Pi^A _C \right) = N_B \Pi^C _K \mathrm{N}^B _A \nabla_i \Pi^A _C,
	\end{equation*}
	and this must hold for any section $N^B$ of the normal tractor bundle, whence the result follows.
  Substituting $\Pi^A _C = \delta^A _C - \mathrm{N}^A _C$ into~\eqref{eqn_tractor_2ff_der_proj} and using that $\delta^A _C$ is parallel for the tractor connection then gives the second equality of the proposition.
\end{proof}

\medskip

It will be convenient to have a second (equivalent) object that we also term the \emph{tractor second fundamental form}, which we denote by $\overline{\LL}$ and which is the section of $T^* \Sigma \otimes \cN^{\perp *} \otimes \cN$ defined by 
\begin{equation}\label{LLbar}
  \overline{\LL}_{i A} {}^B = \Pi^J _A \LL_{i J} {}^B.
\end{equation}
Since $\Pi^J _A$ is an isomorphism $\cT^* \Sigma \to \cN^{\perp *}$, this is clearly equivalent to the original tractor second fundamental form (and eventually when there is no possibility of confusion we will simply denote both objects by $\LL$).
The gain of using $\overline{\LL}$ is that both its tractor indices are ambient tractor indices, as we shall see shortly. 

It is useful to observe that $\overline{\LL}$ arises naturally in several different ways.
\begin{lem}
  The $B$ index of $\mathrm{N}_A ^C \nabla_i \mathrm{N}^A _B$ is tangential, i.e., for any $N^B \in \Gamma(\cN^B)$, one has 
  \begin{equation*}
    N^B \mathrm{N}_A ^C \nabla_i \mathrm{N}^A _B = 0.
  \end{equation*}
\end{lem}

\begin{proof}
Let $N^B \in \Gamma(\cN^B)$. We calculate $\nabla_i (\mathrm{N}^A _B N^B)$ in two different ways.
  On the one hand, one has 
  \begin{equation*}
    \nabla_i ( \mathrm{N}^A _B N^B) = \nabla_i N^A,
  \end{equation*}
  while on the other
  \begin{equation*}
    \nabla_i ( \mathrm{N}^A _B N^B ) = \mathrm{N}^A _B \nabla_i N^B + N^B \nabla_i \mathrm{N}^A _B.
  \end{equation*}
  Hence 
  \begin{equation*}
    N^B \nabla_i \mathrm{N}^A _B = \nabla_i N^A - \mathrm{N}^A _B \nabla_i N^B.
  \end{equation*}
  Thus 
  \begin{equation*}
    N^B \mathrm{N}_A ^C \nabla_i \mathrm{N}^A _B
    = \mathrm{N}^C _A ( \nabla_i N^A - \mathrm{N}^A _B \nabla_i N^B ) 
    = \mathrm{N}^C _A \nabla_i N^A - \mathrm{N}^C _B \nabla_i N^B 
    = 0.
  \end{equation*}
\end{proof}

Thus we see that the $\Pi^C _K$ of equation~\eqref{n_grad_n} is merely identifying the already tangential $C$ index with a submanifold tractor index.
Thus we see the following.

\begin{prop} \label{NgradN=L}
  \begin{equation}\label{l_bar_n_grad_n}
    \overline{\LL}_{i B} {}^C= - \mathrm{N}^C _A \nabla_i \mathrm{N}^A _B. 
  \end{equation}
\end{prop}

\begin{proof}
  By equation (\ref{n_grad_n}) we have
\begin{equation*}
  \overline{\LL}_{i B} {}^C
  = - \Pi^K _B \Pi^D _K \mathrm{N}^C _A \nabla_i \mathrm{N}^A _D 
  = - ( \delta^D _B - \mathrm{N}^D _B ) \mathrm{N}^C _A \nabla_i \mathrm{N}^A _D,
\end{equation*}
and so the result follows from the previous lemma.
\end{proof}

\begin{rem}\label{analog}
  Note that \eqref{l_bar_n_grad_n} is equivalent  to
  $
\LL_{i J} {}^C= - \Pi^B_J\mathrm{N}^C _A \nabla_i \mathrm{N}^A _B .
  $
  A similar argument  shows that
  $$\II_{ij} {}^c=-\Pi^b _j
  \mathrm{N}^c _a \nabla_i \mathrm{N}^a _b.
  $$
  \end{rem}

\begin{lem}\label{lem_der_normal_tractor_proj}
  Let $\mathrm{N}^C _B$ be the normal tractor projector. 
  Then 
  \begin{equation}\label{eqn_der_normal_tractor_proj}
    \nabla_i \mathrm{N}^C _B = - \overline{\LL}_i {}^C {}_B - \overline{\LL}_{iB} {}^C,
  \end{equation}
  where 
  \begin{equation*}
    \overline{\LL}_i {}^C {}_B = h^{CD} h_{AB} \overline{\LL}_{i D} {}^A.
  \end{equation*}
\end{lem}

\begin{proof}
  Noting that $\mathrm{N}^C _B = \mathrm{N}^C _A \mathrm{N}^A _B$, we have 
  \begin{equation*}
    \nabla_i \mathrm{N}^C _B = \mathrm{N}^A _B \nabla_i \mathrm{N}^C _A + \mathrm{N}^C _A \nabla_i \mathrm{N}^A _B.
  \end{equation*}
  The second term is exactly the negative of equation~\eqref{l_bar_n_grad_n}.
  Using that the normal projector is symmetric, the first term is clearly a transpose of this (so for example on this term, the $C$ index is tangential).
\end{proof}

We now use this result to compute an explicit formula for the tractor second fundamental form. 
We first prove a lemma about the tractor normal projector.
\begin{lem}\label{lem_norm_tractor_proj_formula}
  For a choice of scale, the tractor normal projector is given by 
  \begin{equation}\label{eqn_norm_tractor_proj_formula}
    \mathrm{N}^A _B = \mathrm{N}^a _b Z^A _a Z^b _B + H^a Z^A _a X_B + H_b X^A Z^b _B + (H^d H_d) X^A X_B
  \end{equation}
  where the $H^c$ is the mean curvature vector in the chosen scale.
\end{lem}

\begin{proof}
    The right-hand side of~\eqref{eqn_norm_tractor_proj_formula}
    defines a conformally invariant bundle map $\cE^A \to \cN^B$ which
    moreover acts as the identity on sections of $\cN^A$ as defined in
    following (\ref{norm_bundle}). The latter is easily verified by working in
    a minimal scale.

    Alternatively, \eqref{eqn_norm_tractor_proj_formula} follows immediately from \eqref{fact-or}, as in any scale
    $$
N^a_B=N^a_bZ^b_B+H^aX_B ,
$$
and $\mathrm{N}^A _B= N^A_aN^a_B$.
\end{proof}

\begin{thm}\label{tractor2ff_XZ}
  The tractor second fundamental form is given by 
  \begin{equation}\label{eqn_tractor2ff_XZ}
    \begin{split}
      \LL_{iJ} {}^C 
    &= \mathring{\II}_{ij} {}^c Z^j _J Z^C _c  
    + \mathrm{N}^c _a \left( P_i {}^a - \nabla_i H^a \right) X_J Z^C _c \\
    &\hspace{2em} + H_c \mathring{\II}_{ij} {}^c Z^j _J X^C 
    + H_a\left( P_i {}^a - \nabla_i H^a \right) X_J X^C .
    \end{split}
  \end{equation}
\end{thm}
\begin{rem} Note that, in particular, $\LL=0$ if and only if $\mathring{\II}=0$ and $\mathrm{N}^c _a \left( P_i {}^a - \nabla_i H^a \right) =0$. For later use when computing examples it is useful to note that also that 
$\mathrm{N}^c _a \left( P_i {}^a - \nabla_i H^a \right) = \frac{1}{n-2}\mathrm{N}^c _a \mathrm{Ric}_i{}^a - \nabla_i^{\perp} H^c$.
\end{rem}

\begin{proof}
  We compute $\mathrm{N}^C _A \nabla_i \mathrm{N}^A _B$ using the formula from Lemma~\ref{lem_norm_tractor_proj_formula}.
  We then apply $\Pi^B _J$, the formula for which is given in Theorem~\ref{thm_norm_perp_intrinsic_isom} to complete the proof.

  First, differentiating~\eqref{eqn_norm_tractor_proj_formula} gives 
  \begin{align*}
    \nabla_i \mathrm{N}^A _B 
    &= (\nabla_i \mathrm{N}^a _b ) Z^A _a Z^b _B + \mathrm{N}^a _b \left( -P_{ia} X^A - \confmet_{ia} Y^A \right) Z^b _B + \mathrm{N}^a _b Z^A _a \left( - P_i ^{b} X_B - \Pi_i ^b Y_B \right) \\
    &\hspace{2em}+ (\nabla_i H^a) Z^A _a X_B + H^a \left( -P_{ia} X^A - \confmet_{ia} Y^A \right) X_B + H^a Z^A _a Z_{B i} \\
    &\hspace{2em}+ (\nabla_i H_b) X^A Z^b _B + H_b Z^A _i Z^b _B + H_b X^A \left( -P_i {}^b X_B - \Pi_i ^b Y_B \right) \\
    &\hspace{2em}+ 2 (H^d \nabla_i H_d) X^A X_B + H^d H_d Z^A _i X_B + H^d H_d X^A Z_{Bi} \\
    &= \left( \nabla_i \mathrm{N}^a _b + H^a \confmet_{ib} + H_b \Pi^a _i \right) Z^A _a Z^b _B \\
    &\hspace{2em}+ \left( -\mathrm{N}^a _b P_{ia} + \nabla_i H_b + H_d H^d \confmet_{ib} \right) X^A Z^b _B \\
    &\hspace{2em}+ \left( - \mathrm{N}^a _b P_i {}^b + \nabla_i H^a + H_d H^d \Pi_i ^a \right) Z^A _a X_B \\
    &\hspace{2em}+ \left( -H^a P_{ia} - H_b P_i {}^b + 2H^d \nabla_i H_d \right) X^A X_B, 
  \end{align*}
  where, recall, $\confmet_{ia}$ means $\Pi^b_i\confmet_{ba}$.
  From~\eqref{eqn_norm_tractor_proj_formula}, it follows that 
  \begin{equation*}
    \mathrm{N}^C _A Z^A _a = \mathrm{N}^c _a Z^C _c + H_a X^C
    \hspace{1em} \textrm{and} \hspace{1em} 
    \mathrm{N}^C _A X^A = 0.
  \end{equation*}
  Hence
  \begin{align*}
    \mathrm{N}^C _A \nabla_i \mathrm{N}^A _B 
    &= \left( \nabla_i \mathrm{N}^a _b + H^a \confmet_{ib} + H_b \Pi^a _i \right) \left( \mathrm{N}^c _a Z^C _c + H_a X^C \right) Z^b _B \\
    &\hspace{2em}+ \left( - \mathrm{N}^a _b P_i {}^b + \nabla_i H^a + H_d H^d \Pi_i ^a \right) \left( \mathrm{N}^c _a Z^C _c + H_a X^C \right) X_B \\
    &= \left( \mathrm{N}^c _a \nabla_i \mathrm{N}^a _b + H^c \confmet_{ib} \right) Z^b _B Z^C _c 
    + \mathrm{N}^c _a \left( \nabla_i H^a - P_i {}^a \right) X_B Z^C _c \\
    &\hspace{2em}+ H_a \left( \nabla_i \mathrm{N}^a _b + H^a \confmet_{ib} \right) Z^b _B X^C 
    + H_a \left( \nabla_i H^a - P_i {}^a \right) X_B X^C.
  \end{align*}
  All that remains is to apply the tangential tractor projector $\Pi^B _J$.
  According to~\eqref{eqn_norm_perp_intrinsic_isom},
  \begin{equation}\label{pi_X_Z}
    \Pi^B _J Z^b _B = \Pi^b _j Z^j _J 
    \hspace{1em} \textrm{and} \hspace{1em} 
    \Pi^B _J X_B = X_J.
  \end{equation}
Therefore 
    \begin{align*}
    \Pi^B _J \mathrm{N}^C _A \nabla_i \mathrm{N}^A _B 
    &= \Pi^b _j \left( \mathrm{N}^c _a \nabla_i \mathrm{N}^a _b + H^c \confmet_{ib} \right) Z^j _J Z^C _c 
    + \mathrm{N}^c _a \left( \nabla_i H^a - P_i {}^a \right) X_J Z^C _c \\
    &\hspace{2em}+ H_a \Pi^b _j \left( \nabla_i \mathrm{N}^a _b + H^a \confmet_{ib} \right) Z^j _J X^C 
    + H_a \left( \nabla_i H^a - P_i {}^a \right) X_J X^C \\
    &= \left( - \II_{ij} {}^c + H^c \confmet_{ij} \right) Z^j _J Z^C _c 
    + \mathrm{N}^c _a \left( \nabla_i H^a - P_i {}^a \right) X_J Z^C _c \\
    &\hspace{2em}+ H_c \left( -\II_{ij} {}^c + H^c \confmet_{ib} \right) Z^j _J X^C 
    + H_a \left( \nabla_i H^a - P_i {}^a \right) X_J X^C \\
    &= - \mathring{\II}_{ij} {}^c Z^j _J Z^C _c 
    + \mathrm{N}^c _a \left( \nabla_i H^a - P_i {}^a \right) X_J Z^C _c \\
    &\hspace{2em}- H_c \mathring{\II}_{ij} {}^c Z^j _J X^C 
    + H_a \left( \nabla_i H^a - P_i {}^a \right) X_J X^C. \\
  \end{align*}
 where we note that $H_a \nabla_i \mathrm{N}^a _b = H_c \mathrm{N}^c
  _a \nabla_i \mathrm{N}^a _b$, and we have used the observation from
  Remark~\ref{analog} to replace $\Pi^b _j
  \mathrm{N}^c _a \nabla_i \mathrm{N}^a _b$ with $-\II_{ij} {}^c$.
  
Finally, applying $\Pi^B_J$ to equation~\eqref{eqn_der_normal_tractor_proj} shows that $\LL_{i
  J} {}^C$ is equal to negative of the above, from which the claim in the theorem follows.
\end{proof}

\subsection{Conformal submanifolds of dimension \texorpdfstring{$m\geq 3$}{m>2}} \label{geq3}

For submanifold dimensions $m\geq 3$ the conformal structure of $\Sigma$ 
determines a tractor connection $D$  and compatible tractor metric $h
_\Sigma$ on  $\cT \Sigma$. We refer to these as the {\em intrinsic} tractor connection and metric for   $\cT \Sigma$. We continue to use abstract indices from the later part of the alphabet to distinguish submanifold objects from their ambient analogues, so, e.g., we write $h
_\Sigma$ as  
$h_{IJ} \in \Gamma(\cE_{(IJ)})$.

Unlike the Riemannian case, the checked connection is not exactly the
intrinsic tractor connection.
\begin{prop}\label{SandF}
  Along a submanifold $\Sigma$ of dimension $m\geq 3$ the checked and intrinsic tractor connections are related by
  \begin{equation}\label{diff_tractor_connection}
  \check{\nabla}_i V^J = D_i V^J + \mathsf{S}_i {}^J {}_K V^K, 
\end{equation}
with $D_i$ the intrinsic submanifold tractor connection and 
\begin{equation}\label{difference_tractor}
  \mathsf{S}_{i J K} := 2 \left( P_{ij} - p_{ij} + H_c \mathring{\II}_{ij} {}^c + \frac{1}{2} H_c H^c \bg_{ij} \right) Z^j _{[J} X_{K]},
\end{equation}
where $\Pi_{ij} := \Pi^a _i \Pi^b _j P_{ab}$ is the restriction of the ambient Schouten tensor to the submanifold and $p_{ij}$ is the intrinsic Schouten tensor.
\end{prop}
\begin{proof}Fix metrics $g \in \bc$ and $g_\Sigma \in \bc_\Sigma$ such that $\iota^* g =
g_\Sigma$ to facilitate calculation.
The inverse isomorphism of~\eqref{eqn_norm_perp_intrinsic_isom} is the map
$\cN^\perp \to \cT \Sigma$ given by the matrix
\begin{equation}
	\label{eqn_inverse_intrinsic_norm_perp}
	\begin{pmatrix}
		1 & 0 &
		0       \\ -H^a & \Pi^a _i & 0 \\ -\frac{1}{2} H^c H_c & 0 & 1
	\end{pmatrix}
	.
\end{equation}
Using this we have
\begin{align*}
	\check{\nabla}_i V^J & =
	                         \begin{pmatrix}
		                         1                    & 0    & 0 \\ 0 & \Pi^j _b & 0 \\
		                         -\frac{1}{2} H^c H_c & -H_b & 1
	                         \end{pmatrix}
	\nabla_i \left[
		                \begin{pmatrix}
			                1
			                 & 0 & 0 & \\ -H^b & \Pi^b _k & 0 \\ -\frac{1}{2} H^c H_c & 0 & 1
		                \end{pmatrix}
		\begin{pmatrix}
			\sigma \\ \mu^k \\ \rho
		\end{pmatrix}
	\right]                                                                                                                      \\  & =
	   \begin{pmatrix}
		1 & 0 & 0 \\ 0 & \Pi^j _b & 0 \\ -\frac{1}{2} H^c H_c & -H_b &
		   1
	\end{pmatrix}
	   \nabla_i
	         \begin{pmatrix}
		         \sigma \\ \mu^b - H^b \sigma \\ \rho -
		         \frac{1}{2} H^c H_c \sigma
	         \end{pmatrix}
	\\  & =
	\begin{pmatrix}
		1        & 0 & 0 \\ 0        &
		\Pi^j _b & 0     \\ -\frac{1}{2} H^c H_c & -H_b & 1
	\end{pmatrix}
	   \begin{pmatrix}
		\nabla_i \sigma - \mu_i                                                 \\ \nabla_i \left( \mu^b - H^b \sigma \right) + P_i
		{}^b \sigma + \Pi_i ^b \left( \rho - \frac{1}{2} H^c H_c \sigma \right) \\
		\nabla_i \left( \rho - \frac{1}{2} H^c H_c \sigma \right) - P_{ic} \left( \mu^c
		- H^c \sigma \right)
	\end{pmatrix}
	\\  & =
	\begin{pmatrix}
		1  & 0 & 0 \\ 0  & \Pi^j
		_b & 0     \\ -\frac{1}{2} H^c H_c & -H_b & 1
	\end{pmatrix}
	   \begin{pmatrix}
		\nabla_i \sigma - \mu_i \\ \nabla_i \mu^b - ( \nabla_i H^b ) \sigma - H^b
		\nabla_i \sigma + P_i {}^b \sigma + \Pi^b _i \rho - \frac{1}{2} \Pi^b _i H^c
		H_c \sigma              \\ \nabla_i \rho - (H^c \nabla_i H_c) \sigma - \frac{1}{2} H^c H_c
		\nabla_i \sigma - P_{ic} \mu^c + P_{ic} H^c \sigma
	\end{pmatrix}
	\\  & =
	   \begin{pmatrix}
		\nabla_i \sigma - \mu_i                     \\ \Pi^j _b \left( \nabla_i \mu^b - (
		\nabla_i H^b ) \sigma - H^b \nabla_i \sigma + P_i {}^b \sigma + \Pi^b _i \rho -
		\frac{1}{2} \Pi^b _i H^c H_c \sigma \right) \\ -\frac{1}{2} H^c H_c ( \nabla_i
		\sigma - \mu_i ) - H_b \left( \nabla_i \mu^b - ( \nabla_i H^b ) \sigma - H^b
		\nabla_i \sigma + P_i {}^b \sigma + \Pi^b _i \rho - \frac{1}{2} \Pi^b _i H^c
		H_c \sigma \right)                          \\ +\nabla_i \rho - (H^c \nabla_i H_c) \sigma - \frac{1}{2}
		H^c H_c \nabla_i \sigma - P_{ic} \mu^c + P_{ic} H^c \sigma
	\end{pmatrix}
\end{align*}

Using the agreement of the intrinsic Levi-Civita connection with the pullback of the ambient Levi-Civita connection then gives
\begin{align*}
  	\check{\nabla}_i V^J & =      
	   \begin{pmatrix}
		D_i \sigma - \mu_i \\ D_i \mu^j + P_i {}^j \sigma + \delta_i ^j
		\rho - (\Pi^j _b \nabla_i H^b) \sigma - \frac{1}{2} \delta_i ^j H^c H_c \sigma
		\\ D_i \rho - P_{ic} \mu^c + \frac{1}{2} H^c H_c \mu_i - H_b \nabla_i \mu^b
	\end{pmatrix}
	\\  & =
	\begin{pmatrix}
		D_i \sigma - \mu_i \\ D_i \mu^j + p_i {}^j
		\sigma + \delta_i ^j \rho - (- \mathring{\II}_i {}^j {}_b - \delta_i ^j H_b)
		H^b\sigma - \frac{1}{2} \delta_i ^j H^c H_c \sigma + (P_i {}^j - p_i {}^j)
		\sigma             \\ D_i \rho - p_{ic} \mu^c + \frac{1}{2} H^c H_c \mu_i - H_b
		(\mathring{\II}_{ij} {}^b + \confmet_{ij} H^b) \mu^j + (P_i {}^j - p_i {}^j)
		\mu^j
	\end{pmatrix}
	\\  & =
	\begin{pmatrix}
		D_i \sigma - \mu_i             \\ D_i \mu^j + p_i
		{}^j \sigma + \delta_i ^j \rho \\ D_i \rho - p_{ic} \mu^c
	\end{pmatrix}
	   +
	   \begin{pmatrix}
		0                                             \\ \left(P_i {}^j - p_i {}^j + H_b \mathring{\II}_i {}^{jb} +
		\frac{1}{2}H^c H_c \delta_i ^j \right) \sigma \\ - \left( P_{ij} - p_{ij} + H_b
		\mathring{\II}_{ij} {}^b + \frac{1}{2} H_b H^b \confmet_{ij} \right) \mu^j
	\end{pmatrix}
	\\  & =
	D_i ^{\cT \Sigma}
	\begin{pmatrix}
		\sigma \\ \mu^j \\ \rho
	\end{pmatrix}
	 +
	                                       \begin{pmatrix}
		                                       0                                                         & 0 & 0 \\ \cF_i{}^j
		                                       & 0 & 0 \\ 0                     & - \cF_{ij}
		                                                                               & 0
	                                       \end{pmatrix}
	\begin{pmatrix}
		\sigma \\ \mu^j \\ \rho
	\end{pmatrix}
	.
\end{align*}
where
\begin{equation}\label{fialkow}
  \cF_{ij} : = P_{ij} - p_{ij} + H_b \tfII_{ij} {}^b + \frac{1}{2} H_b H^b \bg_{ij}.
\end{equation}
\end{proof}

We call the tensor $\cF_{ij}$  given in (\ref{fialkow})  the \emph{Fialkow tensor} (since this quantity seems to have appeared first in the work of Fialkow \cite{FialkowSubspaces}, cf.\ \cite{Juhl,YuriThesis}).
Since the checked connection, the intrinsic tractor connection and
$Z^j _{[J} X_{K]}$ are all conformally invariant, it follows that the
Fialkow tensor is also conformally invariant. Proposition \ref{SandF} tells us that the Fialkow tensor measures the failure of the induced tractor connection $\check{\nabla}$ to be \emph{normal} (in the sense of corresponding to a normal Cartan connection \cite{CG-tams}).

There is an alternative formula for the Fialkow tensor which is manifestly conformally invariant.
To derive it, one substitutes the Weyl-Schouten decomposition of the ambient and intrinsic Riemann tensors into the Gau{\ss} formula, and then applies the map 
\begin{equation}\label{eq:Schouten-projection}
T_{ijkl} \mapsto \frac{1}{m-2} \left( T_{ikj} {}^k - \frac{T_{kl} {}^{kl}}{2(m-1)} \bg_{ij} \right)
\end{equation}
to both sides (to take the ``submanifold Schouten'' part).
After writing the second fundamental form as $\tfII_{ij} {}^c + \bg_{ij} H^c$ and rearranging, one finds that the Fialkow tensor  \eqref{fialkow}
is equal to 
\begin{equation}\label{fialkow_weyl}
  \cF_{ij} = \frac{1}{m-2} \left( W_{icjd} \mathrm{N}^{cd} + \frac{W_{abcd} \mathrm{N}^{ac} \mathrm{N}^{bd}}{2(m-1)} \bg_{ij} + \mathring{\II}_i {}^{kc} \mathring{\II}_{jkc} -\frac{\mathring{\II}_{klc} \mathring{\II}^{klc}}{2(m-1)} \bg_{ij} \right).
\end{equation}
All objects on the right-hand side are conformally invariant. Since the Fialkow tensor has already been observed to be conformally invariant, it it sufficient to establish the formula \eqref{fialkow_weyl} in a minimal scale. Let $g\in\bc$ be a minimal scale and $g_{\Sigma}=\iota^*g$. Thus $\II=\tfII$. Applying the submanifold and ambient Weyl-Schouten decompositions in the Gauss equation we obtain
\begin{align*}
\nonumber
W_{ijkl} + P_{\vphantom{\hat{a}}ik\:\!}\bg_{\vphantom{g}jl} - P_{\vphantom{\hat{a}}jk\:\!}\bg_{\vphantom{g}il} - P_{\vphantom{\hat{a}}il\:\!}\bg_{\vphantom{g}jk} + P_{\vphantom{\hat{a}}jl\:\!}\bg_{\vphantom{g}ik} 
&=w_{ijkl} + p_{\vphantom{\hat{a}}ik\:\!}\bg_{\vphantom{g}jl} - p_{\vphantom{\hat{a}}jk\:\!}\bg_{\vphantom{g}il} - p_{\vphantom{\hat{a}}il\:\!}\bg_{\vphantom{g}jk} + p_{\vphantom{\hat{a}}jl\:\!}\bg_{\vphantom{g}ik}\\*
&\quad \qquad \;+ \bg_{cd}\II_{li}{^c}\II_{jk}{^d} - \bg_{cd}\II_{lj}{^c}\II_{ik}{^d}
\end{align*}
where $W_{ijkl}$ denotes the full projection $\Pi^a_i\Pi^b_j\Pi^c_k\Pi^d_l W_{abcd}$ of the ambient Weyl curvature, $P_{ij}=\Pi^a_i\Pi^b_j P_{ab}$, and $w_{ijkl}$ denotes the submanifold intrinsic Weyl tensor. Applying the map $T_{ijkl}\mapsto \frac{1}{m-2}\left(T_{ikj}{^k} - \frac{T_{kl}{}^{kl}}{2(m-1)}\bg_{ij}\right) $ on both sides of the above display we get
\begin{equation*}
-\tfrac{1}{m-2}\left( W_{icjd}\mathrm{N}^{cd} + \tfrac{W_{acbd}{}_{\vphantom{g}}\mathrm{N}^{ab}\mathrm{N}^{cd}}{2(m-1)}\bg_{ij} \right) + P_{ij} = p_{ij} + \tfrac{1}{m-2}\left(\tfII_i{}^{kc}\tfII_{jkc} - \tfrac{\tfII\vphantom{I}^{klc}{}_{\vphantom{g}}\tfII_{klc}}{2(m-1)}\bg_{ij} \right),
\end{equation*}
noting that $\bg^{kl}W_{ikjl}=\bg^{kl}\Pi^c_k\Pi^d_l W_{icjd}=-N^{cd}W_{icjd}$ since $W_{abcd}$ is trace free, and similarly $W_{kl}{}^{kl}=W_{acbd}\mathrm{N}^{ab}\mathrm{N}^{cd}$. The result then follows from \eqref{fialkow}.
\begin{rem}
In the $m=2$ case $(\Sigma, \bc_{\Sigma})$ we do not have an (intrinsicaly defined) Schouten tensor for a metric $g_{\Sigma}\in \bc_{\Sigma}$, and the map \eqref{eq:Schouten-projection} does not make sense. Crossing off the $\frac{1}{m-2}$ in \eqref{eq:Schouten-projection} does not help, as when $m=2$ the map $T_{ijkl} \mapsto  T_{ikj} {}^k - \frac{T_{kl} {}^{kl}}{2(m-1)} \bg_{ij}$ is easily seen to be equal to the zero map.  Since in two dimensions the space of algebraic curvature tensors is one dimensional, when $m=2$ the Gauss equation is equivalent to the scalar equation given by its ``double trace.'' If $K$ denotes the Gaussian curvature of $g_{\Sigma}\in \bc_{\Sigma}$ then contracting the Gauss equation twice gives
\begin{equation}\label{K-Gauss}
K = \bg^{ij}P_{ij} + \frac{1}{2}W_{ijkl}\bg^{ik}\bg^{jl} + 2|H|^2 - \frac{1}{2}|\II|^2.
\end{equation}
We'll make use of this in Section \ref{lowdim} below when we show that the conformal embedding of $(\Sigma, \bc_{\Sigma})$ into $(M,\bc)$ determines a canonical (extrinsically defined) Schouten tensor $p_{ij}$ for a metric $g_{\Sigma}\in \bc_{\Sigma}$.
\end{rem}

Recall that along $\Sigma$ we may decompose the ambient standard tractor bundle $\cT=\cT M$ as
$\cT|_{\Sigma} = \cT\Sigma \oplus \cN$.
If $V\in\Gamma(\cT|_{\Sigma})$ is given by $(V^{\top},V^{\perp})$ with respect to this decomposition, then by  \eqref{tractor_gauss}, \eqref{normal_tractor_connection}, \eqref{diff_tractor_connection} and the fact that $\nabla$ preserves the ambient tractor metric we have
\begin{equation}
\nabla_X V =
\left(\begin{array}{cc} 
D_X+\mathsf{S}(X)  &  -\LL(X)^T   \\
 \LL(X)  &  \nabla^{\cN} _X \\ 
\end{array}\right)
\left(\begin{array}{c} 
V^{\top}  \\
 V^{\perp}  \\ 
\end{array}\right)
\end{equation}
for any $X\in\mathfrak{X}(\Sigma)$, where $\LL(X)^T$ is the transpose of $\LL(X)$ with respect to the ambient tractor metric (cf.\ Proposition \ref{prop_alt_tractor_2ff}). We therefore write
\begin{equation}\label{eq:trac-conn-pullback}
\iota^*\nabla=
\left(\begin{array}{cc} 
D+\mathsf{S}  &  -\LL^T   \\
 \LL  &  \nabla^{\cN}  \\ 
\end{array}\right) 
\end{equation}
on $\cT|_{\Sigma} = \cT\Sigma\oplus \cN$. We will most often make use of the above display in the case where $V$ is a section of $\mathcal{T}\Sigma$ ($V^{\perp}=0$). In this case we obtain the following form of the \emph{tractor Gau{\ss} formula}:
\begin{equation}\label{TracGaussFormula}
\nabla_i V^B = \Pi^B_J\left(D_i V^J + \mathsf{S}_i{^J}_K V^K \right) +  \mathbb{L}_{iK}{}^B V^K,
\end{equation}
for any section $V$ of $\mathcal{T}\Sigma$. For a section

By computing the curvature of $\iota^*\Omega$ of $\iota^*\nabla$ using the decomposition \eqref{eq:trac-conn-pullback} 
one may easily obtain conformal tractor analogues of the Riemannian Gauss, Codazzi, and Ricci equations (cf.\ the derivation of the Riemannian Gau{\ss}-Codazzi-Ricci equations in Section~\ref{riem_subs}):
\begin{equation}\label{tractor_gauss_equation}
  \Omega_{ijKL} = \Omega^\Sigma {}_{ijKL} + 2 D_{[i} \mathsf{S}_{j]KL} + 2 \mathsf{S}_{[i| K M} \mathsf{S}_{|j]}{}^{M}{}_{L} + 2 \LL_{[i|L}{}^C \LL_{|j] KC}
\end{equation}
\begin{equation}\label{tractor_codazzi}
\Omega_{ij}{^E}^{}_L \mathrm{N}^C _E = 2 D_{[i} \LL_{j] L}{}^C + 2 \LL_{[i| K}{}^{C} \mathsf{S}_{|j]} {}^K {}_{L}
\end{equation}
\begin{equation}\label{tractor_ricci}
  \Omega_{ij} {}^A {}_B \mathrm{N}^C _A \mathrm{N}^B _D = \Omega^\cN {}_{ij} {}^C {}_D - 2 g^{KL}\LL_{[i K} {}^{C} \LL_{j] LD},
\end{equation}
where indices between bars are exempt from antisymmetrisation, $\Omega_{ijKL} = \Omega_{ijCD} \Pi^C _K \Pi^D _L$ for $\Omega_{ijCD}$ the curvature of the pullback connection, and $\Omega^\cN$ is the curvature of the normal tractor connection, characterised by 
\begin{equation}
  \Omega^\cN {}_{ij} {}^C {}_D N^D = \left( \nabla^\cN _i \nabla^\cN _j - \nabla^\cN _j \nabla^\cN _i \right) N^C,
\end{equation}
for any section $N$ of the normal tractor bundle $\mathcal{N}$.

\begin{rem}\label{BCremark}
In \cite{CalderbankBurstall} Burstall and Calderbank define a
`M\"obius reduction' to be a rank $(m+2)$ subbundle $\cV$ of
$\cT|_{\Sigma}$ containing the rank $m+1$ subbundle spanned by the
canonical tractor $X^A$ and its covariant derivatives in submanifold
tangential directions (with respect to the ambient tractor connection
coupled with the Levi-Civita connection of some, equivalently any,
metric $g\in\bc$). One then decomposes the ambient tractor connection
along $\Sigma$ as (using notation similar to the above)
\begin{equation*}
\iota^*\nabla=
\left(\begin{array}{cc} 
\nabla^{\cV}  &  -(\LL^{\cV})^T   \\
 \LL^{\cV}  &  \nabla^{{\cV}^{\perp}}  \\ 
\end{array}\right) 
\end{equation*}
on $\cT|_{\Sigma}=\cV \oplus \cV^{\perp}$. The definition of `M\"obius reduction' implies that $\LL^{\cV}{}_{\:\!iJ}{^C}X^J=0$ and $\LL^{\cV}{}_{\;\!iJ}{^C}X_C=0$, so that there is a well-defined projection $\II^{\cV}{}_{\;\!ij}{^c}:=\LL^{\cV}{}_{\;\!iJ}{^C}Z^J_jZ^C_c$ of $\LL^{\cV}_{\;\!iJ}{^C}$.
Burstall and Calderbank then define the unique `canonical M\"obius reduction' $\cV_{\Sigma}$ by imposing an algebraic normalisation condition on 
\begin{equation*}
\left(\begin{array}{cc} 
0 &  -(\LL^{\cV})^T   \\
 \LL^{\cV}  &  0 \\ 
\end{array}\right)
\end{equation*}
similar to the algebraic normalisation condition imposed on the curvature of the normal Cartan/tractor connection \cite{CS-book,CG-tams}, see Section 9.3 of \cite{CalderbankBurstall}. This algebraic normalisation condition amounts to the requirement that $\bg^{ij}\II^{\cV}{}_{ij}{^c}=0$. Since by Theorem \ref{tractor2ff_XZ} the tractor second fundamental form has invariant projection $\mathring{\II}_{ij}{^c}=\LL_{iJ}{^C}Z^J_jZ^C_c$ the `canonical M\"obius reduction' $\cV_{\Sigma}$ is the same as the orthogonal complement $\cN^{\perp}$ of the normal tractor bundle and hence gives an abstract characterization of this bundle (equivalently of the normal tractor bundle $\mathcal{N}$). Our approach differs in that we explicitly construct $\mathcal{N}$, and then further explicitly identify $\cN^{\perp}$ with the intrinsic tractor bundle $\mathcal{T}\Sigma$. 
\end{rem}

\subsection{Low-dimensional conformal submanifolds}\label{lowdim}

 In this section we treat submanifolds $\Sigma$ such that
 $\operatorname{dim}(\Sigma)$ is $m=1$ or $m=2$. Note that Section
 \ref{sub-tr} has no restriction on the submanifold dimension
 $m$. However in Section \ref{geq3}, just above, we make the
 restriction to $m\geq 3$ to discuss the intrinsic tractor connection
 then available. When $m=1,2$ the conformal structure on $\Sigma$ is
 not sufficient to determine a canonical connection on $\cT \Sigma$.
 The purpose of this section is to observe that in these dimensions the conformal
 embedding does determine distinguished tractor connections on $\cT\Sigma$, and then using this we get analogues of the results from Section \ref{geq3}.

First recall that Equation~\eqref{checked_connection} defines a connection
$\check{\nabla}$ on the bundle $\cT \Sigma$ also when $\dim \Sigma$ is 1 or
2.

Riemannian manifolds of dimension 1 or 2 are not naturally equipped
with an (intrinsically determined) Schouten tensor. However conformal
submanifolds of these dimensions inherit a natural replacement, as
follows.  First recall that for submanifolds of dimension at least 3,
the difference tractor of is equivalent to the Fialkow tensor of the
submanifold according to expressions~\eqref{difference_tractor} and
\eqref{fialkow}.  In dimensions 1 and 2, we will, in essence, turn
this around and use the formula \eqref{fialkow} {\em to determine} a
submanifold Schouten tensor (for a given $g_{\Sigma}\in \bc_{\Sigma}$).  In these dimensions all terms in
\eqref{fialkow} are well defined, as usual, except the submanifold
Schouten $p_{ij}$ and the Fialkow tensor $\cF_{ij}$. Note that fixing
one of these two determines the other via \eqref{fialkow}. Moreover,
from the conformal transformation formulae of the terms in
\eqref{fialkow}, it follows that {\em any} natural conformally
invariant choice of $\cF_{ij}$ determines a submanifold tensor
$p_{ij}\in S^2T^*\Sigma$ that transforms conformally according to
\eqref{schouten_trans}. From a conformal geometry point of view it makes sense to define $\cF_{ij}$ first (in a conformally invariant fashion) and view the definition of the Schouten tensor $p_{ij}$ of a metric $g_{\Sigma}\in \bc_{\Sigma}$ as a consequence of this. Such a Schouten tensor $p_{ij}$ then yields a
conformally invariant tractor connection $D$ on $\cT\Sigma$ (in a
choice of scale) according to the usual formula, namely
\begin{equation}\label{12-tracc}
  D_i \begin{pmatrix}\si \\\mu_j \\\rho \end{pmatrix}=  \begin{pmatrix}D_i \si-\mu_i \\D_i\mu_j +p_{ij}\si+\bg_{ij}\rho \\ D_i \rho-p_{ij}\mu^j \end{pmatrix} ,
\end{equation}
where $D$ on the right-hand side is the intrinsic Levi-Civita
connection.
Equivalently the tractor
connection is determined by formula
\eqref{diff_tractor_connection} with
$$
  \mathsf{S}_{i J K} := 2 \cF_{ij}  Z^j _{[J} X_{K]}. 
$$

  Thus it remains to specify an invariant $\cF_{ij}$, or equivalently
  a $p_{ij}$ that transforms according to \eqref{schouten_trans}.
  Recall that in all dimensions we have the conformally invariant
  tractor connection $ \check{\nabla}$ on $\cT\Sigma$. It is given by 
\begin{equation}  
  	\check{\nabla}_i V^J =      
	   	\begin{pmatrix}
		D_i \sigma - \mu_i             \\ D_i \mu^j + \left(P_i {}^j + H_b \mathring{\II}_i {}^{jb} +
		\frac{1}{2}H^c H_c \delta_i ^j \right)\sigma + \delta_i ^j \rho \\ D_i \rho - \left( P_{ij}  + H_b
		\mathring{\II}_{ij} {}^b + \frac{1}{2} H_b H^b \confmet_{ij} \right) \mu^j
	\end{pmatrix} ,
\end{equation}
where we have computed using  a choice of ambient scale $g\in \cc$,
and $V^J\stackrel{g_\Sigma}{=}(\si,~\mu^j,~\rho )$. Note that although the terms $P_{ij}$ and $H^a$ appearing on the right hand side in the above display depend on the extension $g\in \bc$ of $g_{\Sigma}$, the right hand side itself does not (by the conformal invariance of the $\check{\nabla}$ and the fact that $g_\Sigma$ is sufficient to determine the tractor bundle splitting); in particular, the quantity $P_{ij}  + H_b
		\mathring{\II}_{ij} {}^b + \frac{1}{2} H_b H^b \confmet_{ij}$ is independent of the extension $g$ of $g_{\Sigma}$. Moreover, from the conformal invariance of
$\check{\nabla}$ it can be shown that $P_{ij}  + H_b
		\mathring{\II}_{ij} {}^b + \frac{1}{2} H_b H^b \confmet_{ij}$
transforms in the same way as a Schouten tensor when $g_{\Sigma}$ is rescaled conformally.  In particular, if we
define
\begin{equation}\label{eq:p-in-terms-of-Fialkow}
p_{ij}:= \left( P_{ij}  + H_b
\mathring{\II}_{ij} {}^b + \frac{1}{2} H_b H^b \confmet_{ij} \right) -\cF_{ij} ,
\end{equation}
for any (decreed to be) conformally invariant $\cF_{ij}\in \Gamma(\ce_{(ij)})$, $p_{ij}$ depends only on $g_\Sigma$ (and the conformal embedding of $
\Sigma$ in $(M,\cc)$), \eqref{fialkow} holds, and $p_{ij}$ has conformal transformation
\begin{equation}
\widehat{p}_{ij}=p_{ij}-D_i\Upsilon_j+\Upsilon_i\Upsilon_j - \frac{1}{2} \bg_{ij} \Upsilon^k\Upsilon_k.
\end{equation}
This formula is also easily checked directly. When such a choice of submanifold Schouten tensor has been made we will denote its trace by
\begin{equation}\label{2-J}
  \lj:=\bg^{ij}p_{ij}.
  \end{equation}
Then under conformal change
\begin{equation}\label{J-trans}
\widehat{\lj}= \lj -D_i\Upsilon^i +(1-\frac{m}{2})\Upsilon^i\Upsilon_i .
\end{equation}

In dimension 1 we will simply set $\cF_{ij}:=0$. Thus given any submanifold scale $g_{\Sigma}$ and any extension $g\in
\cc$ one has
\begin{equation}
  p_{ij} := P_{ij} + \frac{1}{2} H_b H^b \bg_{ij},
\end{equation}
  since, for a curve, the trace-free second
  fundamental form is trivially zero. In this case, of course, $p_{ij}={\emph{\j}}g_{ij}$ where ${\emph{\j}}:=
  \bg^{ij}p_{ij}$. Note that the extension  $g\in
\cc$ can be chosen such that $H^b=0$ and then one simply has $ p_{ij} := P_{ij}$.

  In dimension 2 we shall set $\cF_{(ij)_0}=0$. We are not free to set $\bg^{ij}\cF_{ij}$ equal to zero, though, as we require that our corresponding Schouten tensor (for $g_{\Sigma}\in \bc_{\Sigma}$) satisfy
\begin{equation}\label{eq:p-curvature-normalization}
R^{g_\Sigma}_{ijkl} = 2\bg_{k[i}p_{j]l}-2\bg_{l[i}p_{j]k}.
\end{equation}
(This condition on $p_{ij}$ is natural by analogy with higher dimensions, but also from the point of view of Cartan geometry/tractor calculus. Specifically, requiring \eqref{eq:p-curvature-normalization} is  equivalent to requiring that the submanifold tractor connection $D$ be \emph{normal}, in the sense of \cite{CG-tams}; this is ultimately because it amounts to vanishing of the ``middle slot'' of the curvature of the tractor connection $D$, the slot where the Weyl curvature sits in dimensions four and higher, as can be seen by an easy calculation using \eqref{12-tracc}.) Equation \eqref{eq:p-curvature-normalization} normalizes the trace of $p_{ij}$, and hence of $\cF_{ij}$. Indeed, \eqref{eq:p-curvature-normalization} is equivalent to the requirement that the trace $\lj=\bg^{ij}p_{ij}$ equals the Gaussian curvature of $g_{\Sigma}$:
\begin{equation}
\lj=\bg^{ij}p_{ij}:= \frac{1}{2}\bg^{ij}\operatorname{Ric}^{g_\Sigma}_{ij} = K,
\end{equation}
where $K=K^{g_{\Sigma}}$ is viewed as as section of $\cE[-2]|_{\Sigma}$ in the natural way. With the trace of $p_{ij}$ normalized in this way, the trace of $\cF_{ij}$ is therefore given by (tracing \eqref{fialkow} or \eqref{eq:p-in-terms-of-Fialkow}):
\begin{equation}
\bg^{ij}\cF_{ij} = \bg^{ij}P_{ij} - K +|H|^2.
\end{equation}
Using the (twice contracted) Gauss equation \eqref{K-Gauss} for $K$ we have, equivalently, 
\begin{equation}
\bg^{ij}\cF_{ij} = \frac{1}{2}|\II|^2 - |H|^2 - \frac{1}{2}W_{ijkl}\bg^{ik}\bg^{jl} = \frac{1}{2}|\mathring{\II}|^2  - \frac{1}{2}W_{ijkl}\bg^{ik}\bg^{jl}.
\end{equation}
Writing the submanifold total trace $W_{ijkl}\bg^{ik}\bg^{jl}$ of the ambient Weyl curvature as $\mathrm{tr}^2\, \iota^*W$ the Fialkow tensor for $m=2$ is therefore defined to be
\begin{equation}\label{eq:Fialkow-2D}
\cF_{ij} = \frac{1}{4}\left(|\mathring{\II}|^2  - \mathrm{tr}^2\, \iota^*W \right)\bg_{ij}.
\end{equation}
The Schouten tensor $p_{ij}$ of $g_{\Sigma}$ defined by \eqref{eq:p-in-terms-of-Fialkow} then satisfies \eqref{eq:p-curvature-normalization}.

With the conventions we have just established for $m=2$,  the curvature of the tractor connection $D$ is then given by
$$
\Omega^{\Sigma}_{ijKL} = -2 c_{ijk} X_{[K} Z_{L]}{}^k,
$$
where $c_{ijk} := 2D_{[i}p_{j]k}$. This should be compared with \eqref{ctract-curv} in the three dimensional case where the Weyl tensor term vanishes. Note that, while in three and higher dimensions the Cotton tensor is trace free, since we are in two dimensions the tensor $c_{ijk}$ can be written as 
$$
c_{ijk} = \frac{1}{2}\boldsymbol{\epsilon}_{ij}c_k
$$
where $c_k = \boldsymbol{\epsilon}^{ij}c_{ijk}$ and hence $c_{ijk}$ is determined by its trace
$$
c_{ijk}\bg^{jk} = \frac{1}{2}\boldsymbol{\epsilon}_{ij}c^j = D_{i}\lj - D^jp_{ij}.
$$
In the $2$-dimensional case, the choice of $p_{ij}$ for each scale (with the trace part normalized by setting $g^{ij}p_{ij}$ to be the Gau{\ss} curvature) is equivalent to a choice of M\"obius structure on $\Sigma$ in the sense of \cite{CalderbankMobiusStructures}. The invariant $c_{ijk}$, which we refer to as the \emph{Cotton tensor} in this setting, is precisely the curvature of this M\"obius structure and vanishes if and only if $\Sigma$ (with the M\"obius structure just defined) is locally equivalent to the conformal M\"obius sphere (i.e.\ comes from a system of local coordinates on $\Sigma$ related by M\"obius transformations); see \cite{CalderbankMobiusStructures} for more details. Note that the M\"obius structure we have just defined on $2$-dimensional submanifolds $\Sigma$ agrees with the notion of
\emph{induced conformal M\"obius structure} in \cite{CalderbankBurstall}.

  In both dimensions 1 and 2 we then have that, in any scale $g\in
  \cc$, and with $g_\Sigma=g|_{T\Sigma}$, the tractor connection on $[\cT\Sigma]_{g_\Sigma}$ is
  given by \eqref{12-tracc}.
  We note that this formula exactly agrees with the formula
for the usual tractor connection, as defined by~\eqref{ctrids} (but has $p_{ij}$ as defined here).  This
connection in turn defines a Thomas-D operator $\mathbb{D} : \cE[1]
\to \cT \Sigma$ via the usual BGG splitting operator characterisation
(see e.g.\ \cite{CGH-Duke}), namely that, for any $\si\in
\Gamma(\ce[1])$, $D_i\mathbb{D}_I \si$ must be $0$ in the top slot and be
trace-free in the middle slot. 
Thus, in a scale $g_\Sigma \in \bc_\Sigma$, this takes the form
\begin{equation}\label{submanifold_D}
  \frac{1}{m} \mathbb{D}_K \sigma 
  \stackrel{g_\Sigma}{=}
  \begin{pmatrix}
    \sigma \\ 
    D_k \sigma \\
    -\frac{1}{m} \left( \bg^{i\ell}D_iD_\ell + \lj  \right) \sigma 
  \end{pmatrix}, 
\end{equation}
(cf.\ \eqref{thomas_D}). 
It follows at once that (as in higher dimensions $m$) any parallel standard tractor $I$ is necessarily in the image of
$\frac{1}{m} \mathbb{D} $. Moreover we have that, for any  $\si\in \Gamma(\ce[1])$, and  on the set where $\si$ is non-vanishing, the scale tractor 
$$
I:= \frac{1}{m} \mathbb{D} \si 
$$
satisfies that
$$
h^{KL}I_KI_L= -\frac{2}{m} \lj \, \si^2= -\frac{2}{m}\lj^{g_\Sigma}
$$
where $\lj^{g_\Sigma} =\lj \si^2 $ is the (weight zero true) $J$-curvature for the scale $g_\Sigma$, and this is thus constant if $I$
is parallel. (E.g., for surfaces the scale tractor parallel being implies the corresponding metric has constant Gau{\ss} curvature.) So this fits with the situation in higher dimensions.

Finally observe that the Thomas-$D$ formula \eqref{submanifold_D}
evidently provides a conformally invariant isomorphism
$$
J^2 \cE_\Sigma [1]/S^2 _0 T^*\Sigma [1] \stackrel{\cong}{\longrightarrow} \cT \Sigma,
$$
cf. \eqref{T_isom}.
Thus our description of the tractor bundle from the
introduction still applies in dimensions $m=1$ and $m=2$, where we identify $\cT\Sigma $ with $\cN^\perp$.

\begin{rem}\label{low-comments}
%

In dimension $m=1$ parallel transport using \eqref{12-tracc}
is equivalent in an obvious way to a third order linear ODE along the
curve. In dimension $m=2$ parallel tractors, with the top slot $\si\in
\Gamma(\ce [1])$ non-vanishing, correspond to solutions of
\begin{equation}\label{ae-2}
D_{(i}D_{j)_0}\si+ p_{(ij)_0}\si=0
  \end{equation}
  that also satisfy that the conformal
  invariant $c_{ij}{}^j = D_i \mbox{\j} - D^jp_{ij}$ is zero (if \eqref{ae-2} alone holds then following the usual prolongation argument one can easily show that $I^K=\frac{1}{2}\TD^K\sigma$ satisfies $D_jI^K = \sigma (D^ip_{ij}-D_j \mbox{\j})X^K$, cf.\ \cite[Section 3.4]{Curry-G-conformal}); it follows from the above discussion that nontrivial parallel tractors only exist when the M\"obius structure is locally flat, in which case they define a metric of constant Gaussian curvature on the open dense set where $\sigma\neq 0$.
  \end{rem}

\subsection{Normal forms}\label{norm_forms}

Since $\dim N^* \Sigma = d$, $\Lambda^d N^* \Sigma$ is a line bundle.
Moreover, $\Sigma$ is oriented and it thus follows that there is a unique section $N_{a_1 a_2 \cdots a_d}$ of $\Lambda^d N^* \Sigma [d]$ which is compatible with the orientations of $\Sigma$ and $M$, and such that 
\begin{equation*}
  N^{a_1 a_2 \cdots a_d} N_{a_1 a_2 \cdots a_d} = d!~;
\end{equation*} 
here by \emph{compatible with the orientations} we mean that
\begin{equation*}
\boldsymbol{\epsilon}^{\Sigma}_{a_1a_2\cdots a_m} \wedge N_{a_{m+1}\cdots a_n} = \boldsymbol{\epsilon}_{a_1a_2\cdots a_n},
\end{equation*}
where $\boldsymbol{\epsilon}^{\Sigma}$ and $\boldsymbol{\epsilon}$ are the weighted volume forms for $\Sigma$ and $M$ respectively.
We call $N_{a_1 a_2 \cdots a_d}$ the \emph{Riemannian normal form} for the submanifold $\Sigma$.
It is not hard to show that this contains the same data as the normal projector $\mathrm{N}^a _b$. Indeed, one can obtain the latter from the former via $N^a_b = \frac{1}{(d-1)!}N^{aa_2\cdots a_d}N_{ba_2\cdots a_d}$, cf.\ Proposition \ref{prop_normal_projector_normal_forms}.  

This object also has a tractor analogue.
Recall that the normal tractor bundle $\cN^*$ is isomorphic to $N^* \Sigma [1]$.
Thus it follows that, for any $k$ such that $1 \leq k \leq \codim \Sigma$, one has $\Lambda^k \cN^* \cong \Lambda^k N^* \Sigma [k]$. 
Explicitly, for $\nu_{a_1 a_2 \cdots a_k} \in \Lambda^k N^* \Sigma [k]$, the isomorphism is given by 
\begin{equation} \label{eqn_exterior_normal_isom}
  \nu_{a_1 a_2 \cdots a_k} \mapsto \nu_{a_1 a_2 \cdots a_k} \mathbb{Z}^{a_1 a_2 \cdots a_k} _{A_1 A_2 \cdots A_k} + k \cdot \nu_{b a_2 \cdots a_k} H^b \mathbb{X}_{A_1 A_2 \cdots A_k} ^{\phantom{A_1} a_2 \cdots a_k},
\end{equation}
by taking the $k$-th exterior power of~\eqref{norm_bundle}. (Note that in the case $k=1$ this
  map is simply the map $N^* \Gamma [1] \to \cN^*$
  of~\eqref{norm_bundle}.)
Invariance of this map may independently be checked via the transformation formulae for the tractor form projectors and the mean curvature:
\begin{align*}
	 & \hat{\nu}_{a_1 a_2 \cdots a_k} \hat{\mathbb{Z}}_{A_1 A_2
		\cdots A_k} ^{a_1 a_2 \cdots a_k} + k \cdot \hat{\nu}_{b a_2 \cdots a_k}
	\hat{H}^b \hat{\mathbb{X}}_{A_1 A_2 \cdots A_k} ^{\phantom{A_1} a_2 \cdots a_k}
	\\  & = \nu_{a_1 a_2 \cdots a_k} \left( \mathbb{Z}_{A_1 A_2 \cdots A_k} ^{a_1 a_2
		   \cdots a_k} + k \cdot \Upsilon^{a_1} \mathbb{X}_{A_1 A_2 \cdots A_k}
	^{\phantom{A_1} a_2 \cdots a_k} \right)                                           \\  & \hspace{1em} + k \cdot \nu_{b a_2
		   \cdots a_k} \left( H^b - \mathrm{N}^b _{a_1} \Upsilon^{a_1} \right)
	\mathbb{X}_{A_1 A_2 \cdots A_d} ^{\phantom{A_1} a_2 \cdots a_k}                   \\  & = \nu_{a_1
	a_2 \cdots a_k} \mathbb{Z}_{A_1 A_2 \cdots A_k} ^{a_1 a_2 \cdots a_k}             \\
	 & \hspace{1em} + k \cdot \left( \nu_{b a_2 \cdots a_k} H^b + \nu_{a_1 a_2 \cdots
		a_k} \Upsilon^{a_1} - \nu_{b a_2 \cdots a_k} \mathrm{N}^b _{a_1} \Upsilon^{a_1}
	\right) \mathbb{X}_{A_1 A_2 \cdots A_k} ^{\phantom{A_1} a_2 \cdots a_k}           \\  & =
	   \nu_{a_1 a_2 \cdots a_k} \mathbb{Z}^{a_1 a_2 \cdots a_k} _{A_1 A_2 \cdots A_k}
	   + k \cdot \nu_{b a_2 \cdots a_k} H^b \mathbb{X}_{A_1 A_2 \cdots A_k}
	   ^{\phantom{A_1} a_2 \cdots a_k}.
\end{align*}

In particular, it follows that there is a distinguished section of the line bundle $\Lambda^d \cN^*$, where $d = \codim \Sigma$, given by the image of the Riemannian normal form under this isomorphism.
We write $N_{A_1 A_2 \cdots A_d}$ for this section and call it the \emph{tractor normal form} for the submanifold $\Sigma$.
From the above isomorphism, one has
\begin{equation}\label{tractor_normal_form}
  N_{A_1 A_2 \cdots A_d} = N_{a_1 a_2 \cdots a_d} \mathbb{Z}_{A_1 A_2 \cdots A_d} ^{a_1 a_2 \cdots a_d} 
  +d\cdot N_{b a_2 \cdots a_d} H^b \mathbb{X}_{A_1 A_2 \cdots A_d} ^{\phantom{A_1} a_2 \cdots a_d},
\end{equation}
and hence 
\begin{align*}
  N_{A_1 A_2 \cdots A_d} N^{A_1 A_2 \cdots A_d}
  &= N_{a_1 a_2 \cdots a_d} N^{b_1 b_2 \cdots b_d} \mathbb{Z}_{A_1 A_2 \cdots A_d} ^{a_1 a_2 \cdots a_d} \mathbb{Z}^{A_1 A_2 \cdots A_d} _{b_1 b_2 \cdots b_d} \\
  &= N_{a_1 a_2 \cdots a_d} N^{a_1 a_2 \cdots a_d} \\
  &= d!,
\end{align*}
since all other contractions of the $\mathbb{X}$ and $\mathbb{Z}$ projectors are zero, and where the indices of $N^{A_1 A_2 \cdots A_d}$ have been raised with the tractor metric.

Given a local orthonormal basis $\{N^1 _A, \ldots , N^d _A \}$ for the normal tractor bundle, which may be constructed from a local orthonormal basis of $N^* \Sigma[1]$, one sees that 
  \begin{equation}\label{eqn_normal_form_wedge}
    d! \cdot N^1 _{[A_1} \cdots N^d _{A_d]} = N^1 _{A_1} \wedge \cdots \wedge N^d _{A_d}
  \end{equation}
is clearly a section of $\Lambda^d \cN^*$ and satisfies the above
normalisation condition.  If the chosen basis is compatible with the
orientation, then~\eqref{eqn_normal_form_wedge} recovers the tractor
normal form $N_{A_1 A_2 \cdots A_d}$. By construction one then has
that
\begin{equation*}
\epsilon^{\Sigma}_{A_1A_2\cdots A_{m+2}}\wedge N_{A_{m+3}\cdots A_{n+2}} = \epsilon_{A_1A_2\cdots A_{n+2}},
\end{equation*}
where $\epsilon^{\Sigma}$ and $\epsilon$ are the tractor volume forms of $\Sigma$ and $M$ respectively.

Our task is now to relate the tractor normal form and its derivative to the other tractor objects introduced, namely, the tractor normal projector and the tractor second fundamental form.
These relationships will lay the foundation for the notion of distinguished submanifold that we will introduce in the following chapter.

First, we express the tractor normal projector in terms of the normal tractor form. Note that the normal tractor can be thought of as a ``volume form'' for the normal bundle. Recall that for the volume form $\boldsymbol{\epsilon}$ on $M$ we have the following identity
\begin{equation}
\boldsymbol{\epsilon}^{a_1\cdots a_kb_{k+1}\cdots b_n} \boldsymbol{\epsilon}_{b_1\cdots b_k b_{k+1}\cdots b_n} = k!(n-k)!\delta^{[a_1}_{b_1}\cdots \delta^{a_k]}_{b_k},
\end{equation}
which can be thought of as an index notation version of the standard identity $\star\star = (-1)^{k(n-k)}$ on $k$-forms (since it is equivalent to $\frac{1}{(n-k)!}\boldsymbol{\epsilon}^{b_{1}\cdots b_{n-k}}{}_{c_1\cdots c_k} \frac{1}{k!}\boldsymbol{\epsilon}^{a_1\cdots a_k}{}_{b_{1}\cdots b_{n-k}}=(-1)^{k(n-k)}\delta^{[a_1}_{c_1}\cdots \delta^{a_k]}_{c_k}$). Since the tractor metric is positive definite when restricted to $\cN$, the same algebraic identity applies to the tractor normal form, giving
\begin{equation}\label{eq-partial-cont-N-forms}
N^{A_1\cdots A_kB_{k+1}\cdots B_d} N_{B_1\cdots B_k B_{k+1}\cdots B_n} = k!(d-k)!\mathrm{N}^{[A_1}_{B_1}\cdots \mathrm{N}^{A_k]}_{B_k},
\end{equation}
for $k\in\{1,\ldots, d\}$. In particular, we have:

\begin{prop}\label{prop_normal_projector_normal_forms}
	The tractor projector $\mathrm{N}^A _B$ is equal to
	\begin{equation}\label{eq-N-proj-N-form}
    \mathrm{N}^{A} _{B} = \frac{1}{(d-1)!} N^{A B_2 \cdots B_d} N_{B B_2 \cdots B_d}.
	\end{equation}
\end{prop}

Differentiating the formula obtained in the above display leads to a relationship between the tractor second fundamental form and the derivative of the tractor normal form. An alternative route to this is via the following lemma, which we record for completeness.
\begin{lem}
  The derivative of the tractor normal form expressed in the tractor projector notation is 
  \begin{equation}\label{eqn_tractor_normal_form_derivative_projectors}
    \begin{split}
      \nabla_i N_{A_1 A_2 \cdots A_{d-1} A_d}
      &= \left[ \nabla_i N_{a_1 a_2 \cdots a_{d-1} a_d} + d \cdot N_{b a_2 \cdots a_{d-1} a_d} H^b \confmet_{i a_1} \right] \mathbb{Z}^{a_1 a_2 \cdots a_d} _{A_1 A_2 \cdots A_d} \\
      &\quad + d \cdot \left[ \nabla_i \left( N_{b a_2 \cdots a_{d-1} a_d} H^b \right) - N_{a_1 a_2 \cdots a_{d-1} a_d} P_i {}^{a_1} \right] \mathbb{X}^{\phantom{A_1} a_2 \cdots a_d} _{A_1 A_2 \cdots A_d}. \\ 
    \end{split}
  \end{equation}
\end{lem}
\begin{proof}
  Recall 
  \begin{equation*}
    N_{A_1 A_2 \cdots A_d} = N_{a_1 a_2 \cdots a_d} \mathbb{Z}^{a_1 a_2 \cdots a_d} _{A_1 A_2 \cdots A_d} + d\cdot N_{b a_2 \cdots a_d} H^b \mathbb{X}^{\phantom{A_1} a_2 \cdots a_d} _{A_1 A_2 \cdots A_d}.
  \end{equation*}
  Differentiating this,
  \begin{align*}
    \nabla_i N_{A_1 A_2 \cdots A_{d-1} A_d}
    &= \left( \nabla_i N_{a_1 a_2 \cdots a_d} \right) \mathbb{Z}^{a_1 a_2 \cdots a_d} _{A_1 A_2 \cdots A_d} \\ 
    &\hspace{2em}+ N_{a_1 a_2 \cdots a_d} \left( - d \cdot P_i {}^{a_1} \mathbb{X}^{\phantom{A_1} a_2 \cdots a_d} _{A_1 A_2 \cdots A_d} - d \cdot \delta_i {}^{a_1} \mathbb{Y}^{\phantom{A_1} a_2 \cdots a_d} _{A_1 A_2 \cdots A_d} \right) \\
    &\hspace{2em}+ d \cdot \nabla_i \left( N_{b a_2 \cdots a_d} H^b \right) \mathbb{X}^{\phantom{A_1} a_2 \cdots a_d} _{A_1 A_2 \cdots A_d} \\
    &\hspace{2em}+ d \cdot N_{b a_2 \cdots a_d} H^b \left( \confmet_{i a_1} \mathbb{Z}^{a_1 a_2 \cdots a_d} _{A_1 A_2 \cdots A_d} - (d-1) \cdot \delta_i {}^{a_2} \mathbb{W}^{\phantom{A_1 A_2} A_3 \cdots A_d} _{a_1 a_2 a_3 \cdots a_d} \right) \\
    &= \left[ \nabla_i N_{a_1 a_2 \cdots a_{d-1} a_d} + d \cdot N_{b a_2 \cdots a_{d-1} a_d} H^b \confmet_{i a_1} \right] \mathbb{Z}^{a_1 a_2 \cdots a_d} _{A_1 A_2 \cdots A_d} \\
    &\hspace{2em} + d \cdot \left[ \nabla_i \left( N_{b a_2 \cdots a_{d-1} a_d} H^b \right) - N_{a_1 a_2 \cdots a_{d-1} a_d} P_i {}^{a_1} \right] \mathbb{X}^{\phantom{A_1} a_2 \cdots a_d} _{A_1 A_2 \cdots A_d},
  \end{align*}
  where we use the fact that any terms where the $i$ index is contracted into the normal form will vanish, since $i$ is a tangential index.
\end{proof}

\subsection{The proof of Theorem~\ref{key1}} \label{proofkey1}

Here we give a proof of the equivalence of conditions 1--4 in Theorem~\ref{key1}. This is fairly straightforward (in the end all four conditions are equivalent to normal bundle $\cN$ being parallel). In establishing the result, however, we will make some calculations which are of independent interest. Henceforth we will only use indices to distinguish $\overline{\LL}$ and
$\LL$ which (as sections of different bundles) represent the same
object, i.e. we will write $\LL_{i A} {}^B = \Pi^J _A \LL_{i J} {}^B$
instead of \eqref{LLbar}.

Toward Theorem~\ref{key1}, first observe that Proposition \ref{NgradN=L} and
Lemma \ref{lem_der_normal_tractor_proj} together give equivalence of 1
and 2 in that theorem. 
The following theorem shows that 1 implies 3.

\begin{thm}\label{thm_grad_normal_form}
  The derivative of the tractor normal form is given in terms of the tractor second fundamental form by
  \begin{equation}\label{eqn_grad_norm}
    \nabla_i N_{A_1 A_2 \cdots A_{d-1} A_d} = -d \cdot \LL_{i [A_d} {}^{C} N_{A_1 A_2 \cdots A_{d-1}] C}.
  \end{equation}
\end{thm}

\begin{proof}
Recall that by \eqref{eq-N-proj-N-form} we have
\begin{equation}
\nabla_i \mathrm{N}^C_B = - \LL_{iB}{}^C - \LL_{i}{}^C{}_B.
\end{equation}
In order to utilize the above formula we recall also that
\begin{equation}
\frac{1}{d!}N^{A_1\cdots A_d}N_{B_1\cdots B_d} = \mathrm{N}^{[A_1}_{B_1}\cdots \mathrm{N}^{A_d]}_{B_d}.
\end{equation}
Differentiating both sides of the above display we obtain that 
\begin{equation*}
\begin{aligned}
\frac{1}{d!}(\nabla_iN^{A_1\cdots A_d})&N_{B_1\cdots B_d} + \frac{1}{d!}N^{A_1\cdots A_d}\nabla_iN_{B_1\cdots B_d} \\& = - \LL_{iB_1}{}^{[A_1}\mathrm{N}_{B_2}^{A_2}\cdots \mathrm{N}_{B_d}^{A_d]} -\mathrm{N}_{B_1}{}^{[A_1}\LL_{iB_2}{}^{A_2}\cdots \mathrm{N}_{B_d}^{A_d]} - \cdots -  \mathrm{N}_{B_1}{}^{[A_1}\mathrm{N}_{B_2}^{A_2}\cdots \LL_{iB_d}{}^{A_d]}\\
& \phantom{=} \;\;- \LL_{i}{}^{[A_1}{}_{B_1}\mathrm{N}_{B_2}^{A_2}\cdots \mathrm{N}_{B_d}^{A_d]} -\mathrm{N}_{B_1}{}^{[A_1}\LL_{i}{}^{A_2}{}_{B_2}\cdots \mathrm{N}_{B_d}^{A_d]} - \cdots -  \mathrm{N}_{B_1}{}^{[A_1}\mathrm{N}_{B_2}^{A_2}\cdots \LL_{i}{}^{A_d]}{}_{B_d}.
\end{aligned}
\end{equation*}
Note that, since $\LL_{iB}{}^C$ is tangential in the second index and normal in the third, in the right hand side of the above display the top line is proportional to $N^{A_1\cdots A_d}$ and has zero contraction with $N^{B_1\cdots B_d}$, whereas the bottom line is proportional to $N_{B_1\cdots B_d}$ and has zero contraction with $N_{A_1\cdots A_d}$.  Note also that $N^{A_1\cdots A_d}N_{A_1\cdots A_d}=d!$ implies that $N^{A_1\cdots A_d}\nabla_iN_{A_1\cdots A_d}=0$. Contracting the above display with $N_{A_1\cdots A_d}$ therefore gives
\begin{equation*}
\begin{aligned}
\nabla_i N_{B_1\cdots B_d} &= - \LL_{iB_1}{}^{A_1}N_{A_1B_2\cdots B_d} - \LL_{iB_2}{}^{A_2}N_{B_1A_2\cdots B_d}- \cdots -\LL_{iB_d}{}^{A_d}N_{B_1B_2\cdots A_d}\\
& = - d\cdot \LL_{i[B_d}{}^{A_d}N_{B_1B_2\cdots B_{d-1}] A_d},
\end{aligned}
\end{equation*} 
and the result follows by relabeling indices.
\end{proof}

We now invert the relationship between $\nabla N$ and $\LL$ to see that 3
implies 1 in Theorem \ref{key1}.
\begin{thm}\label{prop_norm_grad_norm}
The tractor second fundamental form is given in terms of the
derivative of the tractor normal form by
	\begin{equation}\label{eqn_normal_form_tractor_2ff}
    N^{C A_2 \cdots A_d} \nabla_i N_{B A_2 \cdots A_d} = -
    (d-1)! \cdot \LL_{i B} {}^{C}.
	\end{equation}
\end{thm}
\begin{rem}
The formula \eqref{eqn_normal_form_tractor_2ff} can be obtained from Theorem \ref{thm_grad_normal_form} by contracting both sides of \eqref{eqn_grad_norm} with the tractor normal form on $d-1$ indices and using \eqref{eq-N-proj-N-form}, but it is also easy to obtain directly and so we do this below. One can also obtain \eqref{eqn_grad_norm} from \eqref{eqn_normal_form_tractor_2ff} by taking the appropriate product with a normal form and using \eqref{eq-partial-cont-N-forms} for $k=1$.
\end{rem}

\begin{proof}
Recall from \eqref{eq-N-proj-N-form} that $\mathrm{N}^C_B = \frac{1}{(d-1)!}N^{CA_2\cdots A_d}N_{BA_2\cdots A_d}$. Differentiating both sides of this identity and using \eqref{eqn_der_normal_tractor_proj} on the left hand side we obtain that
\begin{equation}\label{eq-deriv-normal-proj-N-form}
-\mathbb{L}_{iB}{}^C - \mathbb{L}_{i}{}^C{}_B = \frac{1}{(d-1)!}\left( (\nabla_iN^{CA_2\cdots A_d})N_{BA_2\cdots A_d} + N^{CA_2\cdots A_d}\nabla_i N_{BA_2\cdots A_d}\right).
\end{equation}
Note that if we project to the normal tractor bundle in the index $C$ in the above displayed equation, then the left hand side becomes $-\mathbb{L}_{iB}{}^C$ (the $\mathbb{L}_{i}{}^C{}_B$ term projects to zero since $\LL$ is tangential in the second index). We'd like to see that one of the two terms on the right hand side also drops out when we do this. To this end, note that since  $N^{A_1\cdots A_d}\nabla_iN_{A_1\cdots A_d}=0$ it follows that if $N_1^A, \ldots, N_d^A$ are normal tractor fields then $N_1^A \cdots N_d^A\nabla_iN_{A_1\cdots A_d}=0$. In particular, if $N^A$ is a normal tractor field then $N^{C}N^{BA_2\cdots A_d}\nabla_iN_{CA_2\cdots A_d}=0$. In other words, $(\nabla_iN^{CA_2\cdots A_d})N_{BA_2\cdots A_d}$ is tangential in the $C$ index. Thus projecting \eqref{eq-deriv-normal-proj-N-form} to the normal tractor bundle in the index $C$ gives the result.
\end{proof}

Theorem~\ref{key1} follows easily from the above observations:

\begin{proof}[Proof of Theorem~\ref{key1}]
 As remarked above, the equivalence of 1 and 2 follows from Proposition
 \ref{NgradN=L} and Lemma \ref{lem_der_normal_tractor_proj}. The
 equivalence of 1 and 3 follows from Theorems
 \ref{thm_grad_normal_form} and \ref{prop_norm_grad_norm} above. The
 equivalence of 3 and 4 is a trivial consequence of the definition of the Hodge-$\star$ (see \eqref{tstar}),  the fact that the
 volume tractor (see \eqref{tvf}) is parallel, and that $\star\star$ is $\pm$ the identity.
\end{proof}

\begin{rem}\label{remark_results_apply_to_riemannian}
  Many of the results concerning submanifold tractors that were derived in the above sections  used nothing more than
  a local orthonormal basis for the normal bundle and the Gau{\ss}
  formula.  Since the normal tractor bundle is isomorphic to the usual
  normal bundle, and we have a Gau{\ss} formula in both cases, such
  proofs of these results may be repeated mutatis mutandis for the
  Riemannian objects to yield analogous statements and formulae; the one caveat being that one should keep in mind that the connection $\check{\nabla}$ on $\cT\Sigma$ induced from the ambient tractor connection differs in general from the submanifold tractor connection $D$ by \eqref{diff_tractor_connection} when going from the Riemannian to the conformal setting. 
\end{rem}

\subsection{Submanifold invariants}\label{sinvts}

We have seen above that trace-free second fundamental form arises from
using the (ambient) tractor connection acting on $N^{A}_B$. More
generally the tools we have developed can be used to proliferate
submanifold invariants in obvious ways. We sketch some routes.

Let us fix some submanifold $\Sigma$, as usual of dimension $1\leq
m\leq n-1$ and codimension $d$, in a conformal manifold $(M,\cc)$.
Let us write $\ol{\TD}$ for the Thomas operator of the intrinsic
conformal structure $(\Sigma, \cc)$. This is given by the formula
\eqref{tD-formula} except that we couple the tractor connection
$\nabla^{\cT\Sigma}$ to the intrinsic Levi-Civita connection $D_i$
and replace $n$ with $m$. Also in dimensions $m=1,2$ we replace $J$
with $\lj$ as described in Section \ref{lowdim}. In fact it is straightforward to verify this formula
\eqref{tD-formula} provides a conformally invariant operator if we couple
the Levi-Civita connection to {\em any} invariant connection on any vector bundle. (The key point is that verifying its conformal invariance does not involve
commuting any derivatives.)  To exploit this
observation, we will write $\ol{\TD}$ also for the conformally
invariant operator given by the same formula, but where the intrinsic
Levi-Civita connection is coupled to any invariant connection. In
practice here, the latter will be the ambient tractor connection as
well as also the intrinsic tractor connection on $\cT\Sigma$.

For example, along $\Sigma$, $\ol{\TD}_BN^C_D$ is well defined and
conformally invariant, as is $\ol{\TD}_A\ol{\TD}_BN^C_D$. Similarly we
may instead use the normal form $N_{F_1\cdots F_d}$. And this comes to
the main point. The collection
$$
N_{F_1\cdots F_d}, \,\, \ol{\TD}_EN_{F_1\cdots F_d}, \,\,
\ol{\TD}_C\ol{\TD}_EN_{F_1\cdots F_d}, \,\,  \ol{\TD}_B \ol{\TD}_C\ol{\TD}_EN_{F_1\cdots F_d},  \cdots 
$$
embeds the jets of the submanifold into sections of tractor bundles in
a conformally invariant way, up to any desired order. These objects can then be contracted or partially contracted to produce non-linear invariants. For example
$$
(\ol{\TD}^C\ol{\TD}^E N^A_B ) (\ol{\TD}_C\ol{\TD}_E N^B_A) 
$$ is a non-trivial scalar conformal invariant of submanifolds for
most dimensions $m$.
Similarly (for $m\geq 4$) we may form
$$
\ol{W}^{CDEF}(\ol{\TD}_D\ol{\TD}_F N^A_B )(\ol{\TD}_C\ol{\TD}_E N^B_A) ,
$$
where $\ol{W}$ is the $W$-tractor, as defined in \cite{G-advances}, but for the intrinsic geometry of the submanifold $\Sigma$.
In the parlance of invariant theory such obvious
complete contractions are called scalar {\em Weyl invariants}
\cite{BaiEG}. A slightly more subtle  construction uses the
idea of {\em quasi-Weyl invariants}, as in \cite{G-advances}, but this will
  still proceed using the tools developed here. Indeed the results
  from \cite{G-advances} (for conformal invariants) suggest it is likely
  that these techniques would, in a suitable sense, produce almost all
  scalar invariants. 

 The construction of tensor-valued invariants is slightly more
 complicated, and involves ideas as here plus the use of differential
 splitting operators that map (in a conformally invariant way) between
 tensor and tractor bundles (see, e.g., Theorem \ref{normp} and \eqref{eqn_conf_killing_splitting_op} in Section \ref{fi-Sect}). Some applications of these for the construction of hypersurface invariants are given in \cite{BGW}.

\section{Characterising  and generalising mean curvature, and applications}\label{sec_minimal_scales}

Our aim in this section is to show that the tractor formalism leads to natural generalizations of the notions of mean curvature and various conditions on the mean curvature from the Riemannian to the almost-Riemannian setting. The basic idea is that the mean curvature captures (and is captured by) the relation of the scale tractor $I$ of the Riemannian metric to the submanifold tractor bundle, which gives a way of talking about mean curvature that generalizes immediately to the almost-Riemannian setting. In particular, one gets a notion of ``mean curvature tractor'' that is well-defined and smooth up to the conformal infinity. 

To motivate this definition we begin with the Riemannian case.  Note that for a given submanifold $\Sigma$, in a Riemannian manifold
$(M,g)$, its mean curvature vector $H^a=\frac{1}{m}\bg^{ij}\II_{ij}^{a}\in \Gamma(N\Sigma[-2])$ can
equivalently be captured by the {\em mean curvature tractor}
$$
H^A:=\si N_a^AH^a \in \Gamma(\cN)
$$
via the isomorphism of Lemma \ref{Nlem}, where $\si\in \Gamma(\ce[1])$ is the scale giving $g$, meaning $g=\si^{-2}\bg$. With this terminology and notation, we can state the following result. 
\begin{prop}
	\label{mean}
	Let $\Sigma$ be a submanifold in a Riemannian manifold  $(M,g)$.
        Then
        \begin{equation}\label{I.N}
          H^A=  N^A_B I^B.
          \end{equation}
\end{prop}
\begin{proof}
  In the scale of the metric $g=\si^{-2}\bg$ the scale tractor takes
  to form $I^A=\si Y^A+\rho X^A$, for some weight density $\rho$, or
  weight $-1$.  So from formula \eqref{eqn_norm_tractor_proj_formula} in
  Lemma \ref{lem_norm_tractor_proj_formula}, we see that $N^A_BX^B=0$ and 
  \begin{equation}\label{NIexplicit}
N^A_B I^B= \si H^aZ_a^A+ \si (H^bH_b)X^A .
\end{equation}
But, from \eqref{norm_bundle}, this is exactly $H^A$. 

\end{proof}

Thus {\em minimal submanifolds}, meaning those with $H^a=0$, are nicely
captured by orthogonality of the scale tractor to the normal
tractors, as follows (as was known in the case of
hypersurfaces \cite{Gover-Leitner-class}).
\begin{cor}\label{prop_minimal_normal_charac}  In a Riemannian manifold $(M,g)$, let $I$ denote the scale tractor of $g$.
   A submanifold $\Sigma$, of dimension $m$,  is minimal if and only if,
one of the following  equivalent conditions holds
  \begin{enumerate}
    \item $I^A N_A^B=0$ ;
    \item $I^{A_1} N_{A_1 A_2 \cdots
				A_d} = 0$; 
    \item $I^{[A_0} \star \hspace{-2.5pt} N^{A_1 A_2 \cdots A_{m+2}]} = 0$;
    \item $I\in \Gamma(\cN^\perp)$;
      \item $H^AI_A=0$.
  \end{enumerate}
  \end{cor}
\begin{rem}
\emph{(i)} Corollary \ref{prop_minimal_normal_charac} here generalizes Theorem
2 from~\cite{GoverSnell}, as a minimal 1-dimensional submanifold in a
Riemannian manifold is exactly a geodesic. 
\emph{(ii)} Note also that the corollary shows that for a minimal submanifold
$\Sigma$ the ambient scale tractor $I_A$ can, along $\Sigma$, be identified with
 a section of the intrinsic tractor bundle $\cT\Sigma$ via
\eqref{eqn_norm_perp_intrinsic_isom} of Theorem
\ref{thm_norm_perp_intrinsic_isom}.
\end{rem}

It is natural to say that a Riemannian  submanifold has {\em constant mean curvature} (CMC) if the function
$$
\si^2 H^aH_a \in \Gamma(\ce[0]|_{\Sigma}) 
$$
is constant on $\Sigma$, where $\si$ is the scale of the metric
$g$ used to calculate the mean curvature (the reader is cautioned that this is only standard terminology for the case of hypersurfaces; in higher codimension there are other possibilities for the definition of CMC). We will say that a $\Sigma$ has {\em parallel mean curvature} if
$$
\nabla_i^{\perp}  H^b=0,
$$
or equivalently $\nabla_i^{\perp} (\si H^b)=0$; clearly this is stronger than the CMC condition. These notions are also usefully captured by tractors. 
\begin{prop} \label{cmc-prop}
In a Riemannian manifold $(M,g)$, let $I$ denote the scale tractor of $g$. 
A submanifold $\Sigma$:
\begin{itemize}
\item is CMC if and only if
$$
N_{AB}I^AI^B, \qquad \mbox{or equivalently,}\qquad H^AI_A
$$
is constant along $\Sigma$;
\item has  parallel mean curvature if and only if 
  $$
\nabla_i^{\mathcal{N}} H^B =0,\qquad \mbox{or equivalently,}\qquad N^A_B\nabla_i H^B =0.
  $$
  \end{itemize}
  \end{prop}
\begin{proof}
  Continuing in notation and choices of the Proof of Proposition \ref{mean}, the
  first statement follows by contracting $I_A=\si Y^A+ \rho X^A$ into \eqref{NIexplicit}.  The second is immediate from Lemma \ref{nablaNlem}.
\end{proof}
\begin{rem}
  Note that if a submanifold $\Sigma$ in $(M,g)$ has parallel mean curvature, then it is CMC as, in the scale of the metric $g=\si^{-2}\bg$,
  $$
\nabla_i (\si^2H^aH_a)= 2 \si^2 H^a  \nabla_i^\perp H_a.
$$
The converse does not hold. For example in Euclidean 3-space a
round 2-circle (say in the $x-y$-plane) has is parallel mean curvature
(and so is also CMC). But a regular spiral is CMC (by dint of its
invariance under the obvious group action) but does not have parallel
mean curvature.

Note that the stronger notion of parallel mean curvature
$$
\nabla_i (\si H^a)=0
$$
implies CMC also. Thus $|\si H|=\sqrt{\si^2 H^aH_a}$ is constant
and $\si H^a=|\si H|\hat{n}^a$ for some a unit normal along $\Sigma$
that must be parallel. Such a parallel unit normal means that the
acceleration of any curve in $\Sigma$ is orthogonal to $\hat{n}$, so
the second fundamental form and $H^a$ are orthogonal to $\hat{n}$. But
the latter obviously means $H^a=0$.
  \end{rem}

\newcommand{\ul}[1]{\underline{#1}}
Part of the importance of Proposition \ref{mean}, Corollary
\ref{prop_minimal_normal_charac}, and Proposition \ref{cmc-prop}, is
that in means that these quantities and notions at once extend to the
setting of almost-Riemannian manifolds (as defined in Section
\ref{scalet}). For emphasis we make this a definition.
\begin{defn}
	\label{defn_generalized_minimal}
	\emph{Let $(M, \bm{c}, I)$ be an almost-Riemannian manifold with degeneracy locus $\mathcal{Z}(\sigma)$. We say that an embedded submanifold $\Sigma$ of $M$ is an \emph{almost-Riemannian submanifold} of $(M, \bm{c}, I)$ if $\Sigma \setminus \mathcal{Z}(\sigma)$ is dense in $\Sigma$. We say that such a submanifold $\Sigma$ is, respectively, \emph{CMC} or has  \emph{parallel mean} \emph{curvature}  (in the almost-Riemannian sense) if one of the conditions displayed in Proposition \ref{cmc-prop} holds.
        Similarly we say that it is \emph{minimal} (in the almost-Riemannian sense) if any one of the
        equivalent conditions of Corollary
        \ref{prop_minimal_normal_charac} holds.}
\end{defn}
\noindent For an almost-Riemannian manifold $(M, \cc, I)$, the zero
locus $\mathcal{Z}(\si)$, of $\si:=X^AI_A$, is (closed and) nowhere
dense. Thus, by continuity, the notions in the definition extend those on $M\setminus \mathcal{Z}(\si)$, as in the following proposition.
\begin{prop}
  Let $(M, \bm{c}, I)$ be an almost-Riemannian manifold and $\si:=X^AI_A$.
  Then an almost-Riemannian submanifold $\Sigma$ is minimal, CMC, or mean curvature parallel in the sense of Definition \ref{defn_generalized_minimal} if and only if satisfies the corresponding condition (in the non-generalised sense) on $M \backslash \cZ(\sigma)$ for the metric 
  $g:= \sigma^{-2} \confmet$.
  \end{prop}

This perspective enables an easy recovery of the following result, which is well-known from other perspectives.

\begin{prop}
	On a conformally compact manifold, any  minimal
        submanifold that extends smoothly to the boundary meets the boundary
        orthogonally.
\end{prop}

\begin{proof}
	On a conformally compact manifold $\partial M = \cZ(\sigma)$,
\begin{equation}
	I^A
	|_{\partial M} = (\nabla^a \sigma) Z^A _a - \frac{1}{n} \Delta \sigma X^A,
\end{equation}
	and  $\nabla_a \sigma$ is
	nowhere-zero along the boundary. See Section \ref{scalet}.
        	Thus if $\Sigma$ meets $\partial M$ then we have
        $$
I_A N_B^A= N^a_B \nabla_a \sigma \quad \mbox{along} \,\, \Sigma,
        $$
(using  \eqref{eqn_norm_tractor_proj_formula}) and so $\Sigma$ minimal, meaning $I_A N_B^A=0$, implies $N^a_B \nabla_a \sigma =0$ and hence
$$
N^a_b \nabla_a\si=0.
$$
That is   $\nabla_a
		\sigma$ (the conormal to the boundary $\partial M$) is orthogonal to the normal projector 
	 of $\Sigma$.  
	\end{proof}

Suppose now that $(M,\bm{c}, I)$ is an almost-Einstein manifold.
If $\Sigma$ is minimal then, as observed above, $I_A$ may be
identified with a submanifold tractor.
Since $I_A$ is parallel for the standard tractor connection, and $I_A$ is a
submanifold tractor, $I_A$ is also parallel for the connection
$\check{\nabla}$:
\begin{equation*}
	\check{\nabla}_i I_J = \Pi^A _J \nabla_i
	\left( \Pi^K _A I_K \right) = \Pi^A _J \nabla_i I_A = 0,
\end{equation*}
as defined in \eqref{checked_connection}.
Therefore, from the decomposition
\eqref{diff_tractor_connection}, one sees that $I_J$ is
parallel for the submanifold tractor connection if, and only if, $\mathsf{S}_i
	^{J} {}_K I^K = 0$.

Choosing a background scale to split the tractor bundles, we have that
\begin{align*}
	\mathsf{S}_i {}^J {}_K I^K & = \cF_{ij}
	\left( Z^{Jj} X_K - Z^j _K X_J \right) \left( \sigma Y^K + \nabla_k \sigma
	Z^{Kk} - \frac{1}{n} \left( \Delta + \mathsf{J} \sigma \right) X^K \right)                  \\
	                           & = \cF_{ij} \left( \sigma Z^{Jj} - \nabla^j \sigma X^J \right).
\end{align*}
Recall that on almost-Einstein manifolds the 1-jet $j^1 \sigma$ can only  vanishes at isolated points (see the
discussion of Section~\ref{scalet}, and references therein, for details). 
Therefore away from these points we must have $\cF_{ij} = 0$, and then also at
those points by continuity.
Thus we have  the following result. 
\begin{prop} 
	\label{thm_einstein_fialkow_zero} Let $\Sigma \hookrightarrow M$ be a minimal
	almost-Riemannian submanifold of an almost-Einstein manifold $(M, \bm{c}, I)$. Then $\sigma=X^A I_A$ defines an almost-Einstein scale on $\Sigma$ if, and only if, $\cF_{ij} = 0$.
\end{prop}

By \eqref{tractor2ff_XZ} distinguished submanifolds are necessarily totally umbilic. Thus, in a Riemannian manifold, if $\Sigma$ is distinguished and minimal then it is totally geodesic. This has a converse if the Riemannian manifold is Einstein; a totally geodesic submanifold in an Einstein manifold is both minimal and distinguished. Moreover, if we say that an almost-Riemannian submanifold in an almost-Riemannian manifold is \emph{totally geodesic} (in the almost-Riemannian sense) if it is minimal (in the sense that $N_A^BI^A=0$) and totally umbilic, then the proof generalises without change to the almost-Einstein setting:

\begin{prop}
	Let $(M, \bm{c}, I)$ be an almost-Einstein manifold, and $\Sigma$ an almost-Riemannian submanifold. If
        $\Sigma$ is totally geodesic (in the almost-Riemannian sense), then
        $\Sigma$ is a distinguished submanifold.
\end{prop}

\begin{proof}
  We must show any of the equivalent conditions of
  Theorem~\ref{key1}.
	That $\Sigma$ is minimal implies that $H^B = 0$ on $\Sigma$. Since also
$\Sigma$ is totally umbilic it follows that on $\Sigma\setminus
\mathcal{Z}(\si)$ we have that it is totally geodesic, and hence
$\nabla_i N_{a_1 a_2 \cdots a_d} = 0$.
  The almost-Einstein condition implies
  that $N_{b a_2 \cdots a_d} P_i {}^b = 0$ on $M\setminus
  \mathcal{Z}(\si)$, where we calculate in the scale of the metric $g=\si^{-2}\bg$, with $\si:=X^AI_A$. Combining these observations and using formula \eqref{eqn_grad_norm} we have that
 $\nabla_i N_{A_1 A_2 \cdots A_d} = 0$ on $M\setminus
\mathcal{Z}(\si)$. But then by continuity $\nabla_i N_{A_1 A_2 \cdots A_d} = 0$ on $\Sigma$, as $Z(\si)$ is nowhere dense.
\end{proof}

Thus if our ambient space is almost-Einstein, for submanifolds that totally geodesic (in the generalised/almost-Riemannian sense) first
integrals may be proliferated using Corollary \ref{fi-cor} (and the
theory to be developed in Section \ref{fi-Sect}).  These conserved
quantities will extend to/across singularity sets of these geometries
where they exist.

\section{Distinguished submanifolds and conformal circularity} \label{sec4}

We fix some notational conventions for this chapter.  We will denote
by $\gamma$ a smooth curve in a conformal manifold $(M, \bc)$.
By this we will here mean a smooth, regular curve $\gamma:I \to M$
(for some interval $I$). We will often identify
$\gamma$ with its image and we typically assume that this an embedded
submanifold.

The symbols $u^b$ and $a^b$ will denote, respectively, the velocity
and acceleration the curve $\gamma$, so $a^b = u^a \nabla_a u^b$. Note that the acceleration $a^b$ depends on a choice of metric $g\in \bc$ and is not conformally invariant; it is easy to check that if $\hat{g}=\Omega^2g$ then $\hat{a}^b = a^b - u_au^a \Upsilon^b + 2u^a\Upsilon_a u^b$, where $\Upsilon_a = \Omega^{-1}\nabla_a\Omega$. We
also define $ u:= \sqrt{\confmet_{ab} u^a u^b} \in
\Gamma(\cE[1]|_{\gamma})$.  For some connection $\nabla$, we will also
use the notation $\frac{d^\nabla}{dt}$, or $\frac{d}{dt}$ when the
meaning is clear by context, to mean $u^a \nabla_a$.
The connection
$\nabla$ may be a Levi-Civita connection or the standard tractor
connection; this should be unambiguous from context.  Sometimes we will prefer to work with \emph{weighted} versions of the
velocity and acceleration vectors.  These will be denoted by
$\mbf{u}^b := u^{-1}u^b \in \Gamma(\cE^b [-1]|_{\gamma})$ and $\mbf{a}^b := \mbf{u}^c \nabla_c \mbf{u}^b \in \Gamma(\cE^b [-2]|_{\gamma})$
respectively.  

\subsection{Background on conformal circles}\label{sec_circles}

A smooth curve $\gamma$ is said to be a
\emph{(projectively parametrised) conformal circle} if, with respect to some (equivalently any) choice of $g\in \bc$, its velocity
and acceleration satisfy \cite{Bailey1990a} 
\begin{equation}\label{eqn_proj_param_conf_circ}
  u^c \nabla_c a^b = u^2  u^c P_c {}^b + 3 u^{-2} \left( u_c a^c \right) a^b - \frac{3}{2} u^{-2} \left( a_c a^c \right) u^b - 2 u^c u^d P_{cd} u^b,
\end{equation}
where $u^2 =u \cdot u = u_a u^a$ here should be understood to be \emph{unweighted}. Equation \eqref{eqn_proj_param_conf_circ} is a third order, conformally invariant analog of the geodesic equation in Riemannian geometry and solutions of \eqref{eqn_proj_param_conf_circ} are sometimes referred to as \emph{conformal geodesics} \cite{FriedSchmidt}. As with the geodesic equation $u^c\nabla_c u^b =0$, equation \eqref{eqn_proj_param_conf_circ} can be broken up into its tangential and normal components along the curve and any curve in $M$ can be parametrized so that the tangential part of \eqref{eqn_proj_param_conf_circ} holds; such a parametrisation is determined up to the action of $\mathrm{PSL}(2,\mathbb{R})$ and a curve with such a parametrisation is said to be \emph{projectively parametrised} \cite{Bailey1990a}. (The existence of such a parametrisation, and likely the notion of conformal circles also, goes back to \'Elie Cartan; see \cite{cartan1923espaces}. For early treatments of conformal circles, see \cite{Ferrand,FriedSchmidt,Muto1939,Ogiue1967,Schouten1954,Yano1938,Yano1957}. In the literature conformal circles are sometimes taken to be parametrised by arclength with respect to chosen metric $g$ rather than projectively parametrised, and in this case they satisfy a slightly different equation; see, e.g. \cite[Chapter VII, $\S 2$]{Yano1957}.) Asking only that the normal (to the curve) part of \eqref{eqn_proj_param_conf_circ} holds gives a notion of conformal circles that does not depend on the parametrisation, and any such curve can be reparametrised so that \eqref{eqn_proj_param_conf_circ} holds.

Note that a geodesic for a metric $g\in \bc$ need not be a conformal circle; indeed, this could not be the case since any curve $\gamma$ in $M$ is locally an affinely parametrised geodesic for some choice of metric $g\in \bc$ (see Remark \ref{rem_minimal_sc_1d}). Following this line of thought, however, one sees as a direct consequence of \eqref{eqn_proj_param_conf_circ} that a curve $\gamma$ is a projectively parametrized conformal circle if and only if there is a metric $g\in \bc$ with respect to which $\gamma$ is an affinely parametrised geodesic and $u^cP_c{}^b=0$ \cite{Bailey1990a}. Note also that in the special case where one has an Einstein metric $g$ in the conformal class it follows from \eqref{eqn_proj_param_conf_circ} that geodesics for $g$ are conformal circles, though the unit speed parametrization is not a projective parametrization  except in the Ricci flat case.

The notion of conformal circles arises naturally from the Cartan geometric description of conformal structures in dimensions $n\geq 3$ (and M\"obius conformal structures in two dimensions) and as such it is natural that they can be simply described using tractor calculus (the corresponding calculus of associated bundles). With this in mind we now introduce some important tractor fields associated to the curve $\gamma$.
Recall that the canonical tractor $X^B$ can be viewed as a section of $\cE^B [1]$.
Hence $u^{-1} X^B$ is an unweighted tractor along the curve and so the tractor covariant derivative of $u^{-1} X^B$ along the curve is well defined (conformally invariant).
Following \cite{BEG}, we define 
\begin{equation}\label{eqn_velocity_tractor}
  U^B := u^a \nabla_a \left( u^{-1} X^B \right) 
\end{equation}
and 
\begin{equation}\label{eqn_accel_tractor}
  A^B := u^a \nabla_a U^B,
\end{equation}
which we call the \emph{velocity} and \emph{acceleration tractors} respectively.
Explicitly, one has 
\begin{equation}\label{eqn_velocity_tractor_slots}
  U^B 
  \overset{g}{=}
  \begin{pmatrix}
    0 \\
    u^{-1} u^b \\ 
    - u^{-3} \left( u_c a^c \right) 
  \end{pmatrix}
\end{equation}
and 
\begin{equation}\label{eqn_accel_tractor_slots}
  A^B 
  \overset{g}{=}
  \begin{pmatrix}
    -u \\
    u^{-1} a^b - 2u^{-3} (u_c a^c) u^b \\ 
    -u^{-3} \left( u_c \frac{d a^c}{dt} \right) - u^{-3} a_c a^c + 3u^{-5} {\left( u_c a^c\right)}^2 - u^{-1} P_{cd} u^c u^d 
  \end{pmatrix}.
\end{equation}

It is easily checked that
\begin{equation}
U^BU_B=1,\qquad U^BA_B =0
\end{equation}
and that 
\begin{equation}
A^BA_B = 3u^{-2}a_ba^b + 2u^{-2}u_b u^c\nabla_c a^b - 6 u^{-4}(u_ca^c)^2 + 2P_{ab}u^au^b.
\end{equation}
Consequently, a curve $\gamma:I \to M$ is projectively parametrised if, and only if, $A^BA_B=0$. It was then shown in \cite{BEG} that a  projectively parametrised curve  $\gamma:I \to M$ is a conformal geodesic if, and only if, 
\begin{equation}
\frac{d^{\nabla}A^B}{dt} = 0.
\end{equation}

More recently, it was shown by the second and third named authors and Taghavi-Chabert \cite{GST} that a curve $\gamma$ is an unparametrised conformal circle if, and only if, $d^{\nabla}A^B/dt$ is zero modulo $U^B$ and $X^B$; given the definitions of the velocity and acceleration tractors this is equivalent to saying that $\gamma$ is an unparametrised conformal circle if, and only if, the $3$-tractor
\begin{equation}
\Phi^{ABC}:=6u^{-1}X^{[A}U^B A^{C]} 
\end{equation}
is covariantly constant along $\gamma$. To see this we note the following: Firstly,
\begin{equation}\label{Phi-in-slots}
\Phi^{ABC} = 6\textbf{u}^c X^{[A}Y^BZ^{C]}_c + 6\textbf{u}^b \textbf{a}^c X^{[A}Z^B_{b}Z^{C]}_c,
\end{equation}
where $\textbf{u}^c = u^{-1}u^c$ (so that $\bg_{ab}\textbf{u}^a\textbf{u}^b=1$) and $\textbf{a}^c = \textbf{u}^b\nabla_b\textbf{u}^c = u^{-2}a^c - u^{-4}(u_ba^b)u^c$. It is then easy to show that
\begin{equation}\label{deriv_trac_3form}
u^d\nabla_c \Phi^{ABC} = 6 \left(\textbf{u}^d\nabla_d \textbf{a}^c - \textbf{u}^dP_d{}^c\right)\textbf{u}^bX^{[A}Z^B_{b}Z^{C]}_c.
\end{equation}
On the other hand, the requirement that the normal (to the curve) part of \eqref{eqn_proj_param_conf_circ} holds can be written in terms of the weighted velocity and acceleration as \cite[Lemma 4.9]{GST}
\begin{equation}\label{unparam_circle}
\left(\textbf{u}^d\nabla_d \textbf{a}^{[b}\right)\textbf{u}^{c]} = \textbf{u}^dP_d{}^{[b}\textbf{u}^{c]},
\end{equation}
from which the claim follows. 

Note that from \eqref{Phi-in-slots} one can easily check that $\Phi^{ABC}\Phi_{ABC}=6$, and that $\Phi_{ABC}N^A=0$ for any section $N^A$ of the normal tractor bundle to $\gamma$ (the easiest way to see the latter is to compute in a minimal scale $g$ for $\gamma$, equivalently, a scale for which $\textbf{a}^c=0$). It follows that \
\begin{equation}\label{Phi-is-star-N}
\Phi_{ABC} = \star N^{\gamma\hookrightarrow M}_{ABC},
\end{equation} the Hodge star of the tractor normal form of $\gamma$ as a submanifold of $M$. 

This observation combined with the result from \cite{GST} described in the preceding paragraph shows that Theorem \ref{conf_circ_2ff_zero} follows from Theorem \ref{key1}.

\subsection{Weak conformal circularity of submanifolds}\label{subsec_weak_conf_circ}

With this background established we begin our discussion of conformal circularity.
\begin{defn}
\emph{  A submanifold $\Sigma$ is \emph{weakly conformally circular} if
  any $M$-conformal circle, whose 2-jet at a point lies in $\Sigma$,
  remains in $\Sigma$.  That is, if $\gamma$ is an $M$-conformal
  circle whose 2-jet at some point $p$ lies in $\Sigma$ (with
  $\gamma(0) = p$), then $\gamma(t) \in \Sigma$ for all $t$.}
  \end{defn}
	
In the model case of the conformal sphere (see Section \ref{tractor_calc})  both conformal circles and
totally umbilic submanifolds arise (by ray projectivisation) from the
intersections of suitable linear subspaces with the forward null cone. 
Thus the conformal circles are
all given by (transverse) intersections of totally umbilic
submanifolds, and a submanifold of the conformal
sphere is weakly conformally circular if and only if it is totally
umbilic. It is natural to ask to what extent these facts generalize to
the curved setting. In this case one quickly sees that the condition
of being totally umbilic must be replaced with the stronger condition
of being a distinguished submanifold (in the conformally flat case for submanifolds of dimension greater than one the
vanishing of $\mathring{\II}$ is equivalent to the vanishing of the
tractor second fundamental form $\LL$, but this is no longer true in
general conformal manifolds; see the examples in Section \ref{sec-examples}). If
two distinguished submanifolds intersect transversally in a
$1$-dimensional submanifold $\gamma$ then, since the wedge product of
the two corresponding normal tractors must be parallel along $\gamma$,
$\gamma$ must be a conformal circle (but, due to the sparsity of
distinguished submanifolds in the curved setting, conformal circles no
longer arise this way in general). An extension of this idea shows that a
submanifold is weakly conformally circular if, and only if, it is
distinguished. That is the content of the following theorem.

\begin{thm}\label{thm_weakly_conformally_circular}
  A submanifold $\Sigma \hookrightarrow M$ is weakly conformally circular if, and only if, $\LL_{i J} {}^C = 0$.
\end{thm}

\begin{proof}
	A one dimensional submanifold $\Sigma$ is weakly conformally circular if, and only if, it is a conformal circle. Thus in the one dimensional case the	result follows immediately from Theorem \ref{conf_circ_2ff_zero}, which states that a curve is a conformal circle if, and only if, when viewed as a submanifold its tractor second fundamental form $\LL$ vanishes.
		
		Suppose now that $\Sigma$ has dimension at least 2, and is weakly conformally circular.
  Let $\gamma$ be an $M$-conformal circle whose 2-jet at $p \in \Sigma$ lies in $\Sigma$.
  Then by assumption $\gamma$ remains in $\Sigma$.
  We need to introduce some notation.
  Let
  \begin{itemize}
    \item $N^\GtM _{A_1 A_2 \cdots A_d}$ be the normal form of $\Sigma \hookrightarrow M$,
    \item $N^\gtM _{A_1 A_2 \cdots A_{n-1}}$ be the normal form of $\gamma \hookrightarrow M$, and 
    \item $N^\gtG _{A_1 A_2 \cdots A_{m-1}}$ be the normal form of $\gamma \hookrightarrow \Sigma$, where we are identifying the tractor bundle of $\Sigma$ with $\cN^\perp$, and hence this form is a section of $\Lambda^{m-1} \cT M|_\Sigma$.
  \end{itemize}
  First we note some important relations between these various normal forms.
  First, since the curve $\gamma$ lies in the submanifold $\Sigma$, we have 
  \begin{equation}
    N_\gtG ^{A_1 B_2 \cdots B_{m-1}} N^{\GtM} _{A_1 A_2 \cdots A_d} = 0.
  \end{equation}
  Second, by using the discussion surrounding equation~\eqref{eqn_normal_form_wedge}, one can easily show that
  \begin{equation}
    N^\gtG _{A_1 A_2 \cdots A_{m-1}} \wedge N^\GtM _{A_m \cdots A_{n-1}} = N^\gtM _{A_1 A_2 \cdots A_{n-1}}.
  \end{equation} 
  Finally, since $\gamma$ is an $M$-conformal circle, it follows from Theorem~\ref{conf_circ_2ff_zero} that
  \begin{equation*}
    u^i \nabla_i N^\gtM _{A_1 A_2 \cdots A_{n-1}} = 0.
  \end{equation*}
  Therefore, using the above, 
  \begin{equation*}
    \left( u^i \nabla_i N^{\gtG} _{[A_1 A_2 \cdots A_{m-1}} \right) N^{\GtM} _{A_m \cdots A_{n-1} ]} + N^{\gtG} _{[A_1 A_2 \cdots A_{m-1}} \left( u^i \nabla_i N^{\GtM} _{A_m \cdots A_{n-1}]} \right) = 0, 
  \end{equation*}
  and hence 
  \begin{equation*}
    \left( u^i \nabla_i N^{\gtG} _{[A_1 A_2 \cdots A_{m-1}} \right) N^{\GtM} _{A_m \cdots A_{n-1} ]} + N^{\gtG} _{[A_1 A_2 \cdots A_{m-1}} \left( -d \cdot u^i \LL_{i A_{n-1}} {}^{A_0} N^{\GtM} _{A_{m} \cdots A_{n-2}] A_0} \right) = 0
  \end{equation*}
  by Theorem~\ref{thm_grad_normal_form}, where $\LL$ is the tractor second fundamental form of $\gamma$. Since the downstairs tractor index on $\LL_{iA_{n-1}}{}^{A_0}$ is ``tangential to $\Sigma$'' it is easy to see that the two terms on the left hand side of the above displayed equation lie in complementary subspaces of the bundle of tractor $(n-1)$-forms (the first term is in the ideal generated by $N^{\GtM}$ and the second term is in the orthogonal complement to this ideal) and hence both terms must vanish. Thus, in particular,
\begin{equation*}
N^{\gtG} _{[A_1 A_2 \cdots A_{m-1}} u^i \LL_{i A_{n-1}} {}^{A_0} N^{\GtM} _{A_{m} \cdots A_{n-2}] A_0} = 0.
\end{equation*}
Contracting the above display with $N_{\GtM}^{A_m\cdots A_{n-2}B}$ (cf.\ the proof of Theorem~\ref{prop_norm_grad_norm}) then gives
\begin{equation*}
N^{\gtG} _{[A_1 A_2 \cdots A_{m-1}} u^i \LL_{i A_{n-1}]} {}^{B}=0,
\end{equation*}
which is equivalent to 		
\begin{equation}\label{L-PI}
    u^i \LL_{i A}{}^{B} \Pi^{\gtG} {}^{A} _{C} = 0,
  \end{equation}
  where $\Pi^{\gtG}$ is the projector onto the rank 3  tractor bundle of the 1-manifold $\gamma$, viewed as a subbundle of the ambient tractor bundle along $\gamma$. 
  Therefore it follows that 
  \begin{equation*}
    u^i \LL_{i A} {}^{B} U^{A} = 0,
  \end{equation*}
  where $U^{A}$ is the velocity tractor of the curve (note that $U^{A}$ may be viewed as a section of the intrinsic standard tractor bundle of $\gamma$; one can easily check this by working in a minimal scale for $\gamma$, where $u_ba^b=0$).
  Using Theorem~\ref{tractor2ff_XZ} we now see that, in particular, $\tfII_{ij} {}^c u^i u^j = 0$.
  But the above must hold for any $M$-conformal circle $\gamma$, and hence $\tfII_{ij} {}^c u^i u^j = 0$ for all $u^i \in \Sigma(\cE^i)$, whence $\Sigma$ is totally umbilic by polarization.

Since we have already seen that $\mathring{\II}_{ij} {}^c = 0$, it suffices to show that $\mathrm{N}^c _b \left( P_i {}^b - \nabla_i H^b \right) = 0$.
Returning to \eqref{L-PI}, if we contract this with $Y^{A}$ (or $\Pi_I^AY^I$, cf.\ \eqref{subm_vs_amb_XYZ}) this gives
$$
u^i\mathrm{N}^c _b \left( P_i {}^b - \nabla_i H^b \right)Z^{B}_c +  u^i H_b \left( P_i {}^b - \nabla_i H^b \right) X^{B}=0,
$$ by Theorem~\ref{tractor2ff_XZ}, since the other slots of
$\mathbb{L}$ have already been shown to be zero.  Since, again, this
must hold for all $u^i \in \Gamma(\cE^i)$ we obtain the result.

For the converse, let us consider  a curve
$\gamma$  in $\Sigma$ that satisfies
\begin{equation}\label{checkitout}
u^i\check{\nabla}_i N^\gtG _{A_1 A_2 \cdots A_{m-1}} = 0 \qquad \textrm{and} \qquad A^BA_B=0
\end{equation}
where we have used the connection $\check{\nabla}$.
Then, 
$$
\left( u^i \nabla_i N^{\gtG} _{[A_1 A_2 \cdots A_{m-1}} \right) N^{\GtM} _{A_m \cdots A_{n-1} ]}
=0 .
$$
Suppose now that  $\LL_{i J} {}^C = 0$. Then, 
$$
    \left( u^i \nabla_i N^{\gtG} _{[A_1 A_2 \cdots A_{m-1}} \right) N^{\GtM} _{A_m \cdots A_{n-1} ]} + N^{\gtG} _{[A_1 A_2 \cdots A_{m-1}} \left( -d \cdot u^i \LL_{i A_{n-1}} {}^{A_0} N^{\GtM} _{A_{m} \cdots A_{n-2}] A_0} \right) = 0 ,
$$
    and so
    $$
u^i\nabla_i N^\gtM _{A_1 A_2 \cdots A_{n-1}} = 0.
$$
That is, if $\gamma$ in $\Sigma$ satisfies \eqref{checkitout} and
$\LL_{i J} {}^C = 0$ then $\gamma$ is a conformal circle for
$(M,\cc)$. Now, it is straightforward to check that a curve satisfying \eqref{checkitout} that is further projectively parametrised with respect to the conformal structure on $M$ (meaning its $M$-acceleration tractor satisfies $A^BA_B=0$) is determined by its 2-jet in $\Sigma$ at any point on its path. This follows because, by construction, a projectively parametrised curve satisfying \eqref{checkitout} is characterised by a third order ordinary differential equation  in any local coordinate chart (analogous to how the $(\Sigma,\cc_\Sigma)$-conformal circle equation is equivalent to the requirement that the curve be $\Sigma$-projectively parametrised and satisfy $ u^iD_i N^\gtG _{A_1 A_2 \cdots A_{m-1}} = 0 $, and the $(M,\bc)$-conformal circle equation is equivalent to the requirement that the curve be $M$-projectively parametrised and satisfy $u^a\nabla_a N^\gtM _{A_1 A_2 \cdots A_{n-1}}=0$; cf.\ Remark \ref{rem_weak_conf_circ} below).
Now, suppose $\LL_{iJ}{}^C=0$. Then, given any $2$-jet of a curve in $\Sigma$, the corresponding $M$-projectively parametrised solution of \eqref{checkitout} is also an $M$ conformal circle. Moreover, since a conformal circle in $(M,\cc)$ is determined by its 2-jet at any point on its path, all conformal circles corresponding to two jet initial data lying in $\Sigma$ arise this way; in particular, all such curves lie in $\Sigma$. That is, $\Sigma$ is weakly conformally circular.
\end{proof}

\begin{rem}\label{rem_weak_conf_circ}
Here we give a version of the proof that $\LL_{i J} {}^C = 0$ implies
weak conformal circularity that avoids the use of tractor
calculus. (This is along the lines of a proof of a similar result in
\cite{BelgunFlorin2015}, to an extent the tractor picture provides a conceptual basis for the idea.) For convenience, we work in a minimal scale
$g$. Rather than considering curves $\gamma$ in $\Sigma$ solving
\eqref{checkitout} we will consider the curves $\gamma$ in $\Sigma$
solving the ``adapted conformal circle equation''
  \begin{equation}\label{eqn_proj_param_adapted_conf_geo}
      \frac{d^D a^j}{dt} = u^2 \cdot u^i P_i {}^j + 3 u^{-2} \left( u_k a^k \right) a^j - \frac{3}{2} u^{-2} \left( a_k a^k \right) u^j - 2 u^k u^l P_{kl} u^j,
  \end{equation}
  where as usual $P_i {}^j$ and $P_{kl}$ denote the restriction of the ambient Schouten tensor to the intrinsic tangent and cotangent bundles, and $\frac{d^D}{dt}$ denotes $u^i D_i$ where $D$ is the intrinsic Levi-Civita connection for the pullback $g_{\Sigma}$ of the ambient scale $g$. 
  We say that $\gamma$ is an ``adapted conformal circle'' if it satisfies this equation. (That equation~\eqref{eqn_proj_param_adapted_conf_geo} is equivalent to~\eqref{checkitout} should be clear from what follows, but we do not need this for the argument given in this remark.)
  Note that equation~\eqref{eqn_proj_param_adapted_conf_geo} is a third order ODE on $\Sigma$, and therefore the initial value problem with given $2$-jet  initial data has a unique solution on $\Sigma$ for some interval centered at $0$.
  This solution may also be viewed as a curve in $M$, and one may ask whether it solves a related ODE there.
  Since the $2$-jet of $\gamma$ is initially tangential and $\LL_{i J} {}^C = 0$ implies in particular that $\mathring{\II}_{ij} {}^c = 0$, it follows that in our minimal scale $\Pi^b _j u^j = u^b$, $\Pi^b _j a^j = a^b$ and $\Pi^b _j \frac{d^D a^j}{dt} = \frac{d^\nabla a^b}{dt}$ (the last two identities being consequences of the Gauss formula with $\II=0$). Thus $u_k a^k = u_c a^c$, $a_k a^k = a_c a^c$ and $u^k u^l P_{kl} = u^c u^d P_{cd}$. Moreover, $\LL_{i J} {}^C = 0$ also implies (again for the minimal scale) that $\mathrm{N}^c _b P_i {}^b = 0$, and thus one easily sees that, as a curve in $M$, $\gamma$ satisfies
  \begin{align*}
    \frac{d^\nabla a^b}{dt} &= u^2 \left( \Pi^b _j u^i P_i {}^j \right) + 3 u^{-2} \left( u_c a^c \right) a^b - \frac{3}{2} u^{-2} \left( a_c a^c \right) u^b - 2 u^c u^d P_{cd} u^b\\
		&=u^2 \cdot u^c P_c {}^b + 3 u^{-2} \left( u_c a^c \right) a^b - \frac{3}{2} u^{-2} \left( a_c a^c \right) u^b - 2 u^c u^d P_{cd} u^b,
  \end{align*}
  which is exactly the (projectively parametrized) $M$-conformal circle equation.
  So if $\LL_{i J} {}^C = 0$, and $\gamma$ is an adapted conformal circle, then it is an $M$-conformal circle, and by the uniqueness of solution to an initial value problem, the curve $\gamma$, which lies in $\Sigma$, \emph{is} the unique $M$-conformal circle with the given initial conditions.
  Hence any $M$-conformal circle whose $2$-jet at a point $p \in \Sigma$ is tangential will remain in $\Sigma$, i.e.\ $\Sigma$ is weakly conformally circular.

  \end{rem}

For his approach to submanifold circularity, 
Belgun \cite{BelgunFlorin2015} introduced a conformal invariant  $\mu \in \Gamma(T^* \Sigma \otimes N \Sigma[-2])$ given by 
\begin{equation}
	\label{eqn_mu_tensor}
	\mu_i {}^c := \mathrm{N}^c _b \left( P_i {}^b - \nabla_i H^b + \frac{1}{m-1}
	D^j \mathring{\II}_{ij} {}^b \right),
\end{equation}
when $m\neq 1$ and where
the intrinsic Levi-Civita connection $D$ is coupled to the normal connection
$\nabla^\perp$ (Belgun terms this the \emph{mixed Schouten-Weyl tensor} since the main term is the tangential-normal part of the ambient Schouten tensor, which he calls the Schouten-Weyl tensor).
In the $m=1$ case (i.e.\ when $\Sigma$ is a curve) Belgun defines $\mu$ to be the \emph{conformal curvature} of the curve $\Sigma$, 
\begin{equation}\label{eq:curve-conformal-curvature}
\mu_i {}^c := \mathrm{N}^c _b \left( P_i {}^b - \nabla_i H^b\right) \quad \text{when } m=1.
\end{equation}
When $\dim \Sigma =1$, $\mathring{\II}_{ij} {}^c = 0$, so that $\mathrm{N}^c _b \left( P_i {}^b - \nabla_i H^b\right)$ can be obtained from $\LL$ by an invariant projection and therefore must be conformally invariant. The conformal invariance of $\mu$ in the $m\neq 1$ case can be shown by direct calculation, or by substituting the Weyl-Schouten decomposition of the ambient curvature tensor into the Codazzi equation to obtain \cite{SnellPhD}
\begin{equation}\label{eq:mu-for-m-geq-2}
	\mu_i {}^c = \frac{1}{m-1} W_{ij} {}^{dj}
	\mathrm{N}^c _d.
\end{equation}
From the above displayed formula it is also clear that $\mu=0$ when $\Sigma$ is a hypersurface (since $W_{ij} {}^{dj}\mathrm{N}^c _d = - W_{ia}{}^d{}_b\mathrm{N}^{ab}\mathrm{N}^c _d$ and the normal bundle has rank $1$).
In \cite{BelgunFlorin2015} Belgun
characterises weakly conformal circular submanifolds (termed \emph{weakly conformal geodesic} in \cite{BelgunFlorin2015}) as those for
which $\mathring{\II}=0$ and $\mu=0$. Inspecting \eqref{eqn_tractor2ff_XZ}, in Theorem \ref{prop_alt_tractor_2ff}, one sees immediately that this is equivalent to the vanishing of $\LL$. 
\begin{prop}
	\label{lem_mu_zero_2ff_zero}
	Let $\Sigma \hookrightarrow M$ be a submanifold in a conformal manifold with
	$\LL_{i J} {}^C$ its tractor second fundamental form.
	Then $\LL_{i J} {}^C = 0$ if, and only if, $\mathring{\II}_{ij} {}^c = 0$ and
	$\mu_i {}^c = 0$.
\end{prop}

In fact the invariant $\mu$ arises naturally from $\mathbb{L}$. The
projecting part of $\mathbb{L}$ is necessarily invariant and this is
$\mathring{\II}$.  An obvious question is whether conversely
$\mathbb{L}$ is then image of a natural linear differential operator acting
$\mathring{\II}$, in which case they would be equivalent. (For example
as the tractor curvature is the image of the Weyl curvature in
dimensions at least 4.)  This leads us to the following lemma, which exhibits a conformally invariant differential operator between the relevant bundles:

\begin{lem}
  There is an invariant map
  $\mathbb{M} : S^2 _0 T^* \Sigma \otimes N \Sigma \to T^* \Sigma
		\otimes \cT^* \Sigma \otimes \cN$.
	Written in tractor projectors, this takes the form
	\begin{equation}
		\label{eqn_M_operator}
		\begin{split}
		  \omega_{ij} {}^c \mapsto
                  \mathbb{M} (\omega)_{ij} {}^C
			:= & \omega_{ij} {}^c Z^j _J Z^C _c - \frac{1}{m-1} D^j \omega_{ij} {}^c X_J Z^C _c \\
			& + H_c \omega_{ij} {}^c Z^j _J X^C - \frac{1}{m-1} H_c D^j \omega_{ij} {}^c X_J X^C,
		\end{split}
	\end{equation}
	where again the intrinsic Levi-Civita connection $D$ is coupled to the normal
	Levi-Civita connection $\nabla^\perp$ when acting on $\omega$.
\end{lem}
\begin{proof} Given $\omega_{ij} {}^c\in \Gamma( S^2 _0 T^* \Sigma \otimes N \Sigma)$, using~\eqref{covector_trans} and~\eqref{eqn_norm_conn_transform}, one computes that
	\begin{equation}
		\hat{D}^j \omega_{ij} {}^c = D^j \omega_{ij}
		{}^c + (m-1) \Upsilon^j \omega_{ij} {}^c - \Upsilon_i \omega_{kl} {}^c
		\confmet^{kl} = D^j \omega_{ij} {}^c + (m-1) \Upsilon^j \omega_{ij} {}^c
	\end{equation}
	since $\omega_{ij} {}^c$ is trace-free over the pair of indices
	$(i,j)$.

	Therefore, using the above together with
	equations~\eqref{mean_curvature_transformation} and~\eqref{tractor_proj_transforms}, 
	and that the $X$ tractor is conformally invariant,
	\begin{align*}
		 & \omega_{ij} {}^c \hat{Z}^j _J \hat{Z}^C _c -
		\frac{1}{m-1} \hat{D}^j \omega_{ij} {}^c \hat{X}_J \hat{Z}^C _c +
		\hat{H}_c \omega_{ij} {}^c \hat{Z}^j _J \hat{X}^C - \frac{1}{m-1}\hat{H}_c
		\hat{D}^j \omega_{ij} {}^c \hat{X}_J \hat{X}^C \\  & =
		   \omega_{ij} {}^c \left( Z^j _J + \Upsilon^j X_J \right) \left( Z^C _c +
		\Upsilon_c X^C \right)                                                   \\  & \hspace{2em} - \frac{1}{m-1} \left( D^j \omega_{ij}
			   {}^c + \left( m-1 \right) \Upsilon^j \omega_{ij} {}^c \right) X_J \left( Z^C _c
		+ \Upsilon_c X^C \right)                                                 \\
		& \hspace{2em} + \left( H_c - \mathrm{N}_c ^d \Upsilon_d \right) \omega_{ij}
		{}^c \left( Z^j _J 
		+ \Upsilon^j X_J \right) X_C \\
		&\hspace{2em} - \frac{1}{m-1} \left( H_c - \mathrm{N}_c ^d \Upsilon_d \right)
		\left( D^j \omega_{ij} {}^c + (m-1) \Upsilon^j \omega_{ij} {}^c \right) X_J
		X^C                        \\
		& = \omega_{ij} {}^c Z^j _J
		   Z^C _c - \frac{1}{m-1} D^j \omega_{ij} {}^c X_J Z^C _c + \left( H_c
		   \omega_{ij} {}^c - \omega_{ij} {}^c \Upsilon_c + \omega_{ij} {}^c
		   \Upsilon_c \right) Z^j _J X^C \\
		   & \hspace{2em} + \left(
		   -\frac{1}{m-1} H_c D^j \omega_{ij} {}^c - H_c \omega_{ij} {}^c \Upsilon^j
		   + \frac{1}{m-1} \Upsilon_c D^j \omega_{ij} {}^c + \omega_{ij} {}^c
		   \Upsilon^j \Upsilon_c - \omega_{ij} {}^c \Upsilon^j \Upsilon_c \right.
		   \\
		   &\hspace{2em} \left. - \frac{1}{m-1}
		   \Upsilon_c D^j \omega_{ij} {}^c + H_c \omega_{ij} {}^c \Upsilon^j -
		   \omega_{ij} {}^c \Upsilon^j \Upsilon_c + \omega_{ij} {}^c
		   \Upsilon^j \Upsilon_c \right) X_J X^C \\ 
		&= \omega_{ij} {}^c Z^j _J Z^C _c - \frac{1}{m-1} D^j \omega_{ij} {}^c X_J
		Z^C _c
			+ H_c \omega_{ij} {}^c Z^j _J X^C - \frac{1}{m-1} H_c D^j \omega_{ij} {}^c
			X_J X^C,
	\end{align*}
	which verifies the claimed conformal invariance of the operator $\mathbb{M}_{c
	J} {}^{j C}$. 
  \end{proof}
Asking whether $\LL$ is the image of $\mathring{\II}$ under $\mathbb{M}$ then immediately leads to the $\mu$ invariant:
\begin{thm}
	The tensor $\mu_i {}^c$ is equal to the projecting part of the tractor
\begin{equation}
	\label{eqn_mu_in_tractor} 
        \LL_{i J} {}^C -
	\mathbb{M}(\mathring{\II})_{iJ} {}^C.
\end{equation}
In particular, $\mu_i {}^c$ is a conformal invariant of the embedding.
\end{thm}

\begin{proof}
	By inspection, one sees that~\eqref{eqn_mu_in_tractor} has zero in the $Z^j _J
		Z^C _c$ slot (since $Z^K _j Z_D ^c \LL_{i K} {}^D = \mathring{\II}_{ij} {}^c$)
	and hence projecting out the $X_J Z^C _c$ slot must yield a conformally
	invariant object.
	Such projection is accomplished by contraction with $Y^J Z_C ^c$, and from
	equations~\eqref{eqn_tractor2ff_XZ}  and~\eqref{eqn_M_operator} one
	sees that this projection is equal to
\begin{equation}
	Y^J
	Z_C ^c \left( \LL_{i J} {}^C - \mathbb{M}( \mathring{\II})_{iJ}
	{}^C \right) = \mathrm{N}^c _b \left( P_i {}^b - \nabla_i H^b \right) +
	\frac{1}{m-1} D^j \mathring{\II}_{ij} {}^c,
\end{equation}
	which is exactly $\mu_i {}^c$ as defined in~\eqref{eqn_mu_tensor} (since the
	$c$ index of $D^j \mathring{\II}_{ij} {}^c$ is already normal).
\end{proof}

In fact, the data $(\tfII_{ij} {}^c, \mu_i {}^c)$ is equivalent to the tractor
second fundamental form. We have seen how to obtain $\mathring{\II}_{ij} {}^c$ and $\mu_i {}^c$ from $\LL_{i J} {}^C$. For the reverse direction, note that $\LL_{i J} {}^C$ may be constructed
from $(\tfII_{ij} {}^c, \mu_i {}^c)$ according to 
\begin{equation}
	\label{eqn_tfII_mu_LL}
	\begin{split}
		(\mathring{\II}_{ij} {}^c , \mu_i {}^c) &\mapsto \mathring{\II}_{ij} {}^c Z_J
		^j Z^C _c + \left( \mu_i {}^c - \frac{1}{m-1} D^j \mathring{\II}_{ij} {}^c
		\right) X_J Z^C _c \\ &\hspace{2em} + H_c \mathring{\II}_{ij} {}^c Z_J ^j X^C +
		H_c \left( \mu_i {}^c - \frac{1}{m-1} D^j \mathring{\II}_{ij} {}^c \right) X_J
		X^C.
	\end{split}
\end{equation}

\subsection{Strong conformal circularity}

Unlike the situation for geodesics in Riemannian geometry, in conformal geometry if a submanifold is weakly conformally circular the submanifold conformal circles need not be ambient conformal circles. This leads to the following two stronger notions of conformal circularity. 

\begin{defn}\label{def_strong_cc}
  \emph{Let $\Sigma$ be a submanifold in a conformal manifold $M$.
  Then $\Sigma$ is \emph{strongly conformally circular} if any projectively-parametrised $\Sigma$-conformal circle is also a projectively-parametrised $M$-conformal circle. For the
  cases of submanifolds of dimensions 1 and 2, recall that the
  intrinsic conformal structure does not determine a conformal circle
  equation, however, the induced M\"{o}bius structures defined in
  Section~\ref{lowdim} do determine a conformal circle equation (the usual conformal circle equation \eqref{eqn_proj_param_conf_circ} with the Schouten tensor being as defined in Section~\ref{lowdim}), and
  that is the notion we are using here.}
  \end{defn}

\begin{defn}
  \emph{Let $\Sigma$ be a submanifold in a conformal manifold $M$.
  Then $\Sigma$ is \emph{conformally circular} if any unparametrised $\Sigma$-conformal circle is also an unparametrised $M$-conformal circle.}
\end{defn}

Definition \ref{def_strong_cc} appears in \cite{BelgunFlorin2015}, where the term \emph{strongly conformally geodesic} is used. In the following two theorems we characterise these two notions of conformal circularity in terms of the basic tractor invariants $\LL$ and $\mathsf{S}$ of the conformal submanifold $\Sigma$. The first theorem below is easily seen to be equivalent to \cite[Theorem 5.4(3)]{BelgunFlorin2015}, but we include proofs of both theorems for completeness.
\begin{thm}\label{thm_totally_conformally_circ_param}
  Let $\Sigma$ be a submanifold in a conformal manifold $M$.
  Then $\Sigma$ is strongly conformally circular if, and only if $\LL_{i J} {}^C = 0$ and $\mathsf{S}_{i J K} = 0$ (i.e.\ $\cF_{ij} = 0$).
\end{thm}

\begin{thm}\label{thm_totally_conformally_circ_unparam}
  Let $\Sigma$ be a submanifold in a conformal manifold $M$.
  Then $\Sigma$ is conformally circular if, and only if $\LL_{i J} {}^C = 0$ and $\mathsf{S}_{i J K} \,\propto\; \confmet_{ij} Z^j _{[J} X_{K]}$ (i.e.\ $\cF_{ij} \,\propto\; \confmet_{ij}$).
\end{thm}

Note that following our conventions for the Fialkow tensor in Section~\ref{lowdim}, these results still hold when $\Sigma$ is a submanifold of dimension 2. The $\Sigma$-conformal circle equation is then the usual conformal circle equation (either the projectively parametrized equation~\eqref{eqn_proj_param_conf_circ} or the parametrization-independent weighted equation~\eqref{unparam_circle}) with the Schouten tensor defined in Section~\ref{lowdim} playing the role of the usual Schouten tensor.
For a 1-dimensional submanifold, weakly conformally circular and strongly conformally circular are equivalent (and the Fialkow tensor is defined to be zero), so in this case Theorems \ref{thm_totally_conformally_circ_param} and \ref{thm_totally_conformally_circ_unparam} reduce to the fact that a conformal circle is characterized by $\LL=0$.

\begin{proof}[Proof of Theorem~\ref{thm_totally_conformally_circ_param}]
  First suppose that $\Sigma$ is strongly conformally circular. Since $\Sigma$ is then also weakly conformally circular, we must have $\LL_{i J} {}^C = 0$ by Theorem~\ref{thm_weakly_conformally_circular}. In particular, $\Sigma$ is totally umbilic ($\mathring{\II}_{ij} {}^c = 0$). Now, suppose that $\gamma$ is a projectively parametrised $\Sigma$-conformal circle with initial data a given $2$-jet at $p \in \Sigma$. Then by assumption $\gamma$ is also an $M$-conformal circle with the same initial data, now viewed as the $2$-jet of a curve in $M$. Let $g\in \bc$ be a minimal scale for $\Sigma$. The curve $\gamma$ must satisfy the intrinsic and ambient versions of the conformal circle equation, namely 
  \begin{equation}\label{eqn_totally_conf_circ_intrinsic_conf_circ_eq}
      \frac{d^D a^j}{dt} = u^2 \cdot u^i p_i {}^j + 3 u^{-2} \left( u_k a^k \right) a^j - \frac{3}{2} u^{-2} \left( a_k a^k \right) u^j - 2 u^k u^l p_{kl} u^j
  \end{equation}
  and 
  \begin{equation}\label{eqn_totally_conf_circ_ambient_conf_circ_eq}
      \frac{d^\nabla a^b}{dt} = u^2 \cdot u^c P_c {}^b + 3 u^{-2} \left( u_c a^c \right) a^b - \frac{3}{2} u^{-2} \left( a_c a^c \right) u^b - 2 u^c u^d P_{cd} u^b
  \end{equation}
  respectively.
  From the Gau{\ss} formula~\eqref{eqn_tensor_gauss_formula} and the fact that $\Sigma$ is totally geodesic with respect to $g$, we have that
  \begin{equation*}
    a^b = u^i \nabla_i u^b =  \Pi^b _j u^i D_i u^j + \II_{ij} {}^b u^i u^j= \Pi^b_j a^j
  \end{equation*}
  and
  \begin{equation*}
    \frac{d^\nabla a^b}{dt} = \Pi^b_j\frac{d^D a^j}{dt} ,
  \end{equation*}
  where $a^j = \frac{d^D u^j}{dt}$, i.e.\ the acceleration of the curve calculated intrinsically. Therefore, applying $\Pi^j _b$ to~\eqref{eqn_totally_conf_circ_ambient_conf_circ_eq} and subtracting~\eqref{eqn_totally_conf_circ_intrinsic_conf_circ_eq} from the result, we see that 
  \begin{equation*}
    u^2 \cdot u^k P_k {}^j  - 2 u^k u^l P_{kl} u^j 
  = u^2 \cdot u^k p_k {}^j - 2 u^k u^l p_{kl} u^j,
  \end{equation*}
  where we have used that $\Pi^j _b u^b = u^j$, $\Pi^j _b a^b = a^j$, $u_c a^c = u_k a^k$ (which follows from the Gau{\ss} formula without the need for $g$ to be a minimal scale since $u^c$ is tangent to the submanifold $\Sigma$), and $a^c a_c = a^k a_k$ (since $\Sigma$ is totally geodesic with respect to $g$).  Contracting the above display with $u_j$ yields
  \begin{equation*}
    u^2 \cdot u_j u^k P_k {}^j - 2 u^k u^l P_{kl} \cdot u^2 = u^2 \cdot u_j u^k p_k {}^j - 2 u^k u^l p_{kl} \cdot u^2,
  \end{equation*}
  and hence 
  \begin{equation}\label{eqn_totally_conf_circ_diff_schoutens}
    \left( P_{ij} - p_{ij} \right) u^i u^j = 0.
  \end{equation}
Since $\II_{ij} {}^c = 0$, the term in parentheses in the above display is exactly the Fialkow tensor $\cF_{ij}$ from~\eqref{fialkow}.
  Now, at any point $p \in \Sigma$, any $u \in T_p \Sigma$ can arise as the velocity of a conformal circle, and hence~\eqref{eqn_totally_conf_circ_diff_schoutens} must hold for all $u^i \in \cE^i$.
  Hence $\cF_{ij} = 0$.
  Together with our earlier observation that strong conformal circularity implies weak conformal circularity, this establishes that $\LL_{i J} {}^C = 0$ and $\mathsf{S}_{i J K} = 0$.
  \medskip

  Conversely, suppose that $\LL_{i J} {}^C = 0$ and $\mathsf{S}_{i J K} = 0$.
  Let $\gamma$ be a projectively parametrized $\Sigma$-conformal circle.
  Then the intrinsic acceleration tractor $A^J$ of $\gamma$ satisfies $\frac{d^D A^J}{dt} = 0$.
  We must also show that $\gamma$ satisfies the ambient (projectively parametrized) conformal circle equation.

  For any parametrized curve in $\Sigma$, writing $U^J$ for its velocity tractor, we see from explicit form of the velocity tractor~\eqref{eqn_velocity_tractor_slots} and the formula~\eqref{eqn_norm_perp_intrinsic_isom} for the isomorphism $\Pi^B _J$ that $U^B = \Pi^B _J U^J$, where again we use the tangentiality of $u^c$ together with the Gau{\ss} formula to conclude that $u_c a^c = u_k a^k$.
  Then applying the tractor Gau{\ss} formula \eqref{TracGaussFormula} shows that the ambient acceleration tractor is given by
  \begin{equation*}
    A^B = \frac{d^\nabla U^B}{dt}
    = \Pi^B _J ( u^i D_i U^J + \mathsf{S}_i {}^J {}_K U^K ) + u^i \LL_{i J} {}^B U^J
    = \Pi^B _J \frac{d^D U^J}{dt}
    = \Pi^B _J A^J,
  \end{equation*}
	since $\LL=0$ and $\mathsf{S}=0$. Similarly, it follows that 
  \begin{equation*}
    \frac{d^\nabla A^B}{dt}
    = \Pi^B _J \frac{d^D A^J}{dt}.
  \end{equation*}
So if moreover $\gamma$ satisfies $d^D A^J/dt = 0$, then also $d^\nabla A^B/dt = 0$.
  Finally, recall that the isomorphism $\Pi^B _J$ is also metric-preserving, and so 
  \begin{equation*}
    A^B A_B = A^J A_J = 0,
  \end{equation*}
  since $A^B=\Pi^B_JA^J$ and we assumed $\gamma$ is $\Sigma$-projectively parametrised.  Thus, by the characterisation from \cite{BEG}, $\gamma$ is a projectively-parametrised $M$-conformal circle, and therefore $\Sigma$ is strongly conformally circular.
\end{proof} 

We next prove the parametrization-independent version of Theorem~\ref{thm_totally_conformally_circ_param}, namely Theorem~\ref{thm_totally_conformally_circ_unparam}.
As this theorem is a statement about unparametrised conformal circles, we use the 3-tractor $\Phi$ (which equals the Hodge-star of the normal tractor form) discussed in Section~\ref{sec_circles}.

\begin{proof}[Proof of Theorem~\ref{thm_totally_conformally_circ_unparam}]
  First, suppose that $\Sigma$ is conformally circular, that is, every unparametrised $\Sigma$-conformal circle is an unparametrised $M$-conformal circle.
  Let $\gamma$ be a $\Sigma$-conformal circle.
  In the previous proof, we observed that strong conformal circularity implies weak conformal circularity.
  Parametrisation was not used at all in this part of the proof and hence we may employ the same argument here.
  Thus $\LL_{i J} {}^C = 0$ by Theorem~\ref{thm_weakly_conformally_circular}.

  As described in Section~\ref{sec_circles}, $\gamma$ determines an intrinsic 3-tractor $\phi^{I J K} \in \cE^{[ I J K ]}$ which satisfies $X^{[I} \phi^{J K L]} = 0$ and $u^i D_i \phi^{I J K} = 0$, where $X^I$ is the intrinsic canonical tractor and $D_i$ is the intrinsic tractor connection.  Explicitly,
  \begin{equation*}
    \phi^{I J K} = 6 u^{-1} X^{[I} U^J A^{K]},
  \end{equation*}
  with $U^J$ and $A^K$ defined as in equations~\eqref{eqn_velocity_tractor} and~\eqref{eqn_accel_tractor} respectively, using the intrinsic position tractor and tractor connection.

  On the other hand, viewing $\gamma$ as an ambient curve also defines a 3-tractor,
  \begin{equation*}
    \Phi^{ABC} = 6 u^{-1} X^{[A} U^B A^{C]},
  \end{equation*}
 where $X^A$ is the ambient position tractor and $U^B$ and $A^C$ are the velocity and acceleration tractors of $\gamma$ as a curve in $(M,\bc)$ (note that while $U^B = \Pi^B_J U^J$, we are abusing notation slightly in that the ambient acceleration tractor $A^C$ need not equal $\Pi^C_KA^K$, as can be seen from the tractor Gau{\ss} formula).
 Since $\Sigma$ is conformally circular, $\gamma$ is an $M$-conformal circle and so $\Phi$ must be parallel along the curve, i.e.\ $u^a \nabla_a \Phi^{ABC} = 0$.

Fix a reference metric $g\in \bc$. By \eqref{deriv_trac_3form} the derivatives of the intrinsic and ambient 3-tractors are 
  \begin{equation*}
    \mbf{u}^i D_i \phi^{IJK} = 6 \left( \mbf{u}^i D_i \mbf{a}^k - \mbf{u}^l p_l {}^k \right) \mbf{u}^j X^{[I} Z^J _j Z^{K]} _k,
  \end{equation*}
  and 
  \begin{equation*}
    \mbf{u}^i \nabla_i \Phi^{ABC} = 6 \left( \mbf{u}^i \nabla_i \mbf{a}^c - \mbf{u}^d P_d {}^c \right) \mbf{u}^b X^{[A} Z^B _b Z^{C]} _c,
  \end{equation*}
  respectively.
  Both of the above displays are zero and hence
  \begin{align*}
    0
    &= \Pi^I _A \Pi^J _B \Pi^K _C \left( \mbf{u}^i \nabla_i \Phi^{ABC} \right) - \mbf{u}^i D_i \phi^{I J K} \\ 
    &= 6 \left[ \left( \mbf{u}^i D_i \mbf{a}^k - \mbf{u}^l P_l {}^k \right) - \left( \mbf{u}^i D_i \mbf{a}^k - \mbf{u}^l p_l {}^k \right) \right] \mbf{u}^j X^{[I} Z^J _j Z^{K]} _k \\ 
    &= - 6 \mbf{u}^l \left( P_l {}^k - p_l {}^k \right) \mbf{u}^j X^{[I} Z^J _j Z^{K]} _k, 
  \end{align*}
  where we have used~\eqref{pi_X_Z} and the Gau{\ss} formula in going from the first to second lines.
 It follows that the antisymmetric part of $\mbf{u}^l \left( P_l {}^k - p_l {}^k \right)\mbf{u}^j$ equals zero. Thus
  \begin{equation*}
    \mbf{u}^l \left( P_l {}^k - p_l {}^k \right) \propto \,\mbf{u}^k.
  \end{equation*}
  Since we have already seen that $\tfII_{ij} {}^c = 0$, this means that
  \begin{equation*}
    \mbf{u}^l \cF_l {}^k = \mbf{u}^l \left( P_l {}^k - p_l {}^k + H_c \tfII_l {}^{kc} + \frac{1}{2} H_c H^c \delta_l ^k \right) \propto\, \mbf{u}^k.
  \end{equation*}
Now, since the (weighted) velocity of a $\Sigma$-conformal circle passing through a point $p\in \Sigma$ can be any unit vector of $T_p\Sigma [-1]$, it follows that $\cF_l {}^k$ must equal $f\delta_j^k$ for some smooth weight $-2$ density on $\Sigma$, and hence $\cF_{ij}=f\bg_{ij}$.

  \medskip
  Conversely, suppose that $\LL_{i J} {}^C = 0$ and  $\cF_{ij}=f\bg_{ij}$ for some  $-2$ density on $\Sigma$, and let $\gamma$ be a $\Sigma$-conformal circle.
  The $\Sigma$-conformal circle $\gamma$ determines an intrinsic 3-tractor $\phi^{I J K} =  6 u^{-1} X^{[I} U^J A^{K]}$ which is parallel along $\gamma$ for the intrinsic tractor connection.  To show that $\gamma$ is an $M$-conformal circle, we need to show that the ambient 3-tractor $\Phi^{ABC} = 6 u^{-1} X^{[A} U^B A^{C]}$ satisfies these same properties.  We show this by using the conditions on the tractor second fundamental form and the difference tractor to relate the ambient $X, U$ and $A$ tractors to their intrinsic counterparts.

  From the isomorphism of Theorem~\ref{thm_norm_perp_intrinsic_isom} (cf.\ the proof of Theorem~\ref{thm_totally_conformally_circ_param}), it follows that $X^A = \Pi^A _I X^I$ and $U^B = \Pi^B _J U^J$, where $X^I$ and $U^J$ are the intrinsic submanifold canonical tractor and velocity tractor of $\gamma$ respectively.
  From the tractor Gau{\ss} formula \eqref{TracGaussFormula} we therefore have
  \begin{align*}
    A^B
    &= u^i \nabla_i U^B = u^i \nabla_i ( \Pi^B _J U^J)\\
    &= \Pi^B _J \left[ u^i D_i U^J + u^i f \bg_{ij} \left( Z^{Jj} X_K - Z^j _K X^J \right) U^K \right] \\ 
    &= \Pi^B _J \left( A^J - u  f X^J \right) \\
		&= \Pi^B _J A^J - u  f X^B ,
  \end{align*}
  and 
  \begin{align}
    u^i \nabla_i A^C 
    &= u^i \nabla_i ( \Pi^C _J A^J) - u^i \nabla_i ( u f X^C) \nonumber \\ 
    &= \Pi^C _J ( u^i D_i A^J + f  u^i ( Z^J _i X_K - Z_{iK} X^J) A^K) \nonumber \\ 
    &\hspace{8em} - u^i \nabla_i ( u  f) X^C - u f  u^c Z^C _c \nonumber \\ 
    &= \Pi^C _J (u^i D_i A^J) - 2 u f u^c Z^C _c + \rho X^C, \label{eqn_totally_accel_int}
  \end{align}
  where we have collected all the terms in the bottom slot into $\rho$ (the exact form of $\rho$ will not be important).  Now, recall that $U^B = u^i \nabla_i (u^{-1} X^B)$ and $A^B = u^i \nabla_i U^B$.
  Hence, using the skew-symmetry,
  \begin{equation*}
    u^i \nabla_i \Phi^{ABC} = u^i \nabla_i \left( 6 u^{-1} X^{[A} U^B A^{C]} \right) = 6 u^{-1} X^{[A} U^B (u^i \nabla_i A^{C]} ).
  \end{equation*}
Substituting~\eqref{eqn_totally_accel_int} for the derivative of the acceleration into the above it is easy to see that the $- 2 u f u^c Z^C _c + \rho X^C$ terms drop out due to the skewing with $X^A$ and with $U^B$ (which is proportional to $u^bZ^B_b$ modulo the canonical tractor) and thus we obtain, 
\begin{equation}\label{eqn_totally_grad_sigma_int}
u^i \nabla_i \Phi^{ABC} =  u^i \Pi^A _I \Pi^B _J \Pi^C _K D_i \phi^{I J K}.
\end{equation}
  Finally, the right-hand side of~\eqref{eqn_totally_grad_sigma_int} is zero, since $\gamma$ is a $\Sigma$-conformal circle.
  Thus $\gamma$ is an $M$-conformal circle.
\end{proof}

\subsection{Examples}\label{sec-examples}

We collect here some elementary examples of distinguished submanifolds in conformal manifolds.

\begin{example}\label{example-model} \emph{\textbf{(The Model Case.)}} \emph{Let $(S^n,\bc)$ be the standard conformal sphere (additionally endowed with its standard M\"obius conformal structure if $n=2$, which is induced by the standard embedding $S^2\subseteq \mathbb{R}^3$). When $m=\dim \Sigma >1$ the invariant $\mu$ is given by \eqref{eq:mu-for-m-geq-2}, and hence is trivial due to the vanishing of the ambient Weyl curvature (when $m=\dim \Sigma =1$ this formula does not apply and the vanishing of $\mu$ characterizes conformal circles). Also, from the vanishing of the ambient Weyl curvature together with \eqref{fialkow_weyl} and \eqref{eq:Fialkow-2D} it follows that $\cF_{ij}=0$ for umbilic submanifolds. Therefore, in the model case, the notions of strongly conformally circular, conformally circular, and weakly conformally circular (i.e.\ distinguished) submanifolds are all equivalent; the distinguished submanifolds of $S^n$ are precisely the umbilic ones when $m>1$ and conformal circles when $m=1$.} 
\end{example}
\begin{rem}
The equivalence of all these conditions to $\LL=0$ in the model case leads to a very simple proof of the classical characterisation of umbilic submanifolds in the sphere: The condition $\LL=0$ on $\Sigma\subseteq S^n$ implies that the tractor bundle of $\Sigma$ is parallel in the (flat) ambient tractor bundle and hence corresponds to a fixed subspace $V$ of $\mathbb{R}^{n+1,1}$; viewing $S^n$ as the ray projectivisation of the forward null cone $\mathcal{C}_+$ in $\mathbb{R}^{n+1,1}$ it follows that $\Sigma$ must be a piece of the subsphere $\mathbb{P}_+(V\cap \mathcal{C}_+) \subset S^n$, which corresponds to the intersection of $S^n\subset \mathbb{R}^{n+1}$ with an affine subspace of dimension $\dim \Sigma + 1$. The simplicity of this, and related arguments, is one of the reasons we believe the distinguished submanifold condition should be seen as fundamental in conformal geometry.
\end{rem}


Consideration of the model case leads us to ask whether the various notions of conformal circularity defined above are independent in general. Note that when $m=1$ umbilicity is a trivial condition (satisfied by any curve) and the notions of strong, weak, and conformal circularity are all equivalent to being a conformal circle. When $m=2$ the notions of weak conformal circularity and conformal circularity are equivalent (as we have required $\mathcal{F}_{(ij)_0}$ to be zero when defining the induced M\"obius structure on $\Sigma$). Also, note that when $\Sigma$ is a hypersurface (and $n\geq 3$) then the invariant $\mu$ is trivial, and hence being weakly conformally circular (distinguished) is equivalent to being umbilic. Outside of these cases, however, these notions are distinct, as shown by the following examples:

\begin{example}\label{example-formal} \emph{\textbf{(Extending $(\Sigma,c_{\Sigma})$ with prescribed $\mathring{\II}$, $\mu$, and $\cF$.)}} \emph{Let $(\Sigma,c_{\Sigma})$ be an $m$-dimensional conformal manifold and $M=\Sigma \times \mathbb{R}^d$ with $d\geq 1$ (or, more generally, let $M$ be a rank $d$ vector bundle over $\Sigma$). By considering the freedom in the $2$-jet along $\Sigma$ of an arbitrary extension $g$ of $g_{\Sigma}\in \bc_{\Sigma}$ to $M$ one can show that the geometric quantities $\mathring{\II}$, $\mu$, and $\cF$ can be freely prescribed along $\Sigma$ by an appropriate choice of extension, where the trace free tensor $\mathring{\II}$ is trivially zero if $m=1$, $\mu$ must be zero if $d=1$ (unless $m$ also equals $1$ and $M$ is taken to have a conformal M\"obius tructure) and $\cF$ must be pure trace if $m=2$ and zero if $m=1$; see, \cite[Theorem 4.22]{BelgunFlorin2015}. (Note that when $m>2$ our $\mathcal{F}$ is equivalent to the relative Schouten-Weyl tensor $\rho$ in \cite{BelgunFlorin2015}, but we when $m=2$ or $1$ we do not assume that $\Sigma$ has its own M\"obius/Laplace structure, and hence we define $\mathcal{F}$ differently from $\rho$ in \cite{BelgunFlorin2015}.)}
\end{example}

It follows that when $m>2$ and $d>1$ being strongly conformally circular ($\mathring{\II}=0$, $\mu=0$ and $\cF=0$) is a stronger notion than being conformally circular ($\mathring{\II}=0$, $\mu=0$ and $\mathring{\cF}=0$), which is itself a stronger notion than being weakly conformally circular ($\mathring{\II}=0$ and $\mu=0$, equivalently, $\LL=0$), and this in turn is a stronger notion than being umbilic. 

We can also easily see these inequivalences by considering simple examples.

\begin{example} \emph{\textbf{(A factor in a product of Einstein metrics.)} If $(M_1,g_1)$ and $(M_2,g_2)$ are Einstein manifolds and $\Sigma = M_1\times \{p_2\}$ then $\Sigma$ is totally geodesic in the Riemannian product $(M,g)$ of $M_1$ and $M_1$, hence $\mathring{\II}=0$. Moreover, since the Ricci curvature of the product metric is the sum of the Ricci curvatures of the factors it is readily seen that $N^c_bP_i{}^b = \Pi_i^a N^c_bP_a{}^b =0$ (when computing with respect to $g$), hence $\mu=0$. Similarly, it follows that the Fialkow tensor must be pure trace since (if $m>2$) $\mathcal{F}_{ij} = P_{ij}-p_{ij}$ in the scale $g$ and both $P_{ij}$ and $p_{ij}$ must be proportional to $g_{ij}$ (the metric on $\Sigma$). However, since the formula for $P_{ij}$ involves both the ambient scalar curvature and the ambient dimension, it will typically differ from $p_{ij}$. For example, if $\Sigma = S^m\times \{p\}$ in $S^m\times S^k$ with $m>2$ and $k> 1$ then it is readily checked that $\mathcal{F} = c_{m,k}g_{\Sigma}$ with $c_{m,k}\neq 0$. If $m=2$ we find the same result using the formula \eqref{eq:Fialkow-2D} for $\cF$ instead, since the Weyl curvature is readily computed in terms of Kulkarni-Nomizu products of $g_1$ and $g_2$ from which one can easily check that $\mathrm{tr}^2\,\iota^* W \neq 0$. Thus, in these cases $\Sigma = S^m\times \{p\}$ is conformally circular, but not strongly conformally circular. 
}
\end{example}

It is possible to characterize the Einstein products for which the
submanifolds of the form $M_1\times \{p_2\}$ and $\{p_1\}\times M_2$
are strongly conformally circular. A \emph{special Einstein product}
is the Riemannian product $(M,g)$ of a pair of Einstein manifolds
$(M_1,g_1)$ and $(M_2,g_2)$ with the property that $\mathrm{Ric}^{g_1}
= (n_1-1)\lambda g_1$ and $\mathrm{Ric}^{g_2} = -(n_2-1)\lambda g_2$
for some constant $\lambda$, where $n_i=\dim M_i$, see e.g.\ \cite{Gover-Leitner-class}. (The simplest
example of such a product is $S^m\times H^d$ with the standard product
metric.) Given a general product of $(M_1,g_1)$ and $(M_2,g_2)$ with
$\mathrm{Ric}^{g_1} = (n_1-1)\lambda g_1$, a straightforward
calculation shows that a submanifold of the form $M_1\times \{p_2\}$
has $\cF =0$ if and only if $\mathrm{Ric}^{g_2} = -(n_2-1)\lambda
g_2$. Thus we have:

\begin{thm}
Let $(M,g)$ be a product of Einstein manifolds $(M_1,g_1)$ and $(M_2,g_2)$. The submanifolds of the form $M_1\times \{p_2\}$ (or, equivalently, those of the form $\{p_1\}\times M_2$) are strongly conformally circular if and only if $(M,g)$ is a special Einstein product.
\end{thm}

On the other hand, a factor in a generic product will typically only be weakly conformally circular:

\begin{example} \emph{\textbf{(A factor in a generic product.)}
Let $(M,g)=(M_1,g_1)\times (M_2,g_2)$ and $\Sigma = M_1\times {p_2}$, where $m=\dim \Sigma >2$ and $(M_1,g_1)$ is not Einstein. Then, $\Sigma$ is again totally geodesic and by the product formula for the Ricci curvature one still has that $N^c_bP_i{}^b = \Pi_i^a N^c_bP_a{}^b =0$, when computing with respect to $g$. Hence we have the conformally invariant conditions $\mathring{\II}=0$ and $\mu=0$. On the other hand, with respect to $g$ we have $\mathcal{F}_{ij} = P_{ij}-p_{ij}$ and this will not typically be proportional to the metric $g_{ij}$ on $\Sigma$. For a concrete example that is easy to compute, take, say, $M_1 = S^2\times S^1$ with the product metric and $M_2 = \mathbb{R}^d$, $d\geq 1$, with the Euclidean metric. Let $h$ denote the pullback of the round metric on $S^2$ to $M_1=S^2\times S^1$ and let $d\theta^2$ denote the pullback of the metric on $S^1$. The Ricci curvature of $g_1=h+d\theta^2$ on $M_1$ is then $h$ and the Ricci curvature of $M_2$ is zero. It follows that $p= \frac{3}{4}h-\frac{1}{4}d\theta^2$, but $\iota^*P = \frac{2d+3}{2(d+1)(d+2)} h - \frac{1}{2(d+1)(d+2)}d \theta^2$ and hence $\cF = \iota^*P - p$ is never proportional to $g_{\Sigma} = h+ d\theta^2$. Thus, $\Sigma$ is weakly conformally circular (distinguished) but not conformally circular.
}
\end{example}
Similar examples can be found by considering factors in warped product metrics, or doubly warped products, since these are conformal to Riemannian products. On the other hand, twisted products are not typically conformal to a Riemannian product and give a new class of examples. In what follows we will abuse notation by writing $g_i$ for the pullback $\pi_i^*g_i$ of a metric on $M_i$ via the projection to the $i$th factor $\pi_i:M_1\times M_2 \to M_i$. Let $(M_1,g_1)$ and $(M_2,g_2)$ be Riemannian manifolds. A \emph{doubly warped product} metric on $M=M_1\times M_2$ is a metric of the form $g=f_2 g_1 + f_1 g_2$, where $f_i:M_i\to \mathbb{R}_+$ is a smooth function for $i=1,2$; such a doubly warped product $g$ is conformal to the product metric $\tilde{g} = \tilde{g}_1+\tilde{g}_2$, where $\tilde{g}_1 = f_1^{-1}g_1$ and $\tilde{g}_2 = f_2^{-1}g_2$. Hence the two natural foliations of a doubly warped product manifold are foliations by distinguished submanifolds (we will call such foliations distinguished). If $(M,g)$ is a warped product ($f_2=1$) then the fact that these foliations are distinguished can also easily be seen by computing in the scale $g$, since the submanifolds of the form $M_1\times \{p_2\}$ are totally geodesic, the submanifolds of the form $\{p_1\}\times M_2$ are umbilic with parallel mean curvature, and the mixed part of the Ricci tensor vanishes. For a general doubly warped product, however, then fact that these foliations are distinguished is less obvious from a Riemannian point of view since computing $\mu$ or $\LL$ in the scale $g$ becomes nontrivial:

\begin{example} \emph{\textbf{(A doubly warped product.)}
Let $M = \mathbb{R}^4 = \mathbb{R}^2\times \mathbb{R}^2$ with the doubly warped product metric 
$$
g = e^{2x_3}(dx_1^2 + dx_2^2) + e^{2x_1}(dx_3^2 + dx_4^2).
$$ 
We know that the two natural foliations of $M$ are distinguished, but we wish to show this from the point of view of the metric $g$. Consider, say, $\Sigma = \mathbb{R}^2\times \{(0,0)\}$. Clearly $\Sigma$ is umbilic. Indeed, in the standard coordinates $(x_1,x_2,x_3,x_4)$ the second fundamental form of $\Sigma$ is given in terms of the Christoffel symbols by $\II_{ij}{}^c = \Gamma^{c}{}_{ij}$ for $i,j\in \{1,2\}$ and $c\in \{3,4\}$, where $\Gamma^4{}_{ij}=0$ and $\Gamma^3{}_{ij} = -e^{2x_3-2x_1}\delta_{ij} = -e^{-2x_1}g_{ij}$ for $i,j\in \{1,2\}$. Thus, $\Sigma$ has mean curvature vector $H = -e^{-2x_1}\partial_3$. Using that $\Gamma^{1}{}_{13}=\Gamma^{3}{}_{13}=1$ and $\Gamma^{2}{}_{13}=\Gamma^{4}{}_{13}=0$ it follows that $\nabla_{\partial_1} H = 4e^{-2x_3}\partial_3 - 2e^{-2x_3}(\partial_1+\partial_3) = 2e^{-2x_3}\partial_3 - 2e^{-2x_3}\partial_1$. A straightforward calculation also shows that $\mathrm{Ric}_{13}=2$. Thus, from the formula for $\mu$ and the definition of the Schouten tensor, we compute that $\mu_1{}^3 = P_{1}{}^3 - \nabla_1 H^3 = \frac{1}{2}\mathrm{Ric}_1{}^3 - \nabla_1 H^3 = e^{-2x_1}-e^{-2x_1} = 0$. The other ``mixed'' components of the Schouten tensor and the other components of $\nabla^{\perp}H$ are all zero. Hence, $\mu=0$ and $\Sigma$ is a distinguished submanifold.
}
\end{example}

As noted above, we can get away from examples that are conformal to products by considering generic twisted products. Again let $(M_1,g_1)$ and $(M_2,g_2)$ be Riemannian manifolds. A \emph{doubly twisted product} metric on $M=M_1\times M_2$ is a metric of the form $g=h_1 g_1 + h_2 g_2$, where $h_i:M_1\times M_2\to \mathbb{R}_+$ is a smooth function for $i=1,2$; such a doubly twisted product metric is clearly conformal to the twisted product metric $g_1 + f g_2$, where $f:M_1\times M_2 \to \mathbb{R}_+$ is given by $h_2/h_1$, but such a metric is not typically conformally equivalent to a product metric on $M_1\times M_2$. The conformal structures arising from twisted products give rise to umbilic but not distinguished foliations:
 
\begin{example} \emph{\textbf{(A twisted product.)}
Let $M = \mathbb{R}^4 = \mathbb{R}^2\times \mathbb{R}^2$ with the twisted product metric 
$$
g = dx_1^2 + dx_2^2 + e^{2x_1x_3}(dx_3^2 + dx_4^2)
$$ 
and let $\Sigma = \mathbb{R}^2\times \{(a,b)\}$. Then $\Sigma$ is totally geodesic for $g$, and hence umbilic. Thus $\mu$ is given by the mixed part of the Schouten tensor. By the diagonal form of the metric, this is precisely half of the mixed part of the Ricci tensor. But by a straightforward calculation we find that $R_{13}=-1\neq 0$ and hence $\mu_1{}^3\neq 0$. Therefore $\Sigma$ is umbilic, but not distinguished.
}
\end{example}
In fact, the distinguished submanifold condition characterises precisely when a doubly twisted product metric is conformal to a product metric. Since doubly twisted product are conformal to twisted products, we need only consider the latter. One of the two natural foliations of a twisted product manifold is distinguished if and only if both are, and this is equivalent to the twisted product metric being a warped product and hence conformal to a product:
\begin{thm}\label{thm-twisted-prod}
Let $(M_1,g_1)$ and $(M_2,g_2)$ be Riemannian manifolds of dimension at least two, $f:M_1\times M_2\to \mathbb{R}_+$ a smooth function and $g=g_1+fg_2$ a twisted product metric on $M=M_1\times M_2$. Then the following are equivalent:
\begin{itemize}
\item[$\mathrm{(i)}$] The umbilic submanifolds $M_1\times\{p_2\}$ are distinguished for every $p_2\in M_2$;
\item[$\mathrm{(ii)}$] The umbilic submanifolds $\{p_1\}\times M_2$ are distinguished for every $p_1\in M_1$;
\item[$\mathrm{(iii)}$] The function $f:M_1\times M_2\to \mathbb{R}_+$ can be written as a product $f=f_1f_2$ of functions $f_i:M_i\to \mathbb{R}_+$, $i=1,2$.
\item[$\mathrm{(iv)}$] $\,g$ is a warped product metric $g_1+f_1\tilde{g}_2$ for some metric $\tilde{g}_2$ conformal to $g_2$.
\item[$\mathrm{(v)}$] $\,g$ is conformal to a product metric $\tilde{g}_1 + \tilde{g}_2$ on $M_1\times M_2$.
\end{itemize}
\end{thm}
\begin{proof}
We first note that the equivalence of (iii) and (iv) is obvious, and
that these clearly imply (v). To see that (v) implies (iv) we note
that if $g_1+fg_2 = \lambda(\tilde{g}_1 + \tilde{g}_2)$ for some
smooth function $\lambda: M_1\times M_2 \to \mathbb{R}_+$, then (from
$g_1 = \lambda \tilde{g}_1$) it follows that $\lambda$ must be
independent of the $M_2$ factor and hence $g = g_1 + f_1\tilde{g}_2$
where $f_1=\lambda$. From the discussion above we also know that (v)
implies (i) and (ii). To establish the result it is therefore enough
to show that (i) implies (iii) and that (ii) implies (iii).

Suppose (i) holds. Let $n=\dim M$ and $n_i = \dim M_i$. For $i=1,2$, let $X_i$ be a vector field on $M=M_1\times M_2$ tangent to the $i$th factor. Then the mixed-part of the Ricci tensor is given by (see, e.g., \cite{FGKU2001})
$$
\mathrm{Ric}(X_1,X_2) = (1-n_2)X_1 X_2 \log(f).
$$
But the submanifolds of the form $M_1\times \{p_2\}$ are totally geodesic in $(M,g)$ and hence for these submanifolds $\mu$ is given by the mixed-part of the Schouten tensor, which (due to the diagonal form of the metric with respect to the product decomposition $M=M_1\times M_2$) is $\frac{1}{n-2}$ times the mixed-part of the Ricci tensor. Thus, since (i) holds we must have $X_1 X_2 \log(f)=0$ for any pair of vector fields $X_1$, $X_2$, with $X_1$ tangent to the first factor and $X_2$ tangent to the second factor. It follows easily that $\log(f)$ is the sum of a function on $M_1$ and a function on $M_2$, and hence (iii) holds.

Since the distinguished submanifold condition is conformally invariant, the fact that (ii) also implies (iii) follows from the same argument we just used but now applied to the conformally related metric $f^{-1}g_1 + g_2$ (with the roles of $M_1$ and $M_2$ interchanged). This proves the result.
\end{proof}

The classical de Rham-Wu theorem states that a Riemannian manifold
$(M,g)$ possessing a pair of complementary orthogonal totally geodesic
foliations is locally a product Riemannian manifold, and this holds
globally if $(M,g)$ is complete and simply connected. An analogue of
this result for twisted products is given in
\cite{PongeReckziegel1993}: A Riemannian manifold $(M,g)$ possessing a
totally geodesic foliation and a complementary orthogonal totally
umbilic foliation is locally a twisted product manifold. This result
easily implies the following conformally invariant result: \emph{A
  Riemannian manifold $(M,g)$ possessing a pair of complementary
  orthogonal totally umbilic foliations is locally a doubly twisted
  product manifold.} To see this one merely rescales $g$ to make one
of the umbilic foliations totally geodesic (which can always be done
locally), and then applies the result of \cite{PongeReckziegel1993} to
conclude that the rescaled metric is a twisted product. Combining
these observations with Theorem \ref{thm-twisted-prod} we readily
obtain the following conformal  extension of the
(local version) of the de Rham-Wu theorem:

\begin{thm}\label{thm-conf-prod}
A conformal manifold $(M,\bc)$ is locally the conformal structure of a product manifold if and only if $(M,\bc)$ possesses a pair of complementary orthogonal foliations by distinguished submanifolds. 
\end{thm}
\begin{rem}
In fact, as can easily be seen from Theorem \ref{thm-twisted-prod}, it suffices to know that one of the foliations is distinguished and the other is umbilic (as it is then forced to be also distinguished).
\end{rem}


The above considerations suggest that the ``generic'' umbilic hypersurface will not be distinguished. On the other hand, there are many situations in which the geometry of the ambient manifold $(M,g)$ forces any umbilic submanifold to be distinguished, as we will see below. Towards this end, it is helpful to introduce Nomizu and Yano's notion of extrinsic spheres: A submanifold $\Sigma$ in a Riemannian manifold $(M,g)$ is called an \emph{extrinsic sphere} if it is umbilic and has parallel mean curvature vector (for the connection on the normal bundle). We then have the following easy observation:
\begin{thm}
An extrinsic sphere in an Einstein manifold is conformally distinguished.
\end{thm}

\begin{rem}
If $\Sigma$ is an umbilic hypersurface in an Einstein manifold $(M,g)$
then by contracting the Codazzi equation one readily finds that
$\Sigma$ has parallel mean curvature vector (equivalently, constant
mean curvature) \cite{Chen-umbilic,Kowalski1972} and hence is an
extrinsic sphere. Of course, umbilic hypersurfaces are all
distinguished so the above theorem teaches us nothing new in this
case. On the other hand, the fact that umbilic hypersurfaces are all
distinguished, together with the definition of $\mu$, furnishes an
easy proof that an umbilic hypersurface in an Einstein manifold is an
extrinsic sphere: in the Einstein scale one finds that $\mu_i{}^c =
\nabla^{\perp}_i H^c$ (and $\mu$ is zero).
\end{rem}

Returning to Example \ref{example-model}, we note that since (real) Riemannian space forms are conformally flat it follows that any umbilic submanifold of a real space form is a distinguished submanifold. The same holds for complex space forms (the simply connected examples of which are complex projective spaces, complex Euclidean spaces and complex hyperbolic spaces):

\begin{example}\label{example-c-space-form} \emph{\textbf{(Umbilic submanifolds in complex space forms.)} An inspection of the classification of umbilic submanifolds in complex space forms in \cite{ChenOguie1974} shows that all are extrinsic spheres (counting totally geodesic submanifolds as extrinsic spheres). Hence, since complex space forms are Einstein, every umbilic submanifold in a complex space form is conformally distinguished. }
\end{example}
To be more concrete, we discuss the example of complex projective space in more detail and compute the Fialkow tensor of the umbilic submanifolds:
\begin{example}\label{example-cpn} \emph{\textbf{(Conformal circularity of umbilic submanifolds in $\mathbb{CP}^n$.)} Let $(M,g)=(\mathbb{CP}^n, g_{FS})$, where $g_{FS}$ is the Fubini-Study metric. Let $J$ denote the complex structure on $M$ and let $J_{ab}=g_{ac}J^c{}_b$. With our conventions, the curvature tensor of $g$ is $g_{ac}g_{bd}-g_{ad}g_{bc} + J_{ac}J_{bd}-J_{ad}J_{bc}+2J_{ab}J_{cd}$. The (real) Ricci tensor is $\mathrm{Ric} = (2n+2)g$ and the Schouten tensor is $P=\frac{n+1}{2n-1}g$. It follows that the Weyl curvature is given by $W_{abcd} = \frac{-3}{2n-1}(g_{ac}g_{bd}-g_{ad}g_{bc}) + J_{ac}J_{bd}-J_{ad}J_{bc}+2J_{ab}J_{cd}$. 
Following the classification in \cite{ChenOguie1974} the umbilic submanifolds are of three kinds. The first case is when $\Sigma$ is a totally geodesic $\mathbb{CP}^m$ in $\mathbb{CP}^n$ ($m<n$). In this case one readily computes that the Fialkow tensor is given by $\cF=(\frac{n+1}{2n-1}-\frac{m+1}{2m-1})g_{\Sigma}$ when $m>1$ and it turns out that the same formula holds when $m=1$ (since in the latter case the ambient Weyl curvature pulls back to $\Sigma = \mathbb{CP}^1$ as $W_{ijkl} = (\frac{-3}{2n-1}+3)(g_{ik}g_{jl}-g_{il}g_{jk})$ so that $\mathrm{tr}^2 \,\iota^* W = 12\frac{n-1}{2n-1}$ and $\cF=-3\frac{n-1}{2n-1}g_{\Sigma} = (\frac{n+1}{2n-1}-2)g_{\Sigma}$). So, a totally geodesic $\mathbb{CP}^m$ in $\mathbb{CP}^n$ is conformally circular, but not strongly conformally circular. The second case is when $\Sigma$ is a totally real and totally geodesic $\mathbb{RP}^n$ in $\mathbb{CP}^n$ or a totally geodesic $\mathbb{RP}^m$ inside of such an $\mathbb{RP}^n$. The third case is when $\Sigma$ is an umbilic, but not totally geodesic, $\mathbb{RP}^m$ in such a totally real and totally geodesic $\mathbb{RP}^n$ (note that such an embedded $\mathbb{RP}^m$ necessarily has parallel mean curvature in $\mathbb{RP}^n$ and in $\mathbb{CP}^n$); in this case the induced metric has constant sectional curvature $K>1$ and the mean curvature squared is given by $|H|^2 = K-1$ (when $K=1$ we are back in the second case). Interestingly, in either the second or the third case we find that when $\dim \Sigma > 2$ the Fialkow tensor is given by $\cF = \frac{3}{4n-2}g_{\Sigma}$ (the dependence on the intrinsic sectional curvature $K$ of $\Sigma$ drops out). When $\dim \Sigma =2$ in either the second or the third case the ambient Weyl tensor pulls back to $\Sigma$ as $W_{ijkl} = \frac{-3}{2n-1}(g_{ik}g_{jl}-g_{il}g_{jk})$ so that $\mathrm{tr}^2\,\iota^*W = \frac{-6}{2n-1}$ and once again finds that $\cF = \frac{3}{4n-2}g_{\Sigma}$. Hence for a totally real umbilic submanifold $\Sigma\subset \mathbb{CP}^n$ (of dimension at least two) the Fialkow tensor is universally given by $\cF = \frac{3}{4n-2}g_{\Sigma}$. It follows that the umbilic submanifolds of $\mathbb{CP}^n$ are all conformally circular, but none are strongly conformally circular.
}
\end{example}

Note that in the above example we have only described the conformally
distinguished submanifolds of $\mathbb{CP}^n$ of dimension at least
two. The $1$-dimensional distinguished submanifolds in complex
projective spaces have been characterised in \cite{AMU-circles-1995}
(strictly speaking, the paper is concerned with ``Riemannian'' circles
in $\mathbb{CP}^n$, but since the Fubini-Study metric $g_{FS}$ is
Einstein these are precisely the conformal circles of
$(\mathbb{CP}^n,[g_{FS}])$ parametrized by arclength with respect to
$g_{FS}$); see also \cite{dunajski2019conformal}, where the results of \cite{AMU-circles-1995} are recovered for $\mathbb{CP}^2$ using the first integrals arising from the conformal Killing-Yano tensors on $\mathbb{CP}^2$ (a construction which we generalise in Section \ref{fi-Sect} below). The notion of

Having discussed the umbilic submanifolds in complex space forms (K\"ahler manifolds with constant holomorphic sectional curvature) it is natural to next consider the quaternionic K\"ahler analogues:

\begin{example}\label{example-q-space-form} \emph{\textbf{(Umbilic submanifolds in quaternionic space forms.)} Umbilic submanifolds in quaternionic space forms are classified in \cite{Chen-quaternion}, and all are extrinsic spheres (again, we are counting totally geodesic submanifolds as extrinsic spheres). Since quaternionic space forms are Einstein, it follows that every umbilic submanifold in a quaternionic space form is conformally distinguished.
}
\end{example}

Summarizing what we have learned from Examples \ref{example-model}, \ref{example-c-space-form} and \ref{example-q-space-form} we have the following theorem:
\begin{thm}
An umbilic submanifold in a real, complex or quarternionic space form is a confomally distinguished submanifold.
\end{thm}

The various kinds of space forms we discussed above were Einstein (and the umbilics were all extrinsic spheres). Sasakian space forms, however, are typically only ``$\eta$-Einstein'' (meaning that eigenvalues of the Ricci tensor in the horizontal and characteristic directions are two, possibly different, constants). Nevertheless, in most cases the umbilic submanifolds are still forced to be distinguished:

\begin{example}\label{example-s-space-form} \emph{\textbf{(Umbilic submanifolds in Sasakian space forms.)}
Let $(M,g)$ be a Sasakian space form of dimension $2n+1$ and $\phi$-sectional curvature $c$. Let $\xi$ denote the characteristic direction of $g$ and $\eta$ the corresponding $1$-form (the underlying contact form of the Sasakian structure). The Ricci curvature of $g$ is given by $\mathrm{Ric} = \frac{1}{2}(n(c+3)+c-1)g - \frac{1}{2}(n+1)(c-1)\eta\otimes\eta$ \cite{Blair}. Given any submanifold $\Sigma$ of $M$ we may decompose the characteristic direction $\xi$ along $\Sigma$ as $\xi = \xi^{\top} + \xi^{\perp}$, where $\xi^{\top}$ is tangent and $\xi^{\perp}$ is normal to $\Sigma$. Writing $\eta^{\top}$ and $\eta^{\perp}$ for the corresponding $1$-forms, the mixed part of the Schouten tensor is $-\frac{n+1}{2(2n-1)}(c-1)\eta^{\top}\otimes\eta^{\perp}$, or $-\frac{n+1}{2(2n-1)}(c-1)\eta^{\top}\otimes\xi^{\perp}$ when viewed as a normal bundle valued $1$-form on $\Sigma$. On the other hand, when $\Sigma$ is umbilic it is shown (by contracting the Codazzi equation) in \cite{BlairVanhecke} that $\nabla^{\perp} H = -\frac{1}{4}(c-1)\eta^{\top}\otimes \xi^{\perp}$. Combining these observations gives that, for an umbilic hypersurface in $(M,g)$, $\mu = -\frac{3(c-1)}{4(2n-1)}\eta^{\top}\otimes \xi^{\perp}$. In the classification of umbilic submanifolds in Sasakian space forms in  \cite{BlairVanhecke} there are four classes of examples. In classes (i)--(iii) $(c-1)\eta^{\top}\otimes \xi^{\perp}=0$; indeed, (i) corresponds precisely to the case when $\eta^{\top}=0$ ($\Sigma$ is everywhere tangent to $\mathrm{ker}\,\eta$), (ii) to the case when $\xi^{\perp}=0$ ($\xi$ is tangent to $\Sigma$) and (iii) to the case when $c=1$ (in which case $(M,g)$ is a real space form). Thus, umbilic submanifolds in classes (i)--(iii) are all distinguished. Class (iv), on the other hand, corresponds precisely to those umbilic submanifolds for which $\mu\neq 0$; it is shown in \cite{BlairVanhecke} that these occur when $c<-3$, equivalently, when $g$ has strictly negative sectional curvatures in the ordinary Riemannian sense.}
\end{example}

In particular, we have:

\begin{thm}
  Let $(M,g)$ be a Sasakian space form of $\phi$-sectional curvature $c\geq -3$, then every umbilic submanifold of $(M,g)$ is conformally distinguished. When $c<-3$ there are umbilic subamnifolds in $(M,g)$ that are not conformally distinguished.
\end{thm}

We conclude with an example that motivated the definition and consideration of the invariants $\LL$ and $\cF$:
\begin{example}\label{ex-conf-infty-dist} \emph{\textbf{(Conformal infinities of Poincar\'e-Einstein manifolds.)}
Let $(M,g)$ be a Poincar\'e-Einstein manifold with conformal infinity $(\Sigma, \bc_{\Sigma})$. Then $\Sigma$ is a strongly conformally circular submanifold in $\overline{M}$. This follows from  \cite[Theorem 4.5]{GRiemSig} where it is shown that the tractor bundle $\cT\Sigma$ is parallel as a subbundle of the standard tractor bundle of $\overline{M}$ (which is equivalent to $\LL=0$) and the ambient tractor connection induces the standard submanifold tractor connection on $\cT\Sigma$ (which is equivalent to $\cF=0$). The key observation behind this result is that the scale tractor $I^A$ corresponding to the Einstein metric $g$ on $M$ extends continuously to $\overline{M}$ to give the normal tractor $N^A$ to $\Sigma$ on $\Sigma$; the vanishing of $\LL$ then immediately follows from the fact that $I^A$ is parallel, and the fact that $\cF=0$ can be seen from the fact that $\nabla_{[a}\nabla_{b]}I^C=0$ (since from this one can deduce that the induced tractor connection $\check{\nabla}$ on $\cT\Sigma$ is normal and therefore must agree with the submanifold one).
}
\end{example}


One of our motivations for the consideration of $\LL$ and $\cF$ in the
hypersurface case (in which case $\LL$ is equivalent to
$\mathring{\II}$) is that if $(M,g)$ is a conformally compact manifold
with conformal infinity $(\Sigma, \bc_{\Sigma})$ then the trace-free
second fundamental form $\mathring{\II}$ of $\Sigma$ represents the
first order obstruction to the existence of a (formal)
Poincar\'e-Einstein metric on $M$ in the conformal class of $g$, and
$\cF$ represents the next order obstruction (recall that in the
Poincar\'{e}-Einstein case one obtains that $\mathring{\II}=0$ from
$\nabla_aI^B|_{\Sigma}=0$ and that $\cF=0$ from
$\nabla_{a}\nabla_{b}I^C|_{\Sigma}=0$); see \cite{BGW} for further
development of this idea in the hypersurface case.

While we have given many examples above of distinguished submanifolds
and of conformally circular manifolds above, we are just as much
interested in situations where these conditions fail (as measured by
the nontriviality of $\cF$ and $\LL$). For hypersurfaces, we can
interpret this failure either in terms of the difference between
intrinisic and ambient conformal circles or in terms of the failure to
be (formally) the conformal infinity of a Poincar\'e-Einstein metric to
a given order. The conformal circles interpretation of course holds in
all codimensions, and it seems natural to ask whether these invariants
are also related to the possibility of realising a general codimension
submanifold $\Sigma$ as a zero locus associated to a (formally)
parallel tractor field along $\Sigma$. We will return to this idea in
Section \ref{corbit-sec}, after considering the interaction between
distinguished submanifolds and parallel tractors in Section
\ref{fi-Sect}.

\section{First integrals} \label{fi-Sect}

Here we show that a class of solutions to a very large collection of linear
differential equations provide first integrals for distinguished
submanifolds.  This provides a uniform framework which generalises to
submanifolds (of any proper codimension) the advance for conformal
circles in \cite{GST}.  In \cite{GST} it is explained in detail how
the ideas there extend the usual construction of first integrals for
geodesics, using for example solutions of the Killing equation,
Killing tensors, and Killing-Yano tensors. So we do not repeat that
here.

\subsection{Review of relevant BGG theory}\label{BGGsec}

The class of equations that we interested are the so-called
(conformal) first BGG equations. This is a very large class of
conformally invariant linear overdetermined PDE. It includes the
conformal Killing equation, more generally the conformal Killing
tensors equations of any rank, the conformal Killing-Yano equations.
To understand this infinite class of equations we recall here some
elements of the BGG theory. To put this into context we first recall the homogeneous model for conformal geometry, discussed in Section \ref{tractor_calc}. The model for oriented conformal geometries of Riemannian signature is conformal $n$-sphere $(S^n,\bc)$ viewed as the ray projectivisation of the forward null cone in $(n+2)$-dimensional Minkowski space. The group $G=\mathrm{SO}_0(n+1,1)$ acts on the forward null cone and descends to an action by conformal isometries of $S^n$; the conformal $n$-sphere is therefore naturally viewed as a homogeneous geometry on $G/P\cong S^n$, for $P$ an appropriate (parabolic) subgroup of $G$. Again, see, e.g., \cite{Curry-G-conformal,GRiemSig} for
a more detailed discussion.

Generalising from the model case, it is well known that a  conformal manifold $(M,\cc)$ (of dimension $n\geq 3$) determines a canonical {\em Cartan bundle and connection} (the additional choice of a M\"obius structure is required for this in dimension $n=2$). This consists of  
a $P$-principal bundle $\cG\to M$ equipped with a canonical {\em
  Cartan connection } $\om$ which is a suitably equivariant
$\mathfrak{g}$-valued 1-form that provides a total parallelisation of
$T\cG$. Here $\mathfrak{g}$ denotes the Lie algebra of $G$. In the
case of the model, $\mathcal{G}=G$ and $\om$ is the Maurer-Cartan form.

For any representation $\mathbb{U}$ of $P$, one has a corresponding
associated bundle $\cG \times_P \mathbb{U}$. For example it follows
from the equivariance properties of $\om$ that the tangent bundle $TM$
can be identified with $ \cG \times_P (\mathfrak{g} /\mathfrak{p} )$ where
$\mathfrak{p}$ is the Lie algebra of $P$, and the $P$ action is induced
from its adjoint action on $\mathfrak{g}$.

The tractor bundles are the associated bundles $\cW:=\cG\times_P
\mathbb{W}$ where $\mathbb{W}$ is a linear representation space of $G$
(and hence also of $P$ by restriction). On each of these the
Cartan connection induces a linear connection $\nabla^{\mc{W}}$ and this the tractor connection for the given bundle. In
particular the standard tractor bundle $\cT$ is $\cW:=\cG\times_P
\mathbb{R}^{n+2}$, with $\mathbb{R}^{n+2}$ denoting the defining
representation of $G$. From the latter (for example) the Cartan bundle $\mathcal{G}$ can be recovered as an adapted frame bundle and, on this, the Cartan connection $\om$ can be recovered from the tractor connection, see \cite{CG-tams}.

 Now recall the bundle embedding \eqref{Xinj} (with $k=2$)
$$
  \bX: T^*M \to \Lambda^2\cT \subset \End (\cT),
  $$
  where the tractor metric is used in the obvious way to identify
  elements of $ \Lambda^2\cT$ with skew elements of $\End (\cT)$.
  Sections of $\End (\cT)$ act on tractor bundles in the obvious
  tensorial way and so, via each respective $\mathbb{X}$, we have a
  canonical action of $T^*M$ on any tractor bundle $\cV$ and this
  induces a sequence of invariant bundle maps
\begin{equation}\label{Kostant}
  \partial^* : \Lambda^k T^*M\otimes \cV \to  \Lambda^{k-1} T^*M\otimes \cV,\quad k=1,\cdots,n+1.
\end{equation}
This is the (bundle version of the) Kostant codifferential for
conformal geometry and satisfies
$\partial^*\circ \partial^*=0$; so it determines subquotient bundles
$\cH_k(M,\cV):=\operatorname{ker}(\partial^*)/\operatorname{im}(\partial^*)
$ of the $\cV$-valued tractor bundles $\Lambda^k T^*M\otimes \cV
$.

Now, for each tractor bundle $\cV=\cG\times_P \mathbb{V}$, with $\mathbb{V}$ irreducible for $G$, there is a canonical differential \textit{BGG-sequence} \cite{CSS2000,CD},
$$
  \cH_0 \stackrel{\cD^{\cV}_0}{\rightarrow} \cH_1 \overset{\cD^{\cV}_1}{\rightarrow}\cdots
  \overset{\cD^{\cV}_{n-1}}{\rightarrow} \cH_n \, .
$$ Here $\cH_k= \cH_k(M,\cV)$ and each $\cD^{\cV}_i$ is a linear
  conformally invariant differential operator.

  We are, in particular, interested
  in the operator $\cD^\cV=\cD^\cV_0$, which defines an overdetermined
 differential system.  The parabolic subgroup $P\subset G$ determines a filtration
  on $\mathbb{V}$ by $P$--invariant subspaces. Denoting the largest
  proper filtration component by $\V^0\subset \V$, it is straightforward to show that $\cH_0$ is the
  quotient $\cV/\cV^0$. Here, $\cV^0$ is the corresponding associated
  bundle for $\V^0$, and we write $\pi: \cV \to \cH_0$ for the natural
  projection.
We recall here the construction of the first BGG operators
$\cD^\cV$, as summarised in \cite{CGH}, and also the definition of the special class of so
called normal solutions (cf.\ \cite{Leitner}) for these
operators.

\begin{thm}[\cite{CGH}] \label{normp}
  Let $\mathbb{V}$ be an irreducible $G$-representation and let $\cV: =\cG\times_P \mathbb{V}$.
  There is a unique  invariant differential operator
  $L: \cH_0 \to  \cV$ such that $\pi\circ L$ is the identity map on $\cH_0$ and
  $\nabla \circ L$ lies in $\operatorname{ker}(\partial^*)\subset T^*M\otimes \cV$.
  For $\si\in \Gamma(\cH_0)$, $\cD^\cV \si$ is given by projecting $\nabla (L(\si))$ to $\Gamma (\cH_1)$, i.e. $\cD^\cV \si = \pi(\nabla(L(\si)))$.

  Furthermore  the bundle map
  $\pi$ induces an injection from the space of parallel sections of
  $\mathcal{V}$ to a subspace $\mathfrak{N}(\cD^\cV) $ of $\Gamma(\mathcal{H}_0)$ which is
  contained in the kernel of the first BGG operator
  \begin{equation}\label{fBGG}
    \cD^\cV : \cH_0 \to \cH_1 \, .
  \end{equation}
The operator $L$ restricts to an isomorphism from $\mathfrak{N}(\cD^\cV) $ to the space of parallel tractors in $\Gamma(\cV)$.
\end{thm}
\noindent The differential operator $L: \cH_0 \to \cV$, in the Theorem, is called a {\em BGG splitting operator}. We
sometimes denote this $L^\cV$ to emphasise the particular tractor
bundle involved. Using the notation and setting of the Theorem, we also
use the following terminology:
\begin{defn}
  \emph{Elements of the subspace $\mathfrak{N}(\cD^\cV)\subset \Gamma(\cH_0)$ are called \emph{normal} solutions to the BGG equation $\cD^\cV \si = 0$.}
\end{defn}

By definition normal solutions to \eqref{fBGG} are in 1-1
correspondence with parallel sections of the corresponding tractor
bundle $\cV$. On geometries which are conformally flat all solutions
are normal, and clearly there is
$\operatorname{dim}(\mathbb{V})$-parameter family of such normal
solutions locally.

For the standard tractor bundle the corresponding first BGG equation is the equation
\begin{equation}\label{eq:AE}
\nabla_{(a}\nabla_{b)_0} \si +P_{(ab)_0}\si=0,
\end{equation}
on sections $\si\in \Gamma(\ce[1])$, and all solutions are normal (on any conformal manifold admitting such solutions).
However, this is not typical.  In general, for solutions $\si\in
\Gamma(\cH_0) $ of $\cD^\cV(\si) =0$, $\nabla L(\si)$ is given by
curvature terms acting on $L(\si)$ (see, e.g., \cite{Cap-infin,GLaplacianEinstein,G-Sil-ckforms,HSSS-ex}). Normal
solutions, for which these curvature terms necessarily annihilate
$L(\si)$, often correspond to interesting geometric conditions on the
underlying manifold.

\subsection{The First Integral Theorem}\label{fi-thm-sec}

We work on an arbitrary conformal manifold $(M^n,\bc)$. Let $\Sigma$
be an embedded submanifold of codimension $d$. Recall that $\Sigma$
determines its normal form $N_{A_1\cdots A_d}\in
\Gamma(\Lambda^d\cN)$. This is parallel if (and only if) $\Sigma$ is
distinguished. Thus if the manifold $(M^n,\bc)$ is
equipped with a parallel tractor $S$ that can be contracted
non-trivially into say $m_0$ copies of $N_{A_1\cdots A_d}$ to yield a
function, then this scalar is necessarily constant if $\Sigma $ is
distinguished. Thus we obtain a first integral for such $\Sigma$.  In
general the parallel tractor $S$ would not necessarily itself come
from a $G$-irreducible representation, but rather a tensor product of
such. Thus we have the following result.

As earlier, view
$\mathbb{R}^{n+2}$ as the defining representation for $G:=\SO(h)\cong
\SO(p+1,q+1)$.
Define
$$
  \mathbb{W}(d):= \Lambda^d\mathbb{R}^{n+2} \qquad d=1,\cdots,n-1.
$$
For each $d$, this is also a representation space for $G$. Then we have:
\begin{thm}\label{fi-thm}
  Let $\mathbb{V}_1,\cdots ,\mathbb{V}_k$ be irreducible
  representation spaces of $G$,
  $\cV_i=\cG\times_P \mathbb{V}_i$, and $\cD^{\cV_i}$, $i\in
    \{1,\cdots ,k\}$ the corresponding respective first BGG operators.

  For each $i\in \{1,\cdots ,k\}$, suppose that $\si_i$ is a normal
  solution to the first BGG equation
  \begin{equation}\label{BGGi}
    \cD^{\cV_i} \si_i =0 ,
  \end{equation}
  and $m_i\in\mathbb{Z}_{\geq 0}$.
  Then for each copy of the trivial $G$-representation $\mathbb{R}$ in
  \begin{equation}\label{prod}
    (\odot^{m_0} \mathbb{W}(d) )  \otimes (\odot^{m_1} \mathbb{V}_1)\otimes\cdots \otimes (\odot^{m_k}\mathbb{V}_k)
  \end{equation}
  there is a corresponding distinguished first integral for submanifolds of codimension $d$.
\end{thm}
\begin{proof} The proof is an easy consequence of the reasoning above. Otherwise the formal proof is a trivial adaption of the proof of Theorem 6.1 in \cite{GST}, which treats the case of curves.  
  \end{proof}

\noindent The theorem has used the normal form $N_{A_1\cdots A_d}$ as the basis for
producing first integrals. One can equivalently use its Hodge dual
$\star N_{A_1\cdots A_{m+2}}$, or the normal projector $N^A_B$, or any
combination of these, as by Theorem \ref{key1} any of these are
parallel for distinguished submanifolds. 

 Note that to apply \eqref{prod} of the Theorem for a given
$\Sigma$ we require normal solutions to $k$ first BGG equations. For
case of curves, several examples are given in \cite{GST}, as is also
the explanation of how this is linked to familiar first integrals for
geodesics as obtained from Killing vectors and Killing tensors (which
are solutions of projective BGG equations). Given that resource we
treat just one example here.

\subsection{First integrals from a normal conformal Killing-Yano form}
\label{ex-sect}

We give an example to show how this machinery yields conserved
quantities for distinguished submanifolds. It is easy to follow the
ideas here to produce other examples, see \cite{GST} for the case of $m=1$.

The space
$\cE_{a_1 [ a_2 \cdots a_d]} [w] = \cE_{a_1} \otimes \cE_{[a_2 \cdots
    a_d]} [w]$ is completely reducible under the action of $O(g)$, and has the decomposition
\begin{equation}\label{eqn_killing_yano_decomp}
  \cE_{a_1 [a_2 \cdots a_d] } [w] = \cE_{[a_1 a_2 \cdots a_d]} [w] \oplus \cE_{{\{ a_1 [ a_2 \cdots a_d ] \}}_0} [w] \oplus \cE_{ [ a_3 \cdots a_d ]} [w-2],
\end{equation}
where $\cE_{{\{ a_1 [ a_2 \cdots a_d ] \}}_0} [w]$ consists of tensors
$s_{ a_1 \cdots a_d}\in \cE_{a_1 [ a_2 \cdots a_d] } [w] $ which are,
metric trace-free, completely skew on the indices $ a_2, \ldots a_d$,
and for which $s_{[a_1 a_2 \cdots a_d]} = 0$.  A $(d-1)$-form $k_{a_2 \cdots a_d} \in \Gamma(\cE_{[a_2 \cdots a_d]} [d])$ is
said to be a \emph{conformal Killing-Yano form} or simply
\emph{conformal Killing form} if it satisfies
\begin{equation}\label{eqn_conf_killing_form}
  \nabla_{a_1} k_{ a_2 \cdots a_d } = \phi_{a_1 \cdots a_d} + \bg_{a_1 [a_2} \nu_{a_3 \cdots a_d]},
\end{equation}
where $\phi_{a_1 \cdots a_d} \in \cE_{[a_1 \cdots a_d]}[d]$ and
$\nu_{a_3 \cdots a_d} \in \cE_{[a_3 \cdots a_d]} [d-2]$.
Equivalently,
\begin{equation}\label{eqn_conf_killing_form_alt}
  \nabla_{\{a_1} k_{{a_2 \cdots a_d \}}_0} = 0,
\end{equation}
where the braces and subscript zero denote projection onto the middle
factor of~\eqref{eqn_killing_yano_decomp}.  This equation can be
checked to be conformally invariant, and is moreover a first BGG
equation (which in this context implies conformal invariance).  Thus
solutions to this equation correspond bijectively to a class of
sections of a certain tractor bundle.  To understand this, we proceed
as follows.  For this equation, it is shown in  \cite{G-Sil-ckforms}
that the corresponding tractor bundle is $\Lambda^{d}\cT$, and it follows from the formulae there  that
the BGG splitting operator $L : \cE_{[ a_2 \cdots a_d]} [d] \to
\cE_{[A_1 \cdots A_d]}$ is:
\begin{equation} \label{eqn_conf_killing_splitting_op}
  \begin{split}
  L(k_{a_2 \cdots a_d})
  &= k_{a_2 \cdots a_d} \mathbb{Y}_{ A_1 \cdots A_d} ^{\phantom{A_1} a_2 \cdots a_d}
  + \frac{1}{d} \nabla_{a_1} k_{a_2 \cdots a_d} \,\mathbb{Z}_{A_1  \cdots A_d} ^{a_1 \cdots a_d} + \frac{d-1}{n-d+2} \nabla^c k_{c a_3 \cdots a_d} \mathbb{W}_{A_0 A_1 A_2 \cdots A_d} ^{\phantom{A_0 A_1} a_3 \cdots a_d}  \\
  &\quad \;-\left(\frac{1}{n(d-1)}\nabla^b\nabla_{\{b} k_{a_2 \cdots a_d\}_0} - \frac{1}{n-d+2}\nabla_{[a_2|}\nabla^b k_{b |a_3 \cdots a_d]} - P_{[a_2|}{}^bk_{b |a_3 \cdots a_d]} \right) \mathbb{X}_{A_1 A_2 \cdots A_d} ^{\phantom{A_1} a_2 \cdots a_d}.
\end{split}
\end{equation}
The general theory immediately gives us the following.

\begin{prop}
  Let $k_{a_1 \cdots a_{d-1}} \in \cE_{[a_1 \cdots a_{d-1}]} [d]$ be
  a normal solution to the conformal Killing-Yano equation and
  $\Sigma$ a distinguished submanifold of codimension $d$, with
  corresponding tractor normal form $N_{A_1 \cdots A_d}$.  Let
  $\mathbb{K}_{A_1 \cdots A_d} := L(k_{a_1 \cdots a_{d-1}}) \in
  \cE_{[A_1 \cdots A_d]}$ be the image of $k_{a_1 \cdots a_{d-1}}$
  under the BGG splitting operator $L$
  of~\eqref{eqn_conf_killing_splitting_op}.  Then the scalar function
  \begin{equation}\label{cck}
    \mathbb{K}_{A_1 \cdots A_d} N^{A_1 \cdots A_d}
  \end{equation}
  is constant along
  $\Sigma$.
\end{prop}

\begin{rem}
\emph{(i)} Here $d\geq 2$, but the result still holds in the case where $d=1$ of we understand the hypothesis in this case to mean that $\sigma = k \in \cE[1]$ satisfies the almost-Einstein equation \eqref{eq:AE} and take $L$ to be the corresponding BGG splitting operator. \emph{(ii)} Note that one can of course use $\star\hspace{-0.5pt}N$ rather than $N$ to construct conserved quantities, and in this case we obtain a conserved quantity $\mathbb{K}_{A_1 \cdots A_{m+2}} \star\hspace{-0.5pt}N^{A_1A_2 \cdots A_{m+2}}$ when $k_{a_1\cdots a_{m+1}}$ is a normal solution to the conformal Killing-Yano equation (this is the approach used to construct first integrals for conformal circles in \cite[Theorem 6.8]{GST}). However, since the Hodge-$\star$ operator takes conformal Killing forms to conformal Killing forms, we obtain the same first integrals this way.
\end{rem}

\begin{proof}
  Since $k_{a_1 \cdots a_{d-1}}$ is a normal solution, we have that
  $\nabla_i \mathbb{K}_{A_1 \cdots A_d} = 0$.  Moreover, since
  $\Sigma$ is a distinguished submanifold, $\nabla_i N^{A_1 \cdots
    A_d} = 0$ by Theorem \ref{key1}.  Hence the scalar quantity
  $\mathbb{K}_{A_1 \cdots A_d} N^{A_1 \cdots A_d}$ is constant.
\end{proof}

We show the non-triviality of
the first integral quantity \eqref{cck} by calculating it directly.
From the explicit forms of $\mathbb{K}_{A_1 \cdots A_d}$ and $N_{A_1 \cdots A_d}$, we see that 
  \begin{align}
    \mathbb{K}_{A_1 \cdots A_d} N^{A_1 \cdots A_d}
    &= d \cdot k_{a_1 \cdots a_{d-1}} N^{c b_1 \cdots b_{d-1}} H_c \cdot \mathbb{Y}_{A_1 A_2 \cdots A_d} ^{\phantom{A_1} a_2 \cdots a_d} \mathbb{X}^{A_1 A_2 \cdots A_d} _{\phantom{A_0} b_1 \cdots b_{d-1}} \nonumber \\
    &\quad+ \frac{1}{d} \left( \nabla_{a_1} k_{a_2 \cdots a_d} \right) N^{b_1 b_2 \cdots b_d} \cdot \mathbb{Z}_{A_1 A_2 \cdots A_d} ^{a_1 a_2 \cdots a_d} \mathbb{Z}^{A_1 A_2 \cdots A_d} _{b_1 b_2 \cdots b_d} \nonumber \\
    &= k_{a_1 \cdots a_{d-1}} N^{c a_1 \cdots a_{d-1}} H_c + \frac{1}{d} \left( \nabla_{a_1} k_{a_2 \cdots a_d} \right) N^{a_1 a_2 \cdots a_d}, \label{eqn_K_N_explicit}
  \end{align}
  which verifies non-triviality.

  \medskip

  For the case of $d=n-1$, meaning curves, it was seen in \cite{GST}
  that, for many examples, normality of the BGG solution is actually
  not required in order to obtain a first integral. In the general codimension case, however, when the BGG solution is not normal we may or may not obtain a conserved quantity depending on the situation. The condition required for \eqref{cck} to give rise to a conserved quantity is weaker than the normality of $\mathbb{K}$, and depends on the submanifold $\Sigma$, which we see as follows: If $\Sigma$
  is distinguished then $N^{A_1 \cdots A_d}$ is parallel for the
  tractor connection and we have
    \begin{equation*}
    \nabla_i \left( \mathbb{K}_{A_1 \cdots A_d} N^{A_1 \cdots A_d} \right) = \left( \nabla_i \mathbb{K}_{A_1 \cdots A_d} \right) N^{A_1 \cdots A_d},
  \end{equation*}
where $\mathbb{K}_{A_1 \cdots A_d}=L(k)$ for a general rank $(d-1)$
conformal Killing-Yano form $k$.
Theorem 3.9 of~\cite{G-Sil-ckforms}
gives
  \begin{equation*}
    \left( \nabla_c - \Psi_c \right) \mathbb{K}_{A_1 \cdots A_d} = 0,
  \end{equation*}
  where $\nabla_c$ is the standard tractor connection and $\Psi_c : \cE_{[A_1 \cdots A_d]} \to \cE_{c[A_1 \cdots A_d]}$ is defined by 
  \begin{equation}
    \begin{split}
      \Psi_c ( \mathbb{K}_{A_1 A_2 A_3 \cdots A_d})
    &:= - \frac{1}{2} W_{a_1 a_2 c} {}^e k_{e a_3 \cdots a_d} \mathbb{Z}_{A_1 A_2 A_3 \cdots A_d} ^{a_1 a_2 a_3 \cdots a_d} + \phi_{c a_3 \cdots a_d} \mathbb{W}_{A_1 A_2 A_3 \cdots A_d} ^{\phantom{A_1 A_2} a_3 \cdots a_d} \\
    &\hspace{4em}+ \xi_{a_2 \cdots a_d} \mathbb{X}_{A_1 A_2 \cdots A_d} ^{\phantom{A_0} a_2 \cdots a_d},
  \end{split}
  \end{equation}
  where only the explicit form of the $\mathbb{Z}$ slot will be important.
  Therefore one has 
  \begin{align}
\nonumber    \nabla_i \left( \mathbb{K}_{A_1 A_2 \cdots A_d} N^{A_1 A_2 \cdots A_d} \right)
    &= \left( \nabla_i \mathbb{K}_{A_1 A_2 \cdots A_d} \right) N^{A_1 A_2 \cdots A_d} \\
\nonumber    &= \Psi_i (\mathbb{K}_{A_1 A_2 \cdots A_d}) N^{A_1 A_2 \cdots A_d} \\
\nonumber    &= -\frac{1}{2} W_{a_1 a_2 i} {}^e k_{e a_3 \cdots a_d} N^{b_1 b_2 b_3 \cdots b_d} \cdot \mathbb{Z}_{A_1 A_2 A_3 \cdots A_d} ^{a_1 a_2 a_3 \cdots a_d} \mathbb{Z}^{A_1 A_2 A_3 \cdots A_d} _{b_1 b_2 b_3 \cdots b_d} \\
\label{eq:obs-to-cc}    &= -\frac{1}{2} W_{a_1 a_2 i} {}^e k_{e a_3 \cdots a_d} N^{a_1 a_2 a_3 \cdots a_d}, 
  \end{align}
the vanishing of which is a weaker condition than normality. 

To see that \eqref{eq:obs-to-cc} may indeed be zero or not zero when $k$ is not normal, we consider the following simple examples:
\begin{example} \emph{\textbf{(Non-normal BGG solutions and conserved quantities.)} 
Let $M=S^2\times S^2$, equipped with the standard product metric $g$
and let $g_i$, $i=1,2$, denote the pullback of the standard round metric on
$S^2$ by the projection on the $i$th factor (so $g=g_1+g_2$). Then the
Weyl tensor of $g$ is given in terms of the Kulkarni-Nomizu products
of $g_1$ and $g_2$ by $W=\frac{1}{3}(g_1 \KN g_1 - g_1\KN g_2 + g_2\KN
g_2)$. When the codimension $d$ is $2$ (so $d-1=1$) the BGG solutions
$k$ appearing in \eqref{eq:obs-to-cc} are just the conformal Killing
fields of $(M,[g])$, of which there are many (e.g., the trivial lift
of a Killing field of one of the round $S^2$ factors). From the
formula for the Weyl tensor it is easy to see that none of the
conformal Killing fields on $(M,[g])$ are normal (as the map $k^a
\mapsto k^a W_{abcd}$ is injective). In this case, given a
distinguished submanifold $\Sigma$, the quantity in
\eqref{eq:obs-to-cc} becomes simply
$-\frac{1}{2}W_{abic}k^c\mathrm{N}^{ab}$. From the formula for $W$ it
is easy to see that this always gives zero in the case where $\Sigma$
is $S^{2}\times \{p\}$ or $\{p\}\times S^2$ for some $p\in S^2$ (since
then
$W_{abic}\mathrm{N}^{ab}=W_{abec}\Pi^e_i\mathrm{N}^{ab}=0$). Thus, in
these cases $\mathbb{K}_{AB}N^{AB}$ is a conserved quantity (i.e.\ is
constant) along $\Sigma$ for any conformal Killing vector field
$k$. On the other hand, if $\Sigma$ is the diagonal submanifold in
$M=S^2\times S^2$ then only for a subset of the conformal killing
fields do we obtain a conserved quantity in this way. To see this, let
$(p,p)\in \Sigma$, let $\tilde{X}, \tilde{Y}$ normal vectors to
$\Sigma$ at $(p,p)$ with projection on the first factor given by
$X,Y\in T_p S^2$ respectively, let $\tilde{Z}$ be a tangent vector to
$\Sigma$ at $(p,p)$ with projection on the first factor given by
$Z\in T_p S^2$, and let $k$ be a conformal Killing vector field on
$M$ with components in the direction of the first and second factors
at $(p,p)$ given by $k_1,k_2\in T_pS^2$, respectively. Then, noting
that $\tilde{X}$, $\tilde{Y}$ and $\tilde{Z}$ are all determined by
their components in the direction of the first factor, a
straightforward calculation shows that at $(p,p)$,
$W(\tilde{X},\tilde{Y},\tilde{Z},k) = \langle X, Z\rangle (\langle Y,
k_1\rangle + \langle Y, k_2\rangle) - \langle Y, Z\rangle (\langle X,
k_1\rangle + \langle X, k_2\rangle)$, where $\langle\, ,\, \rangle$
denotes the round metric on $S^2$. From this we see that
$\mathbb{K}_{AB}N^{AB}$ is constant along $\Sigma$ when $k$ is orthogonal
to $\Sigma$ (i.e.\ when $k_2=-k_1$ at each point $(p,p)\in \Sigma$)
but not otherwise.  }
\end{example}

\section{Distinguished submanifolds from curved orbits}\label{corbit-sec}

In this section we show that distinguished submanifolds arise naturally as  {\em curved orbits}, in the sense of \cite{CGH-Duke}, in the presence of certain Cartan holonomy reductions of the conformal structure. We then show (without using the curved orbit theory of \cite{CGH-Duke}) that the same continues to hold under weaker hypotheses. Before coming to the curved orbit theory result and its generalisation, however, we prove Theorem \ref{thm_submanifold_gst}, as this will be needed in the discussion that follows. 

\subsection{Distinguished submanifolds via a moving incidence relation}

Recall that a conformally embedded submanifold, of codimension $d$, determines the
fundamental and equivalent objects $N^A_B$, $N_{A_1\cdots A_d}$, and
$\SN^{A_1\cdots A_{m+2}}$ and then we have Theorem
\ref{key1}. For our purposes, however, it is important
to have a characterisation of distinguished submanifolds
that does not use an initial knowledge of these. Theorem \ref{thm_submanifold_gst} gives us such a characterisation, which we state more explicitly here:
\begin{thm}\label{thm_submanifold_gst-main}
  Let $\Sigma \hookrightarrow M$ be a submanifold of
  codimension $d$ in a conformal manifold $(M, \bm{c})$.  Then
  $\Sigma$ is distinguished if, and only if, there exists a
  nowhere-zero $\Psi_{A_1 A_2 \cdots A_d} \in \Gamma(\Lambda^d \cT^* |_{\Sigma})$
  such that $\Psi_{A_1 A_2 \cdots A_d} X^{A_1} = 0$ and $\nabla_i
  \Psi_{A_1 A_2 \cdots A_d} = 0$ along $\Sigma$.
\end{thm}
\begin{rem}\label{rem-simple}
From the proof we will see that such a tractor field $\Psi$ must be a (locally) constant multiple of the tractor normal form along $\Sigma$, and is therefore simple. This observation will become relevant when we connect certain BGG solutions with distinguished submanifolds in Section \ref{subsec-curved-orbits}.
\end{rem}
\begin{proof}
  If $\Sigma$ is distinguished, then by Theorem~\ref{key1}, the tractor normal form is parallel in tangential directions.
  Moreover, it is clear from the definition of the tractor normal form~\eqref{tractor_normal_form} that $N_{A_1 A_2 \cdots A_d} X^{A_1} = 0$.
  Thus we may take $\Psi_{A_1 \cdots A_d}$ to be the tractor normal form.

  Conversely, suppose that we have $\Psi\in \Gamma(\Lambda^d \cT^*
  |_{\Sigma}) $ which satisfies $\Psi_{A_1 A_2 \cdots A_d} X^{A_1} =
  0$ and $\nabla_i \Psi_{A_1 A_2 \cdots A_d}=0$ along $\Sigma$.
  From~\eqref{eqn_tractor_form_comp_series}, we know that, in a background scale, $\Psi$  can be written
  \begin{align*}
    \Psi_{A_1 A_2 \cdots A_d}
    &= \sigma_{a_2 \cdots a_d} \mathbb{Y}^{\phantom{A_1} a_2 \cdots a_d} _{A_1 A_2 \cdots A_d} + \nu_{a_1 a_2 \cdots a_d} \mathbb{Z}^{a_1 a_2 \cdots a_d} _{A_1 A_2 \cdots A_d} \\
    & \qquad+ \phi_{a_3 \cdots a_d} \mathbb{W}^{\phantom{A_1 A_2} a_3 \cdots a_d} _{A_1 A_2 A_3 \cdots A_d} + \rho_{a_2 \cdots a_d} \mathbb{X}^{\phantom{A_1} a_2 \cdots a_d} _{A_1 A_2 \cdots A_d}. 
  \end{align*}
  But the condition $\Psi_{A_1 \cdots A_d} X^{A_1} = 0$ together
  with~\eqref{projector_contractions} implies that $\sigma_{a_2 \cdots
    a_d} = 0$ and $\phi_{a_3 \cdots a_d} = 0$.

  Moreover, if $u^i \in \Gamma(\cE^i)$, the incidence relation $\Psi_{A_1 A_2 \cdots A_d} X^{A_1}=0$ together with the parallel condition means that
  \begin{equation*}
    0 = u^i \nabla_i \left( X^{A_1} \Psi_{A_1 A_2 \cdots A_d} \right) = u^i Z^{A_1} _i \Psi_{A_1 A_2 \cdots A_d}
  \end{equation*}
  so $u^i Z^{A_1} _i \Psi_{A_1 A_2 \cdots A_d} = 0$ for all $u \in \Gamma(T \Sigma)$.
  Expanding this, again using~\eqref{projector_contractions} and the linear independence of the $X$ and $Z$ projectors, one sees that $\nu_{a_1 a_2 \cdots a_d} u^{a_1} = 0$ and $\rho_{a_2 \cdots a_d} u^{a_2} = 0$.
  Since $u^i$ was an arbitrary submanifold tangent vector, we conclude that $\nu \in (\Lambda^d N^* \Sigma) [d]$ and $\rho \in (\Lambda^{d-1} N^* \Sigma) [d-2]$.
  Thus in particular $\nu_{a_1 a_2 \cdots a_d} = f N_{a_1 a_2 \cdots a_d}$, where $N_{a_1 a_2 \cdots a_d}$ is the Riemannian normal form of $\Sigma$ and $f$ is a function on $\Sigma$.

  Now note that, since $\Psi$ is parallel, $\Psi^{A_1 A_2 \cdots A_d} \Psi_{A_1 A_2 \cdots A_d}$ is constant along $\Sigma$.
  On the other hand, 
  \begin{equation*}
    \Psi^{A_1 A_2 \cdots A_d} \Psi_{A_1 A_2 \cdots A_d}
    = \nu^{a_1 a_2 \cdots a_d} \nu_{a_1 a_2 \cdots a_d}
    = f^2 N^{a_1 a_2 \cdots a_d} N_{a_1 a_2 \cdots a_d}
    = f^2 \cdot d!,
  \end{equation*}
  and therefore the function $f$ is locally constant and nowhere-zero.
  Thus on each connected component of $\Sigma$, $\nu_{a_1 a_2 \cdots a_d}$ is a constant multiple of the Riemannian normal form.

  From equation~\eqref{eqn_tractor_form_projector_derivatives}, we calculate
  \begin{equation} \label{eqn_GST_generalization_intermediate_calc}
    \begin{split}
    0 = \nabla_i \Psi_{A_1 A_2 \cdots A_d}
    &= \left( f \nabla_i N_{a_1 a_2 \cdots a_d} + \rho_{a_2 \cdots a_d} \confmet_{i a_1} \right) \mathbb{Z}^{a_1 a_2 \cdots a_d} _{A_1 A_2 \cdots A_d} \\
    &\hspace{2em} + \left( \nabla_i \rho_{a_2 \cdots a_d} -f \cdot d \cdot  N_{a_1 a_2 \cdots a_d} P_i {}^{a_1} \right) \mathbb{X}^{\phantom{A_1} a_2 \cdots a_d} _{A_1 A_2 \cdots A_d}.
  \end{split}
  \end{equation}
  Now, note that the same argument that yielded
  equation~\eqref{eqn_grad_norm} may be repeated replacing normal
  \emph{tractors} with normal \emph{vectors} (as per
  Remark~\ref{remark_results_apply_to_riemannian}) to give
  \begin{equation}
    \nabla_i N_{a_1 a_2 \cdots a_d} = -d \cdot \II_{i [ a_d} {}^{c} N_{a_1 a_2 \cdots a_{d-1} ] c}.
  \end{equation}
  Substituting this into~\eqref{eqn_GST_generalization_intermediate_calc} gives that in particular
  \begin{equation}\label{rho_wedge_g}
    - f \cdot d \cdot \II_{i [ a_d} {}^{c} N_{a_1 a_2 \cdots a_{d-1} ] c} + \confmet_{i [ a_1} \rho_{a_2 \cdots a_{d-1} a_d ]} = 0.
  \end{equation}

  Contracting the above with $\confmet^{i a_1} = \Pi^i _b \bg^{b a_1}$ allows us to express $\rho_{a_2 \cdots a_d}$ explicitly.
  Since the expression $\II_{i a_{d}} {}^{c} N_{a_1 a_2 \cdots a_{d-1} c}$ is already skew in the indices $a_1 a_2 \cdots a_{d-1}$ we have:
	  \begin{align*}
    \confmet^{i a_1} \II_{i [ a_d} {}^{c} N_{a_1 a_2 \cdots a_{d-1} ] c} 
    &= \frac{1}{d}\confmet^{i a_1} \left( \II_{i a_1} {}^{c} N_{c a_2 \cdots a_{d-1} a_d} +  \II_{i a_2} {}^{c} N_{a_1 c \cdots a_{d-1} a_d} 
		+\cdots + \II_{i a_{d}} {}^{c} N_{a_1 a_2 \cdots a_{d-1} c}  \right) \\
    &=  \frac{m}{d} H^{c} N_{c a_2 \cdots a_{d-1} a_d}.
  \end{align*}
  Similarly, since $\confmet^{ia_1}\rho_{a_1\cdots a_{d-1}} =0$,
	  \begin{align*}
    \confmet^{i a_1} \confmet_{i [a_1} \rho_{a_2 \cdots a_d]} 
    &= \frac{1}{d}\confmet^{i a_1} \left( \confmet_{i a_1}  \rho_{a_2 \cdots  a_d} -  \confmet_{i a_2}  \rho_{a_1a_3 \cdots  a_d} 
		-\cdots - \confmet_{i a_d}  \rho_{a_2 \cdots  a_{d-1}a_1}  \right) \\
    &= \frac{m}{d} \rho_{a_2 \cdots a_d}.
  \end{align*}
  Thus 
  \begin{align*}
    \Psi_{A_1 A_2 \cdots A_d}
    &= f N_{a_1 a_2 \cdots a_d} \mathbb{Z}_{A_1 A_2 \cdots A_d} ^{a_1 a_2 \cdots a_d} + f \left( d \cdot H^b N_{b a_2 \cdots a_d} \right) \mathbb{X}_{A_1 A_2 \cdots A_d} ^{\phantom{A_1} a_2 \cdots a_d} \\
    &= f N_{A_1 A_2 \cdots A_d},
  \end{align*}
  where $N_{A_1 A_2 \cdots A_d}$ is the tractor normal form. 

  Since the function $f$ is locally constant and nowhere-zero, $\nabla_i \Psi_{A_1 A_2 \cdots A_d} = 0$ implies that the tractor normal form is parallel.
  Thus $\Sigma$ satisfies ones of the equivalent conditions of Theorem~\ref{key1}, and is therefore a distinguished submanifold.
\end{proof}

\noindent Note that Theorem \ref{thm_submanifold_gst} follows as the
tractor Hodge-$\star$ operation \eqref{tstar} commutes with the
tractor covariant derivative.

\subsection{Curved orbits and generalisations}\label{subsec-curved-orbits}

The following result shows one way in which distinguished submanifolds
arise as {\em curved orbits}, in the sense of \cite{CGH-Duke}. It generalises \cite[Proposition 7.1]{GST}. Before stating the theorem we introduce some terminology: we shall say that a tractor (or vector) $\mathbb{K}_{A_1 \cdots A_d}$ is {\em timelike} if $\mathbb{K}_{A_1 \cdots A_d}\mathbb{K}^{A_1 \cdots A_d}$ is negative, \emph{spacelike} if this is positive and {\em null} if it is zero. 
\begin{thm}\label{thm_submanifold_zero_locus}
  Suppose $k_{a_1 \cdots a_{d-1}}$ is a normal solution of the
  conformal Killing form equation on $(M, \bm{c}$) such that the
  parallel tractor $\mathbb{K}_{A_1 \cdots A_d}=L(k_{a_1 \cdots a_{d-1}})$ is simple.  Then the
  zero locus of $k$
  is either empty, an isolated point, or a distinguished conformal submanifold of codimension $d$.
  Moreover, writing $\cZ(k)$ for this zero locus:
  \begin{itemize}
    \item if $\mathbb{K}$ is timelike, then $\cZ(k)$ is necessarily empty;
    \item if $\mathbb{K}$ is null, $\cZ(k)$ consists only of isolated points; 
    \item if $\mathbb{K}$ is spacelike, then $\cZ(k)$ is either empty or is a distinguished submanifold of codimension $d$.
  \end{itemize}
\end{thm}
\begin{rem}\label{rem-thm_submanifold_zero_locus} \emph{(i)} In the proof we will observe that in all cases $\mathcal{Z}(k)=\mathcal{Z}(X\intprod \mathbb{K})$. In particular, fixing any metric $g\in \bc$, when $d>1$ we find that $\mathcal{Z}(k)=\mathcal{Z}(\cK)$  where $\cK := (k_{a_1 \cdots a_{d-1}}, \nabla^c k_{c a_2 \cdots a_{d-1}})$.
\emph{(ii)}
The result still holds in the $d=1$ case if we take $\sigma=k$ to be an almost-Einstein scale and $\mathbb{K}=L(k)$ to be the corresponding scale tractor (cf.\ \cite{GRiemSig,Curry-G-conformal} and Example \ref{ex-conf-infty-dist}). 
\end{rem}

\begin{proof}
  Suppose $k_{a_1 \cdots a_{d-1}} \in \cE_{[a_1 \cdots a_{d-1}]} [d]$
  is a normal solution to~\eqref{eqn_conf_killing_form} such that 
  $\mathbb{K} = L(k)$ is simple. Fix any metric $g\in \bc$. Note that from equation~\eqref{eqn_conf_killing_splitting_op}, at a point where $k=0$ we have 
\begin{equation}
\mathbb{K}_{A_1 \cdots A_d}\mathbb{K}^{A_1 \cdots A_d} =  \frac{1}{d^2}|\!|\phi|\!|^2 + |\!|\nu|\!|^2,
\end{equation}
where $\phi_{a_1a_2\cdots a_d} = \nabla_{[a_1}k_{a_2\cdots a_d]}$ and $\nu_{a_3\cdots a_d} = \frac{d-1}{n-d+2} \nabla^c k_{c a_3 \cdots a_d}$. In particular, if $\mathbb{K}$ is timelike then $k$ clearly cannot have any zero and if $\mathbb{K}$ is null then at any zero of $k$ we must have $\nabla k =0$ (moreover, these results hold without needing the simplicity of $\mathbb{K}$). Now, our goal is to apply the curved orbit theory of \cite{CGH-Duke}, and to this end we note that from \eqref{eqn_conf_killing_splitting_op} one sees that $k=0$ at the point $p\in M$ if, and only if, $X \intprod \mathbb{K} = 0 \!\mod X$ (note that a term of the form $\nu_{a_3\cdots a_d}X^{A_0}\mathbb{W}_{A_0 A_1 A_2 \cdots A_d} ^{\phantom{A_0 A_1} a_3 \cdots a_d} = -\frac{1}{d+1}\nu_{a_3\cdots a_d}\mathbb{X}_{A_1 A_2 \cdots A_d} ^{\phantom{ A_1} a_3 \cdots a_d}$ is zero mod $X$). In particular, the condition $k=0$ determines a {\em $P$-type}, in the language
  of~\cite{CGH-Duke} and so Theorem 2.6 of that article applies. It will be more convenient, however, to consider the simpler condition $X \intprod \mathbb{K}$, which also determines a $P$-type. From \eqref{eqn_conf_killing_splitting_op} one sees that $X \intprod \mathbb{K}=0$ at some point $p \in M$ if, and only if, $(k_{a_1 \cdots a_d}, \nabla^c k_{c a_2
    \cdots a_d}) = 0$ at the point $p$, i.e.\ $k=0$ and $\nu=0$ at the point $p$. Let $K:=X \intprod \mathbb{K}$. From the preceding discussion, if $\mathbb{K}$ is timelike or null then clearly $\mathcal{Z}(K)=\mathcal{Z}(k)$ (without the need for $\mathbb{K}$ to be simple).  On the other hand, in the case when  when $\mathbb{K}$ is spacelike, it is easily checked that $k=0$ implies $K=0$ due to the simplicity of $\mathbb{K}$ (this fails when $\mathbb{K}$ is not simple); to see this one merely writes $\mathbb{K}$ at a point as a wedge product of orthogonal spacelike tractor $1$-forms $I^1\wedge\cdots \wedge I^d$ and examines the resulting expression $K$ to see that $K=0 \!\mod X$ only if $K=0$ (since $X$ is null); in this case the zero locus of $k$ turns out to be $\mathcal{Z}(\sigma_1,\ldots,\sigma_d)$ where $\sigma_i = X\intprod I^i$. Thus, in all cases $\mathcal{Z}(k)=\mathcal{Z}(K)$ and we can consider the zero locus of $K:=X\intprod \mathbb{K}$ instead of $k$. Apart from the distinguished submanifold property in the third bullet point (which will follow afterward from Theorem \ref{thm_submanifold_gst}/Theorem \ref{thm_submanifold_gst-main}) by Theorem 2.6 of~\cite{CGH-Duke} this reduces to an elementary consideration of the model case.

Recall that in the model case parallel tractors correspond to constant tensor fields on $\mathbb{R}^{n+2}$ and   the canonical tractor $X^A$ is identified with the position vector field of $\mathbb{R}^{n+2}$ along $\mathcal{C}_+$.  Hence, in the model case, if $\mathbb{K}_{A_1 \cdots A_d}$ is a
  parallel simple $d$-cotractor and $X \intprod \mathbb{K}$ is zero
  at some point $p$, then $X \intprod \mathbb{K}$ is zero along a
  submanifold $p$, given as follows.  The form $\mathbb{K}$ determines
  in $\R^{n+2}$ a unique $(m+2)$-plane through $p$ (as usual $m=n-d$
  and $(m+2)$-plane means a linear subspace of that dimension)
  consisting of the vectors $X^A$ in $\R^{n+2}$ that are in the
  nullity of $\mathbb{K}_{A_1 \cdots A_d}$.
   The submanifold is
  then the ray projectivisation of the intersection of this hyperplane
  with the null quadric for the Minkowski signature inner product on $\R^{n+2}$.

  We now treat the three cases in the statement of the theorem by
  considering the distinct ways that this hyperplane can intersect the
  null cone.  First, if $\mathbb{K}_{A_1 \cdots A_d}$ is simple and timelike,
  then non-zero vectors in the nullity of
  $\mathbb{K}_{A_1 \cdots A_d}$ are spacelike. No non-zero vector in their span is null or timelike.   Therefore in
  this case the $(m+2)$-plane has no intersection with $\mathcal{C}_+$.  Thus  the zero locus $\cZ(K)$ is empty (of course, we had already seen this).  For the second case, note that if $\mathbb{K}_{A_1 \cdots
    A_d}$ is null, then, using the Minkowski signature, it follows
  that the simple $d$-tractor $\mathbb{K}$ can be obtained as the
  exterior product of covectors that are spacelike except for
  exactly one which is null.  Dually, this implies that there is a
  collection of vectors which span the nullity $(m+2)$-hyperplane that
  consists of a single null vector and $m+1$ spacelike vectors.  Thus the hyperplane is tangent to the null cone, and after
  ray projectivisation the intersection descends to an isolated point.
		Finally, suppose that  $\mathbb{K}_{A_1 \cdots
    A_d}$ is spacelike. Then the hyperplane defined by vectors
  in the nullity of $\mathbb{K}_{A_1 \cdots A_d}$ can be spanned by one timelike vector and $m+1$ spacelike vectors.  Such an
  $(m+2)$-plane meets the null cone $\mathcal{C}_+$ transversely, and hence,
  under ray projectivisation, will descend to a submanifold $\cZ(K)$
  of $S^n$ of codimension $d$.  
  Now by Theorem 2.6 of~\cite{CGH-Duke} it then follows that on $M$
  the zero locus of the simple $d$-tractor $\mathbb{K}$ will take the
  same form as on the model.  Thus the three bullet points follow from the analysis just done, of the corresponding cases on the model, save for the very final statement that if $\cZ(K)$ is nonempty and does
  not just consist of isolated points, then the codimension $d$ submanifold $\cZ(K)$ is distinguished. But this follows from 
    Theorem~\ref{thm_submanifold_gst}.
\end{proof}
\noindent We note that such simple parallel tractors $\mathbb{K}$ have arisen in
the study of holonomy and generalisations of almost Einstein structures
\cite{ArmstrongLeitner,Leitner2010,Leitner2012}. In particular, in the presence of multiple (independent) almost-Einstein scales $\sigma_1,\ldots, \sigma_d$ with scale tractors $I^1\cdots I^d$, such a simple parallel tractor $\mathbb{K}_{A_1\cdots A_d}$ is given by $d! I^1_{[A_1}\cdots I^d_{A_d]}$ and arises from the conformal Killing form $k_{a_2\cdots a_d}=\sum_{s\in S_n} \mathrm{sign}(s) \sigma_{s(1)}\nabla_{a_2}\sigma_{s(2)} \cdots \nabla_{a_d}\sigma_{s(d)}$. In this case, when $\mathbb{K}$ is spacelike one finds that $\mathcal{Z}(K)=\mathcal{Z}(\sigma_1,\ldots, \sigma_d)$, which was our motivation for considering this zero locus.

In the case of the model, meaning $S^n$ with its usual conformal
structure, all solutions to first BGG equations are normal (and all cases arise in all codimensions $d$, which is the idea behind the proof). Moreover,
it is easily seen that, in this setting, the space of solutions to the
conformal Killing-Yano equation \eqref{eqn_conf_killing_form}, of a
given rank, is spanned by solutions $k$ with $L(k)$ satisfying the
conditions of the Theorem above. It is interesting and valuable to
determine the extent to which similar results hold in more general
settings (see, e.g., \cite{BMO2011,Derdzinski11,Derdzinski12,Frances} for some related results in the case of conformal Killing fields).  The first and third bullet point of Theorem \ref{thm_submanifold_zero_locus} generalise quite easily, as we see in the following proposition. 

\begin{prop} \label{progress}
  Let $k_{a_1 \cdots a_{d-1}}\in \Gamma (\ce_{[a_1\cdots a_{d-1}]}[d])$ and  $\mathbb{K}_{A_1 \cdots A_d}=L(k_{a_1 \cdots a_{d-1}})$.
Then
\begin{enumerate}

  \item If $\mathbb{K}$ is timelike, then
    $\cZ(k)$ is necessarily empty.

\item If $\mathbb{K}$ is spacelike and simple, and $k$ satisfies \eqref{eqn_conf_killing_form}
    along $\cZ(k)$, then $\cZ(k)$ is either empty or is a submanifold of codimension $d$. 

\item If $\mathbb{K}$ is spacelike, simple, and satisfies $\nabla \mathbb{K}=0$ at all points of $\cZ(k)$, then $\cZ(k)$ is either empty or is a distinguished submanifold of codimension $d$.

\end{enumerate}
\end{prop}
\begin{rem}\label{rem-progress}
 As in Theorem \ref{thm_submanifold_zero_locus} above, fixing any $g\in\bc$, we find in all of the above cases that $\mathcal{Z}(k)=\mathcal{Z}(\cK)$  where $\cK := (k_{a_1 \cdots a_{d-1}}, \nabla^c k_{c a_2 \cdots a_{d-1}})$.
\end{rem}
\begin{proof}
From (\ref{eqn_tractor_form_comp_series}), \eqref{form-ps},  and (\ref{tr-met-form})  it follows that any tractor $d$-form $\mathbb{K}$
satisfying $X \intprod \mathbb{K} = 0$ at  $p\in M$ has
$\mathbb{K}_{A_1\cdots A_d}\mathbb{K}^{A_1\cdots A_d}\geq 0$ at
$p$. This proves 1.

We now consider 2. For convenience, we fix $g\in \cc$ and use $g$ to trivialise the density bundles. Let $I^1,\ldots, I^d$ be a collection of orthogonal spacelike tractor $1$-forms such that $\mathbb{K}=I^1\wedge\cdots \wedge I^d$ and let $\sigma_i=X\intprod I^i$ (which we think of as functions).
It follows that $k$ takes the form
\begin{equation}\label{kform}
\sigma_1 \omega_1 +\cdots + \sigma_d \omega_d
\end{equation}
where the $\omega_i$, $i=1,\cdots ,d$, are each simple $(d-1)$-forms. Moreover, as argued in the proof of Theorem \ref{thm_submanifold_zero_locus} above, from the simplicity of $\mathbb{K}$ we have that $k=0$ implies $X\intprod \mathbb{K}=0$ (again this is comes from examinining the for $K$ in terms of the spacelike tractors $I^1,\ldots, I^d$ and using that $X$ is null to see that $K=0\!\mod X$ only if $K=0$). It follows that $\mathcal{Z}(k)=\mathcal{Z}(K)=\mathcal{Z}(\sigma_1,\ldots, \sigma_d)$, where the last inequality follows from the expression for $K$ in terms of $I^1,\ldots, I^d$ and their linear independence. Since $\nu=0$ when $k=0$ and $\mathbb{K}$ is spacelike, it follows that $\phi \neq 0$ at any point where $k=0$. From this and the formulae for $\phi$ and $\omega_1,\ldots, \omega_d$ in terms of the components of $I^1,\ldots, I^d$ it follows that the $(d-1)$-forms $\omega_1,\ldots, \omega_d$ are linearly independent at any point where $k=0$ (and hence in a neighborhood of $\mathcal{Z}(k)$). Moreover, from the equation (\ref{eqn_conf_killing_form}) we have that, along $\mathcal{Z}(k)$,
\begin{equation}\label{KY}
\nabla k = \phi.
\end{equation}
The above display puts conditions on the $\sigma_i$ and their relation to the $\omega_j$. In particular one easily concludes that, at each point in
$Z(k)$, the equation \eqref{KY} implies that $\{d\sigma_1,\cdots ,d\sigma_{d}\}$
is a linearly independent set. Thus from the constant rank theorem it follows that $\cZ(K)$ is either empty or is a
submanifold of codimension $d$. This establishes 2. 

Item 3 then follows from 2 together with Theorem~\ref{thm_submanifold_gst}, since if  $\nabla \mathbb{K}=0$ at all points of $\cZ(K)$ then $k_{a_1 \cdots a_{d-1}}$ satisfies \eqref{eqn_conf_killing_form}  along $\cZ(K)$.
   \end{proof}
	
Proposition \ref{progress} shows that in fact no hypothesis on $k$ is needed for the first bullet point of Theorem \ref{thm_submanifold_zero_locus} to hold, and that the third bullet point holds under much weaker hypotheses. In particular, in the $d=2$ case (where $k$ is a conformal Killing vector field) Proposition \ref{progress} shows that normality is not required for the third bullet point to hold, since at points where $k=0$ the prolonged conformal Killing equation $\nabla\mathbb{K}=k\intprod \Omega$ reduces to $\nabla\mathbb{K}=0$. For a simple example of non-normal conformal Killing vector field that gives rise to a codimension $2$ distinguished submanifold, consider $\mathbb{CP}^n$ equipped with the Fubini-Study metric: the 1-parameter family of maps $[z_0,z_1 \cdots, z_n] \mapsto [z_0,e^{it}z_1 \cdots, e^{it}z_n]$ corresponds to a Killing vector field that vanishes on the totally geodesic $\mathbb{CP}^{n-1}$ given by $\{z_0=0\}$. One can also show that when $d=2$ the second bullet point of Theorem \ref{thm_submanifold_zero_locus} continues to hold without the assumption of normality. To see this, note that, fixing any $g\in \bc$, the hypothesis that $\mathbb{K}$ is null implies that the components $\phi$  and $\nu$ of $\mathbb{K}=L(k)$ are zero on $\mathcal{Z}(k)$. Hence, at such points $k$ and $\nabla k$ vanish. But since $\mathbb{K}\neq 0$ we must have $\nabla_a\nabla^bk_b \neq 0$ at such points and it follows that any point of $\mathcal{Z}(k)$ is an essential point of $k$; see, e.g., \cite[Theorem 3.4]{Derdzinski12}. But such essential points are isolated \cite{BMO2011} and hence the result follows. Thus the assumption of normality is not required in Theorem \ref{thm_submanifold_zero_locus} when $d=1$ or $d=2$ (cf.\ Remark \ref{rem-thm_submanifold_zero_locus}(ii) for the $d=1$ case). We leave it as an open question whether one can also drop the normality condition when $d>2$.

Although one of the themes of this article has been that of distinguished submanifolds, we are interested in developing conformal submanifold theory in general. From this point of view a key message of this section is that it is natural for codimension $d$ submanifolds to arise as the zero locus of a weighted $(d-1)$-form $k$, and that we can make geometric conclusions about this submanifold by considering the behavior of the corresponding form tractor $\mathbb{K}$ along the submanifold. The results above are highly suggestive of a holographic approach to higher codimension submanifold theory generalising the approach of \cite{AGW,BGW,GWWillmore,GWRenormVol,GWbdy,GWcalc,GWConfHypYamabe} from the hypersurface case, where the role of the scale $\sigma$ and scale tractor $I$ are played by a weighted form $k$ and the corresponding tractor form $\mathbb{K}$, but we leave this to be taken up elsewhere.


\Addresses

\end{document}